    \documentclass[11pt]{amsart}
    
    \usepackage{geometry}
    \usepackage{amsmath}
    \usepackage{amssymb,amsfonts,latexsym}
    \usepackage{appendix}
    \usepackage{verbatim}
    \usepackage[all]{xy}

    \newtheorem{thm}{Theorem}[section]
    \newtheorem{cor}[thm]{Corollary}
    \newtheorem{lem}[thm]{Lemma}
    \newtheorem{prop}[thm]{Proposition}
    \theoremstyle{definition}
    \newtheorem{conj}[thm]{Conjecture}
    \newtheorem{defn}[thm]{Definition}
    \theoremstyle{remark}
    \newtheorem{rem}[thm]{Remark}
    
    \newtheorem{exam}[thm]{Example}
    \numberwithin{equation}{section}

    \newcommand{\eps}{\varepsilon}

    \newcommand{\mP}{\mathbb{P}}

    \newcommand{\K}{\mathbb{K}}

    \DeclareMathOperator\sgn{sgn}
    
    \newcommand\s{s}

    \newcommand{\abs}[1]{\lvert #1 \rvert}

    \newcommand{\ra}{\rightarrow}

    \newcommand\suchthat{{\,:\ }}
    
    \DeclareMathOperator\Gal{Gal}
    \newcommand\divides{\mid}

    \newcommand\oline[1] {{\overline{#1}}}
    
    \newcommand\Hom{{\operatorname{Hom}}}

    \newcommand\Aut{{\operatorname{Aut}}}
    
    \newcommand\lcm{{\operatorname{lcm}}}
    \newcommand\orb{{\operatorname{orb}}}
    \newcommand\Orb{{\operatorname{Orb}}}

    \DeclareMathOperator\Mon{Mon}
    
    \DeclareMathOperator\AGL{AGL}

    \usepackage[colorlinks,pagebackref,pdftex, bookmarks=false]{hyperref}

\numberwithin{table}{section}

\begin{document}
    
    \title
    {Monodromy groups of product type}%

    \def\ann{Department of Mathematics, University of Michigan, Ann Arbor, MI 48109--1043, USA}
    \def\technion{Department of Mathematics, Technion - Israel Institute of Technology, Haifa 32000, Israel}
    \author{ Danny Neftin}
    \address{\technion}
    \email{dneftin@tx.technion.ac.il}%
    \author{ Michael E. Zieve}
    \address{\ann}
    \email{zieve@umich.edu}
    
    

    \begin{abstract}
    The combination of this paper and its companion \cite{NZ}  complete the classification of monodromy groups of  
    indecomposable coverings of complex curves $f:X\ra \mP^1$ of sufficiently large degree in comparison to the genus of $X$.
In this paper we determine all such coverings with monodromy group   $G\leq S_\ell\wr S_t$ of product type for $t\ge 2$.
    \end{abstract}
    
    \maketitle

    \section{Introduction}\label{sec:intro}
    \subsubsection*{Background} 
    The monodromy group $\Mon(f)$ and ramification of a branched covering $f:X\ra\mP^1$ of the complex projective line are fundamental invariants which lie at the heart of many problems in arithmetic geometry, complex analysis, and other areas, cf.~\cite{NZ}. 
    Here, $X$ is a complex curve and $\Mon(f)$ is the Galois group of the Galois closure of $\mathbb C(X)$ over the function field $\mathbb C(t)=\mathbb C(\mathbb P^1)$, viewed as a permutation group on $\Hom_{\mathbb C(t)}(\mathbb C(X),\oline{\mathbb C(t)}
)$.  The ramification type of $f$ is the multiset $\{E_f(P)\,|\,P\in B\}$, where $B$ is the branch locus of $f$  and $E_f(P)$ is the multiset of ramification indices of $f$ over the point $P$. 

    The problem of determining all but finitely many of the possible monodromy groups $\Mon(f)$  for {\it indecomposable} coverings $f:X\ra \mathbb P^1$ of fixed {\it genus} $g_X=g$, and their corresponding ramification types,  originates in the work of Chisini, Ritt and Zariski. Chisini \cite{Chisini} and Ritt \cite{Ritt2} determine the coverings $f:X\ra \mP^1$ of prime degree with $g_X=0$ and solvable monodromy, a result which was soon after generalized by Zariski \cite{Zar1,Zar2}. 
Many authors contributed to this problem in subsequent years. In particular, for any  fixed $g$, the combination of the papers \cite{LSa,LSh,FM} shows that if $f:X\ra\mP^1$ is an indecomposable covering of sufficiently large degree, where $X$ has genus $g$ and the monodromy group of $f$ is simple, then this monodromy group must be an alternating group. 
    Based on this and other results, including  \cite{Asch, GN, GT},  Guralnick gave a conjectural list of the possible monodromy groups of genus-$g$ indecomposable coverings $f:X\ra\mP^1$ whose degree $n$ is sufficiently large compared to $g$ \cite[Conjecture 1.0.1]{GS}.
\begin{conj}\label{gurcon}
For each integer $g\ge 0$, there exists a constant $N_g$ with the following property.  For every indecomposable covering $f:X\ra \mP^1$ of degree $n\geq N_g$ and genus $g_X=g$,  the monodromy group $\Mon(f)$ satisfies one of the following:
    \begin{enumerate}
    \item $\Mon(f) = A_n$ or $S_n$;
    \item $\Mon(f) = A_d$ or $S_d$ with $n=d(d-1)/2$ and $g=0$;
    \item $A_d^2\leq \Mon(f) \leq (S_d^2)\rtimes C_2$, where $n=d^2$, the semidirect action of $C_2$ permutes the two copies of $S_d$, and $g\le 1$;
    \item $\Mon(f)\leq (C_p)^i\rtimes C_k$, with $i\in\{1,2\}$, $k\in \{1,2,3,4,6\}$, and $p$ prime, where $n=p^i$ and $g\le 1$.
\end{enumerate}
\end{conj}

In the companion paper \cite{NZ} we will combine the results of the present paper with additional arguments in order to prove Conjecture~\ref{gurcon}, and moreover we will determine all possible ramification types for $f$ in each of the cases (2)--(4), see \cite[Theorem 1.1]{NZ}. Case (1) is the generic case which occurs for nearly all $f:X\ra \mP^1$. Note that in case (4) the monodromy groups and ramification types were already known to Zariski, cf.~\cite[Proposition 3.8]{GT} and \cite[Proposition 9.5]{NZ}. 

    \subsubsection*{Main result} 
    The proof of Conjecture~\ref{gurcon} splits into cases based on the Aschbacher--O'Nan--Scott structure theorem for primitive groups $G$, see \cite{GT}. 
    The case (A) where $G$ has an abelian minimal normal subgroup is treated by Guralnick--Thompson \cite{GT} and Neubauer \cite{Neu1}. The case (B) where $G$ has more than one minimal normal subgroup  is treated in \cite{Shih,GMN}. Henceforth assume $G$ has a unique minimal normal subgroup $Q$, and that $Q$ is nonabelian, so that $Q=L^t$ for some nonabelian simple group $L$. The case (C1) where $Q$ acts regularly is treated by Guralnick--Thompson \cite{GT}. The case (C2) where $L$ acts regularly is treated by Aschbacher \cite{Asch}. The remaining case (C3), also known as groups of {\it product type}, is when $G$ acts on a set $\Delta^t$, inducing a pointwise action of  $L^t$ on $\Delta^t$,  with respect to a nonregular action of $L$ on $\Delta$. 

    The case $t=1$ is completed in \cite[Theorem 1.2]{NZ}. In this paper, we fix any $t\geq 2$ and determine all product type monodromy groups of sufficiently large degree $\ell^t$. Let $S_\ell \wr S_t = S_\ell^t\rtimes S_t$ denote the standard wreath product where the semidirect product action permutes the $t$ copies of $S_\ell$. 
    \begin{thm}\label{thm:main-wreath}
    Fix an integer $t\geq 2$. There exist positive constants $c=c_t, d=d_t$ depending only on $t$, 
    such that for every indecomposable covering $f:X\ra X_0$ with monodromy group $G$ of product type  acting on $\Delta^t$, 
    either $g_X > c\ell^{t-1}-d\ell^{t-2}$ or the ramification type of $f$ appears\footnote{{The ramification types are listed more concisely using the branch cycles of $f$, see Section \ref{sec:prelim-t=2}.}} in Table~\emph{\ref{table:wreath}}.
    Moreover, for all of the ramification types in Table~\emph{\ref{table:wreath}}, one has $t=2$, and  $A_\ell^2\lneq G\leq S_\ell\wr S_2$, and $g_{X_0}=0$, and  $g_X\leq 1$. 
    \end{thm}
    Conversely, for every type in Table~\ref{table:wreath}, 
    there exists an indecomposable covering $f:X\ra\mP^1$ of this ramification type with monodromy group $A_\ell^2\lneq G\leq S_\ell \wr S_2$, 
    see Section \ref{sec:h1-branches}.

The proof of Theorem \ref{thm:main-wreath} is self contained. As we explain in the proof sketch below, the theorem becomes easier as $t$ grows. The most challenging cases  are $t = 2,3,4$, and $6$, with the hardest being $t=2$. 
The theorem was known previously when $t>8$, in which case it is known that
   $g_X>\ell^{t-1}/2500$ \cite[Corollary 8.7]{GN}.

The combination of Theorem \ref{thm:main-wreath} and 
{\cite[Corollary 8.7]{GN}} shows that for an indecomposable covering $f:X\ra \mP^1$ of sufficiently large degree $n$, with monodromy group of product type with $t\geq 2$, if the monodromy group and ramification type are not in Table~\ref{table:wreath} then $g_X>c'\sqrt{n}$ for some absolute constant $c'>0$, or equivalently $n<(g_X/c')^2$.
Together with the proof of \cite[Theorem 1.1]{NZ}, this shows that the constant $N_g$ in Conjecture~\ref{gurcon} may be chosen to be quadratic in $g$ if we restrict to coverings whose monodromy group is not alternating nor symmetric.
It seems plausible that the same assertion is true without excluding alternating and symmetric groups.
But this quadratic growth rate cannot be improved, in light of \cite[Example 5.5]{NZ}.

    \subsubsection*{About the proof of Theorem \ref{thm:main-wreath}} \label{sec:method} 
    Let $f:X\ra\mP^1$ be a covering with monodromy group $G$ of product type, and assume for simplicity  $A_\ell^t\leq G \leq S_\ell \wr S_t$ with the standard action on $\{1,\ldots,\ell\}^t$ where $t\geq 2$. In particular, $G$ also acts on $\{1,\ldots,t\}$ with kernel $K:=G\cap S_\ell^t$. 
    Letting $H$ be a point stabilizer in the degree-$\ell^t$ action,
one has the following diagram of natural projections: 
    $$
    \xymatrix{ Z=\tilde X/(H\cap K) \ar[r]^{\pi} \ar[d]_{h} & \tilde X/H \cong X\ar[d] \\
     Y= \tilde X/K \ar[r]_{\pi_0} & \tilde X/G\cong \mP^1,
    }
    $$
    where $Y:=\tilde X/K$ and $Z:=\tilde X/(H\cap K)$. 
    Our approach to proving that $g_X$ is large is to show that $g_{Z}$ is large in comparison to the difference $D:=g_{Z} - \deg\pi\cdot g_X$, where the latter quantity can be expressed in terms of the ramification contribution of $\pi$ to the Riemann--Hurwitz formula. 
    This approach was used by Guralnick--Thompson to show that $g_{Z} - D\geq a_1\ell^t$ for some absolute constant $a_1>0$ in all cases where $g_{Y} > 1$, 
    and by Guralnick--Neubauer \cite[Corollary 8.7]{GN}
    to show that 
    $g_{Z}-D > a_2\ell^{t-1}$  for some absolute constant $a_2>0$,  in all cases where $t>8$ and $g_{Y}\leq 1$. 
    
    Estimating the difference $g_Z-D$ using the methods of \cite{GT, GN} becomes notoriously difficult when treating cases with $t\leq 8$. 
    Namely, the most challenging  cases, where $D$ is large, appear  when $g_Y\leq 1$ and one of the branch cycles of $f$ acts on $\{1,\ldots,t\}$ as a transposition, i.e.~only when $t=2,3,4$ and $6$, see Section \ref{sec:transp}. In fact, in some of these cases there are (infinite families of) ramification data $\mathcal R$, that is, multisets of partitions of $\ell^t$,  such that if there were a covering $f:X\ra \mP^1$  with monodromy group $A_\ell^t\leq G\leq S_\ell \wr S_t$ and ramification type $\mathcal R$ then the Riemann--Hurwitz formula would imply that $g_{Z}-D<a_3$ for an absolute constant $a_3>0$. 
    However, it turns out that none of these ramification data correspond to a covering, 
    as the corresponding product-$1$ tuple, required in Section \ref{sec:RET}, does not exist. 
    
    In contrast to \cite{GT, GN, NZ}, for small values of $t$ our approach starts with exploiting the product~$1$ relation associated to a covering, cf.~Section \ref{sec:RET}. 
    When $g_Y\leq 1$, this product-$1$ relation in $S_\ell \wr S_t$ amounts to a single product-$1$ relation in $S_\ell$, which by Riemann's existence theorem gives rise to a new degree $\ell$ covering $\hat f:\hat Y\ra\mP^1$ with transitive monodromy group. The existence of the new covering $\hat f$ gives severe restrictions on the ramification of $f$ via the Riemann--Hurwitz formula for $\hat f$ (or alternatively Ree's inequality). This is carried out in Section \ref{sec:reduced}. 
    
    A key ingredient in the proof is showing that a bound of the form $g_X<\alpha\ell^{t-1}$, for a constant $\alpha>0$, forces  the ramification of $h$ to be similar to the ramification of a Galois covering. More precisely, since $A_\ell^t\leq K\leq S_\ell^t$, the covering $h$ is the fiber product of degree-$\ell$ coverings $h_i$, $i=1,\ldots,t$, each with monodromy $A_\ell$ or $S_\ell$. For sufficiently large $\ell$, the above bound implies that either the number of preimages in $h_i^{-1}(P)$ is bounded in terms of $\alpha$, 
    or all preimages of $P$ have the same ramification index with a bounded amount of exceptions, for each point $P$ of $Y$ and $i=1,\ldots,t$, cf.~Section~\ref{sec:almost-Gal}. This part relies on an idea from \cite{DZ}.

    Using the Riemann--Hurwitz formula for $\hat f$ and the fact that each $h_i$ has ramification resembling that of a Galois covering, we then show that the ramification of $f$ itself is similar to that of a Galois covering, cf.~Section \ref{sec:reduced}. 
    We then analyze the possibility for such ramification types of $f$, in order to obtain a bound $D\leq b_4\ell^{t-2}$  and hence the desired inequality $g_{Z}-D>a_4\ell^{t-1}-b_4\ell^{t-2},$ for some positive constants $a_4=a_{4,t}$, $b_4=b_{4,t}$, depending only on $t$.

     In Section \ref{sec:G/K}, we reduce the proof of Theorem \ref{thm:main-wreath} to proving Theorem \ref{thm:main-Gal} which addresses cases where the ramification of $h$ is similar to a Galois covering in a sense described precisely in Section \ref{sec:almost-Gal}. The case $t=2$ of Theorem \ref{thm:main-Gal} is then treated in Section \ref{sec:t=2}. For $t\geq 3$, Section~\ref{sec:G/K} further reduces Theorem \ref{thm:main-Gal} to proving Proposition \ref{prop:main-transposition} which ensures that the above Riemann--Hurwitz contribution $D$ is bounded by a multiple of $\ell^{t-2}$ depending only on $t$. Due to the bounds on $D$ obtained in Section \ref{sec:genus}, this yields the proof when no power of a branch cycle acts on $I$ as a transposition, thus reducing the proof to the cases listed in Section \ref{sec:transp}. The proof of Proposition \ref{prop:main-transposition} in these last cases is completed in Section \ref{sec:reduced}, completing the proofs of Theorems \ref{thm:main-Gal} and Theorem \ref{thm:main-wreath}. We list the exceptional ramification types for Theorem \ref{thm:main-wreath} in Section \ref{sec:prelim-t=2} and show they correspond to exceptional coverings in Section  \ref{sec:h1-branches}. 
     \subsubsection*{Acknowledgments}
     We thank Robert Guralnick, Arielle Leitner, and Tali Monderer for helpful comments.
The first author is grateful for the support of ISF grants No.\ 577/15, 353/21, and the NSF for support under a Mathematical Sciences Postdoctoral Fellowship. The second author thanks the NSF for the support of grant DMS-1162181. We also thank the United States-Israel Binational Science Foundation (BSF) for its  support under Grant No.\ 2014173, and the SQuaREs program of the American Institute of Math.

    \section{Preliminaries and notation}\label{sec:prelim}
    \subsection{Monodromy} Let $\K$ be an algebraically closed field of characteristic $0$. 
    A {\it covering} $f:X\ra X_0$ is  a morphism of (smooth projective geometrically irreducible) curves defined over $\K$. Letting $\tilde f: \tilde X\ra X_0$ be its Galois closure, we call $\Mon(f):=\Gal(\tilde f)$ the {\it monodromy group} of $f$. We note that $G:=\Mon(f)$ is a permutation group on $n:=\deg f$ letters via its right action on $\Omega:=H\backslash G$, where $H$ is the monodromy group of the natural projection $\tilde X\ra X$. 
    The group $H$ is then a point stabilizer in this action. Given a subgroup $H_2$ of $\Mon(f)$, there is a natural projection $f_2:\tilde X/H_2\ra X_0$. This natural projection is {\it equivalent} to $f$ if there exists an isomorphism $\eta:X\ra\tilde X/H_2$ such that $f=f_2\circ \eta$. The coverings $f_2$ and $f$ are then equivalent if and only if $H_2$ is conjugate to $H$. The covering $f$ is  indecomposable if and only if $H$ is a maximal subgroup of $G$ or equivalently if the $G$-action on $\Omega$ is primitive. 
    
    Note that unless otherwise mentioned all group actions on sets are right actions. 
    Thus, permutation multiplication is left to right, e.g.~$(1,2)(1,3) = (1,2,3)$. Denote $x^y = y^{-1}xy$ for any $x,y\in G$. Let $C_n$ and $D_n$ denote the cyclic and Dihedral groups of order $n$, resp.
    \subsection{Ramification} 
    Let $P$ be a point of $X_0$.  For a point $Q\in f^{-1}(P)$,
    let $I=I(\tilde Q/P)$ denote the inertia group of $Q$ over $P$. 
    Since $\K$ is algebraically closed, $I$ coincides with the decomposition group 
    $\{ \sigma\in G\,|\, \tilde Q^{\oline\sigma} = \tilde Q\}. $
    We call $x\in G$ a {\it branch cycle} of $f$ over a point $P$ of $X_0$ if $x$ generates $I(Q/P)$ for some $Q\in f^{-1}(P)$. We write $\Orb_\Omega(I)$ (resp.~$\Orb_\Omega(x)$) for the set of orbits of $I$ (resp.~$x$) on $\Omega=H\backslash G$. 
    
    Write $e_f(Q):=\abs I$ for the {\it ramification index} of $Q$ under $f$. 
    For a point $P$ of $X_0$ write $E_f(P)$ for the tuple $[e_{f}(Q_1),\ldots, e_{f}(Q_k)]$ where $Q_1, \ldots, Q_k$ are the preimages of $P$ under $f$,
    ordered so that $e_f(Q_1)\geq \ldots \geq e_f(Q_k)$. We write $E_f(P)=[u_1^{k_1},\ldots, u_r^{k_r}]$ if the entry $u_i$ appears $k_i$ times, $i=1,\ldots,r$. 
    We say $Q$ is a {\it ramification point} of $f$ if $e_f(Q)\neq 1$, and $P$ is a {\it branch point} of $f$ if $f^{-1}(P)$ contains a ramification point. 
    We call the multiset $\{E_f(P)\,|\,P\text{ branch point of }f\}$ the {\it ramification type} of $f$. 
    
    The following basic lemma describes ramification using monodromy and inertia groups. In particular in the context of the end of Section \ref{sec:intro}, it will be used to derive the ramification and branch cycles of $h$ from those of $f$. 
    \begin{lem}\label{lem:f-to-h}
    Let $\tilde f:\tilde X\ra X_0$ be a Galois covering with monodromy group $G$,  
    and $I=I(\tilde Q/P)$ the inertia group of $\tilde Q\in \tilde f^{-1}(P)$, for $P\in X_0$. Let $H\leq G$ be a subgroup, $X:=\tilde X/H$, and let  $f:X\ra X_0$ and $\hat f:\tilde X\ra X$ be the natural projections. 
    Then there is a bijection between 
    the double cosets $I\backslash G/H$ and points in $f^{-1}(P)$, given by sending $I\sigma H$ to $Q_\sigma:=\hat f(\tilde Q^\sigma)$. Moreover,  the ramification index $e_\sigma:=e_f(Q_\sigma)$ equals $\abs{I\sigma H}/\abs H$, and $\sigma I^{e_\sigma}\sigma^{-1}$ is an inertia group of $\hat f$ over $Q_\sigma$. 
    \end{lem} 
    \begin{proof} 
    As the action of $G$ on $\tilde f^{-1}(P)$ is transitive and $I$ is a point stabilizer in this action, this action is equivalent to the $G$-action on $I\backslash G$. 
    The image of two points $\tilde Q^\sigma,\tilde Q^\tau \in \tilde f^{-1}(P)$ in $X=\tilde X/H$ coincides if and only if $\tilde Q^{\sigma H} = \tilde Q^{\tau H}$ or equivalently if  $I\sigma H = I\tau H$, for $\sigma,\tau\in G$.
    Hence, the map $I\sigma H\ra Q_\sigma=\hat f(\tilde Q^\sigma)$ is well defined and gives an inclusion $\Orb_\Omega(I)\ra f^{-1}(P).$ It is surjective since the action of $G$ on $\tilde f^{-1}(P)$ is transitive. 
    Since the inertia group $I_{\tilde f}(\tilde Q^\sigma/P)$ is the subgroup $\sigma^{-1} I \sigma$, we have $I_{\hat f}(\tilde Q^\sigma/Q_\sigma)= \sigma^{-1} I \sigma\cap H$. Thus, $$e_f(Q_\sigma)=[\sigma^{-1} I \sigma: \sigma^{-1} I\sigma\cap H] = \frac{ \abs I}{\abs{\sigma^{-1} I\sigma\cap H}} = \frac{\abs{I\sigma H}}{\abs H},$$ and 
    $I_{\hat f}(\tilde Q^\sigma/Q_\sigma)= \sigma^{-1} I \sigma\cap H = \sigma^{-1} I^{e_\sigma}\sigma.$
    \end{proof}
    
    In particular, $E_f(P)$ is the tuple of cardinalities of orbits in $\Orb_\Omega(I)$ in decreasing order. 
    Note that if $f$ is Galois then $e_f(Q)$ is independent of the choice of $Q\in f^{-1}(P)$, in which case we denote $e_f(P):=e_f(Q)$. 
    \subsection{Riemann--Hurwitz} The Riemann--Hurwitz formula expresses the {\it genus} $g_X$  as 
    \begin{align*}
    2(g_X-1) =&  2n(g_{X_0}-1) + \sum_{Q\in X(\K)}R_f(P),\text{ 
    where }\\ 
    R_{f}(P) :=& \displaystyle \sum_{Q\in f^{-1}(P)}(e_f(Q)-1) = \sum_{r\in E_f(P)}(r-1) = n - \abs{E_f(P)} = n - \abs{\Orb_\Omega(x)}.
    \end{align*}
    For a subset $S$ of points of $X_0$, denote by $R_f(S)$ the total contribution 
    $\sum_{P\in S}R_f(P)$ to the Riemann--Hurwitz formula.

    We note that given a degree $m$ covering $\pi: Z\ra X$  and points $Q\in f^{-1}(P),R\in~\pi^{-1}(Q)$, ramification indices are multiplicative $e_{f\circ\pi}(R) = e_\pi(R)e_f(Q)$, and satisfy the equality $\sum_{Q\in f^{-1}(P)} e_{\pi}(Q) = m$. In particular, this gives the {\it chain rule}:
    \begin{equation}\label{equ:chain} 
    \begin{split}
    R_{f\circ\pi}(P) = &  \sum_{Q\in f^{-1}(P)}\sum_{R\in \pi^{-1}(Q)}\bigl( e_\pi(R)(e_f(Q)-1) + e_\pi(R)-1\bigr) \\
    = &  \sum_{Q\in f^{-1}(P)}\bigl( m(e_f(Q)-1) + R_\pi(Q)\bigr) =  mR_f(P) + R_\pi(f^{-1}(P)).  
    \end{split}
    \end{equation}
    
    We restate the following version of Abhyankar's lemma 
    \cite[Lemma 9.2]{NZ}. 
    We say that a covering $h:Z\ra Y$ is {\it minimal that factors through coverings} $h_i:Y_i\ra Y$, $i\in I$ if the only compositions  $h=u\circ v$ for which $u$ factors through $h_i$ are those where $v$ is an isomorphism, that is $u=h_i\circ \hat h_i$ for some coverings $\hat h_i$, $i\in I$ if and only if $\deg v=1$. The greatest common divisor of two positive integers $a,b$ is denoted by $(a,b)$. 
    \begin{lem}\label{lem:abh}
    Let $h_i:Y_i\ra Y$, $i=1,2$ be coverings, and $p:Z\ra Y$ a minimal covering which factors through $h_1$ and $h_2$. Write $p=h_i\circ p_i$, $i=1,2$. 
    Let $P$ be a point of $Y$ and $Q_i\in h_i^{-1}(P)$ a preimage, for $i=1,2$. 
    $$\xymatrix{
    & Z \ar[dl]_{p_1} \ar[dr]^{p_2} \ar[dd]_{p} & \\
    Y_1 \ar[dr]_{h_1}& & Y_2 \ar[dl]^{h_2} \\
    & Y &
    }$$
    \begin{enumerate} 
    \item Then $e_p(Q) = \lcm(e_{h_1}(Q_1),e_{h_2}(Q_2))$ for every $Q\in p_1^{-1}(Q_1)\cap p_2^{-1}(Q_2)$;
    \item If furthermore $\deg p=\deg h_1\cdot\deg h_2$, then  there are $(e_{h_1}(Q_1), e_{h_2}(Q_2))$ points $Q$ in $Z$ with image $p_i(Q)=Q_i$ for both  $i = {1, 2}$. 
    \end{enumerate}
    In particular, 
    \begin{align*}
        R_{p_2}(h_2^{-1}(P)) & =  \sum_{r_1\in E_{h_1}(P),r_2\in E_{h_2}(P)}(r_1,r_2)\Bigl( \frac{\lcm(r_1,r_2)}{r_2}-1 \Bigr) \\
    &  = \sum_{r_1\in E_{h_1}(P),r_2\in E_{h_2}(P)}\bigl( r_1-(r_1,r_2) \bigr). 
    \end{align*}
    \end{lem}
    We note that if the {\it fiber product} $Y_1\#_Y Y_2$ of $h_1$ and $h_2$ is irreducible, then under the assumption of (2), $Z$ is isomorphic to the normalization of this fiber product. 
    \begin{rem}\label{rem:abh}
    Letting $\tilde h_1:\tilde Y_1\ra Y$ be the Galois closure of a covering $h_1:Y_1\ra Y$, we note that Lemma \ref{lem:abh} implies that $e_{\tilde h_1}(P)=\lcm(E_{h_1}(P))$, for every point $P$ of $Y$.  
    Indeed, letting $G_1:=\Mon(h_1)$ and $K_1$ be a point stabilizer,  
    $\tilde h_1$ is a minimal covering which factors through the natural projections $h_1^\sigma:\tilde Y_1/K^\sigma \ra Y$, $\sigma\in G_1$, and
    hence  the assertion follows by applying Lemma \ref{lem:abh} iteratively. \end{rem}
    \subsection{Riemann's existence theorem} \label{sec:RET}
    Fix a transitive subgroup $G\leq S_n$,  conjugacy classes $C_1,\ldots, C_r$ of cyclic subgroups of $G$, 
    and points $P_1,\ldots, P_r$ in $\mP^1$. We call a tuple $x_1,\ldots, x_r\in G$ a {\it product-$1$ tuple for $G$} if $x_1\cdots x_r=1$ and $\langle x_1,\ldots, x_r\rangle =G$. 
    By Riemann's existence theorem, there exists a Galois covering $\tilde f:\tilde X\ra\mP^1$ with monodromy group $G$ such that $C_i$ is the conjugacy class of an inertia group of $\tilde f$ over $P_i$ for $i=1,\ldots,r$ if and only if there exists a product-$1$ tuple $x_i\in C_i$, $i=1,\ldots,r$ for $G$, see \cite[Theorems 2.13 and 5.14]{Vol}. 
    
     Given such $\tilde f$, let $f:\tilde X/H\ra \mP^1$ be the natural projection where  $H:=G\cap S_{n-1}$ is a point stabilizer. 
    Then $E_f(P_i)$ is the multiset of lengths of orbits of $x_i\in C_i$,  $i=1,\ldots,r$. 
    We call the multiset of lengths of orbits of $x_i$, the {\it cycle structure} of $x_i$, and the multiset of all cycle structures of $x_i$, $i=1,\ldots,r$ the {\it ramification type} of the tuple $x_1,\ldots,x_r$. Thus, $E_f(P_i)$ coincides with the cycle structure of $x_i$, $i=1,\ldots,r$, and the ramification type of $f$ coincides with the ramification type of $x_1,\ldots,x_r$. 
    
    A degree $n$ {\it ramification data} is a multiset $\mathcal R$ of partitions $A_1,\ldots, A_r$ of $n$ consisting of positive integers. 
    By Riemann's existence theorem,
    the existence of a degree $n$ covering $f:X\ra \mP^1$ whose ramification type is $\mathcal R$ is equivalent to the existence of a product-$1$ tuple $x_1,\ldots,x_r$ for $G$ such that $x_i$ has cycle structure $A_i$, for $i=1,\ldots, r$. 
    The following basic lemma rules out ramification data from occurring as the ramification type of an indecomposable covering: 
    \begin{lem}\label{lem:hurwitz1}\cite[Lemma 9.1.(a,c)]{NZ}
    Let $p$ be a rational prime, $f:\mP^1\ra \mP^1$  an indecomposable covering with monodromy group $G$, 
    and $P_1,P_2,P_3$ be points of $\mP^1$. 
    \begin{enumerate}
    \item[(a)] 
    If all entries of $E_f(P_1)$ and $E_f(P_2)$ are divisible by $p$, then  $f$ is Galois of degree $p$, and $G\cong C_p$. 
    \item[(c)] If all entries of $E_h(P_1)$ are even and there is a total of exactly two coprime to~$3$ entries in $E_h(P_2)$ and $E_h(P_3)$, then $G$ is a quotient of $A_4$. 
    \end{enumerate}
    \end{lem}
    \subsection{Wreath products and reduced forms}\label{sec:wreath-prod}
    We use the following standard wreath product notation, e.g.~see \cite[Section 2.7]{DM}. Let $\Delta$ and $I$ be sets of cardinalities $\ell\geq 5$ and $t\geq 2$, respectively. Denote by $S_\Delta \wr S_I:=S_\Delta^I\rtimes S_I$ 
     the {\it wreath product} equipped with an action on the set $\Omega:=\Delta^I$ of all tuples $\delta = (\delta(i))_{i\in I}$ given as follows. The semidirect product action is given by 
     $(\sigma^{-1}a\sigma)(i) = a^{\sigma}(i) = a(i^{\sigma^{-1}})$
     for all $a\in S_\Delta^I$, $\sigma\in S_I$ and $i\in I$. The action on $\Omega$ is given by 
    $\delta^{a\sigma}(i) = \delta(i^{\sigma^{-1}})^{a(i^{\sigma^{-1}})},$
    so that $a$ acts pointwise $\delta^a(i) = \delta(i)^{a(i)}$, and $\sigma$ acts by permuting the values of $\delta$ by $\delta^\sigma(i) = \delta(i^{\sigma^{-1}})$,  
    for all $\delta\in \Delta^I, \sigma\in S_I, a\in S_\Delta^I$, and $i\in I$, see \cite[\S 2.5]{DM}\footnote{For $I=\{1,\ldots,t\}$, the the reader may easily switch functions $a\in S_\Delta^I$ with vectors $a=(a_1,\ldots,a_t)$. The $i$-th entry of the vector $a^\sigma$ is then $a_{i^{\sigma^{-1}}}$. We shall adopt this notation when $t=2$, see \S\ref{sec:setup}.}.  Note that $S_\Delta\wr S_I$ is also equipped with an action on $I$ via the projection to $S_I$. 
    
    We shall write $S_\ell\wr S_t$ for the group $S_\Delta\wr S_I$ when $\Delta =\{1,\ldots, \ell\}$ and $I = \{1,\ldots, t\}$. 
    

    To simplify orbit counts, we replace elements in $S_\Delta^I$ by conjugates of the following form. 
    \begin{defn}
    Let $y=a\sigma$, $a\in S_\Delta^I, \sigma\in S_I$, and let $O$ be the set of orbits of $\sigma$ on $I$. We  say $y$  is a {\it reduced form} of $x\in G$ with {\it representatives} $\iota_\theta\in \theta$, $\theta\in O$, if $a(i) = 1$ for every $i\in I\setminus \{\iota_\theta\,|\,\theta\in O\}$, and $y=x^\tau$ for some $\tau\in S_\Delta^I$. 
    \end{defn}
    Elements of reduced forms are sometimes also called {\it (standard) odometers}, see e.g.\  \cite{Pink1}. Every element is conjugate to an element in reduced form:
    \begin{lem}\label{lem:conjugate-1}
    Let $x=a\sigma\in S_\Delta\wr S_\ell$, where $a\in S_\Delta^I$ and $\sigma\in S_I$. 
    Let $O$ be the set of orbits of $\sigma$ and $\iota_\theta\in \theta$, $\theta\in O$, be representatives. 
    Then there is an element $z\in S_\Delta^I$  such that $x^z = b\sigma$  with $b(i) = 1$ for  $i\in \theta \setminus \{\iota_\theta\}$, 
    and $b(\iota_\theta) = \prod_{k=0}^{\abs\theta-1}a(\iota_\theta^{\sigma^{k}})$ for  $\theta\in O$. 
    \end{lem}
    \begin{proof} 
    Define $z\in S_\Delta^I$ on each orbit $\theta\in O$ iteratively. Fix $\theta\in O$ and for simplicity enumerate the orbit $\iota_\theta,\iota_\theta^{\sigma^{-1}},\ldots,\iota_\theta^{\sigma^{-(\abs\theta-1)}}$ by $1,\ldots,\abs{\theta}$, so that $j^\sigma=j-1$ mod $\abs{\theta}$ for $1\leq j\leq\abs{\theta}$. Now set $z(1):=1$ and iteratively set $z(j):=a(j)z(j-1)$ for $j=2,\ldots,\abs{\theta}$. 
    Then $x^z = b\sigma$ where $b:=z^{-1}az^{\sigma^{-1}}\in S_\Delta^I$, and
    \begin{align*}
    b(j) & = z^{-1}(j)a(j)z^{\sigma^{-1}}(j)=z^{-1}(j)a(j)z(j^\sigma) = z^{-1}(j)a(j)z(j-1)=1\text{ for  }j=2,\ldots,\abs{\theta} \\
    b(1) & = a(1)z(1^{\sigma}) = a(1)a(1^{\sigma})z(1^{\sigma^{2}}) = \ldots = \prod_{k=0}^{\abs\theta-1} a(1^{\sigma^{k}}), \text{ as desired}. \qedhere
    \end{align*}
    \end{proof}
    
    \subsection{Groups of product type}\label{sec:product-groups}
    In this paper we consider primitive monodromy subgroups of the following type. We start with the definition of a class of primitive groups $G$ from the Aschbacher--O'Nan--Scott theorem \cite{AS} using its version from  \cite[Theorem 11.2]{Gur}, and embed such $G$ into $S_\Delta\wr S_I$ with respect to the action from  \S\ref{sec:wreath-prod}. Note that \cite[Theorem 11.2]{Gur} is a refinement of the version from \cite[pg.\ 304]{GT}: for, types (i),(ii),(iii) of \cite{Gur} are clearly contained in types (A),(B),(C) of \cite{GT}, and moreover types (iii).(i),(iii).(ii),(iii).(iii) of \cite{Gur} are contained in types (C3), (C2), (C1) of \cite{GT}, resp. As explained in \S\ref{sec:intro}, our concern is with type (C3) of \cite{GT}. As the types from \cite{Gur} are contained in those of \cite{GT}, it suffices to consider type (iii).(iii) of \cite{Gur} which is defined as follows: \\ 
    Let $N_G(L)$ and $C_G(L)$ denote the normalizer and centralizer, resp., of  $L\leq G$. 
    
    \begin{defn}\label{def:product-type}
    Say that a primitive group $G$  with point stabilizer $H$ is of {\it product type} if it has a unique minimal normal subgroup $Q$, this subgroup $Q$ decomposes as $L_1\times\cdots \times L_t$ for isomorphic nonabelian simple groups $L_1,\ldots,L_t$ for $t\geq 1$, and the group $H\cap Q$  decomposes as  $(H\cap L_1)\times \cdots \times (H\cap L_t)$ with $H\cap Q\neq 1$. 
    Moreover for $G$ of product type, the groups $H\cap L_i$, $i=1,\ldots, t$ are conjugate in $H$ and $N_H(L_i)C_G(L_i)$ is a maximal subgroup of $N_G(L_i)$, for $i=1,\ldots,t$, \cite[Theorem 11.2.(iii).(1)]{Gur}. 
    \end{defn}
    
    Throughout the paper, we shall identify groups of product type with subgroups of wreath products as follows. Put $I=\{1,\ldots,t\}$. 
    For $i\in I$, let $S_\Delta^{(i)}\leq S_\Delta^I$ denote the subgroup consisting of tuples $a\in S_\Delta^I$ such that $a(j) =1$ for all $j\in I\setminus \{i\}$. 
    
    \begin{lem} \label{lem:product-type}
    A faithful primitive group $G$ of product type is isomorphic to a subgroup $\hat G\leq S_\Delta\wr S_I$ 
    whose minimal normal subgroup $\hat Q$ decomposes as a direct sum $\hat Q = \bigoplus_{i\in I}(\hat Q\cap S_\Delta^{(i)})$
    of isomorphic nonabelian simple groups acting transitively on $\Delta$. 
    Moreover, the image of $G$ in $S_I$ is transitive, and the action of $\hat G\cap (S_\Delta^I\rtimes S_{I\setminus\{i\}})$ on the $i$-th copy of $\Delta$ is primitive. 
    \end{lem}
    \begin{rem}\label{rem:normal}
    The proof of the lemma relies on the following well known facts:
    \begin{enumerate}
    \item  Let $G$ be a faithful primitive group acting on $\Omega$ and $N$ a nontrivial normal subgroup of $G$, then $N$ acts transitively on $\Omega$ \cite[Theorem 1.6A.v]{DM}.
    \item Let $L$ be a nonabelian simple group acting transitively on a set $\Delta$ and $M$ the normalizer of $L$ in $S_\Delta$. Let $I$ be a finite nonempty set, so that $M\wr S_I$ and $L^I$ act on $\Delta^I$. 
    Let $N$ be the normalizer of $L^I$ in $S_{\Delta^I}$. If $N$ is primitive then $N=M\wr S_I$ \cite[Lemma 4.5A]{DM}. 
    \item Suppose $W$ acts transitively on $\Omega$,  and permutes the set of fibers of a surjection $\phi:\Omega\ra \Upsilon$, and hence also acts on $\Upsilon$. 
    Suppose that the kernel $C\leq W$ of the action on $\Upsilon$ acts transitively on each fiber $\phi^{-1}(x)$,  $x\in\Upsilon$. If $H$ is the stabilizer of $\omega\in \Omega$, then the stabilizer of $\phi(\omega)$ is $H\cdot C$. Indeed for $v\in W$,  one has $\phi(\omega)^v=\phi(\omega)$ if and only if  $\phi(\omega^v)=\phi(\omega)$. As $C$ is transitive on $\phi^{-1}(\phi(\omega))$, this condition is equivalent to the existence of $c\in C$ such that $\omega^{vc}=\omega$, that is, to $v$ being in $H\cdot C$.  
    \end{enumerate}
    \end{rem}
    \begin{proof}[Proof of Lemma \ref{lem:product-type}]
    Since $G$ is primitive, $H$ is maximal in $G$ and $Q$ is transitive by Remark \ref{rem:normal}, so that $G=HQ$. Hence the bijection $H\backslash G\cong (Q\cap H)\backslash Q$ gives a faithful and primitive action of $G$ on $(Q\cap H)\backslash Q$ with point stabilizer $H$. The action is explicitly described as follows: Since $(Hx)^h = Hxh = Hh^{-1}xh$, we have $((Q\cap H) x)^h= (Q\cap H)h^{-1}xh$ and $((Q\cap H)\cdot x)^q= (Q\cap H)xq$, for all $x,q\in Q$, and $h\in H$. Thus, the bijection $(H\cap Q)\backslash Q\cong (L_1\cap H)\backslash L_1\times \cdots \times (L_t\cap H)\backslash L_t$ gives a faithful and primitive $G$-action on $(L_1\cap H)\backslash L_1\times \cdots \times (L_t\cap H)\backslash L_t$, where  $H$ is a point stabilizer, and $Q= L_1\times\cdots \times L_t$ acts by right multiplication. 
    
    Let $\Delta:=(L_1\cap H)\backslash L_1$, and write $I=\{1,\ldots, t\}$. Note that since $G$ acts transitively on the copies $L_1,\ldots,L_t$ of $L$, and since $G=HQ$, so does $H$ acts transitively on the copies. Since an element of $H$ conjugating $L_1$ and $L_i$ also conjugates $L_1\cap H$ and $L_i\cap H$, it defines an isomorphism $\phi_i:L_1\ra L_i$ which induces an isomorphism $\Delta\cong (H\cap L_i)\backslash L_i$, for $i\in I$.  
    Identifying each of the $L_i$-sets $(H\cap L_i)\backslash L_i$ with $\Delta$ via these isomorphisms for all $i\in I$, 
    we obtain a faithful and primitive action of $G$ on $\Omega:=\Delta^I$, where $H$ is a point stabilizer and $Q=L_1\times \cdots \times L_t$ acts on $\Omega=\Delta^t$ coordinatewise. 
    Let $\hat G,\hat H, \hat Q,$ and $\hat L_i = \hat Q\cap S_\Delta^{(i)}$ be the images of $G,H,Q$ and $L_i$ in $S_\Omega$, respectively, for $i\in I$. 
    Since $\hat G$ is contained in the normalizer $N$ of $\hat Q$ in $S_\Omega$, the subgroup $N$ acts primitively on $\Omega$. 
    Since $N$ is primitive, and $\hat Q = \prod_{i=1}^t\hat L_i$,  Remark \ref{rem:normal}.(2) implies that $N=M\wr S_I$, where $M$ is the normalizer of $\hat Q\cap S_\Delta^{(i)}$ in $S_\Delta^{(i)}$, for $i\in I$. This proves the first assertion. 
    
    Since $G$ acts transitively on $\{L_1,\ldots,L_t\}$ by conjugation, it also acts transitively on $\{\hat L_i = \hat Q\cap S_\Delta^{(i)}\suchthat i\in I\}$ by conjugation. 
    It is straightforward to check that this action is equivalent to the $G$-action on $I=\{1,\ldots,t\}$, and hence the latter is also transitive. 
    
    Fix $i\in I$. Since the normalizer of $\hat L_i$ in $N$ is $M^I\rtimes S_{I\setminus\{i\}}$, we have $N_{\hat G}(\hat L_i) = \hat G\cap 
    (S_\Delta^I\rtimes S_{I\setminus\{i\}})$, for $i\in I$. 
    In particular, $N_{\hat G}(\hat L_i)$ acts on the $i$-th copy $\Delta^{(i)}$ of $\Delta$ in $\Omega$. Since $\hat L_i$ is a nonabelian simple group, the kernel of this action is $C_{\hat G}(\hat L_i) = \hat G\cap (M^{I\setminus\{i\}}\rtimes S_{I\setminus\{i\}})$. Since $G$ is of product type, this kernel acts transitively on the fibers of the projection $\Omega\ra \Delta^{(i)}$, and hence Remark \ref{rem:normal}.(3) implies that the point stabilizer in the action of $N_{\hat G}(\hat L_i)$ on $\Delta^{(i)}$ is  $(N_{\hat G}(\hat L_i)\cap \hat H)C_{\hat G}(\hat L_i)=N_{\hat H}(\hat L_i) C_{\hat G}(\hat L_i)$. Since $N_H(L_i)C_G(L_i)$ is maximal in $N_G(L_i)$ by Definition \ref{def:product-type}, it follows that the point stabilizer $N_{\hat H}(\hat L_i)C_{\hat G}(\hat L_i)$ is maximal in $N_{\hat G}(\hat L_i)$, and hence $N_{\hat G}(\hat L_i) = \hat G\cap (S_\Delta^I\rtimes S_{I\setminus\{i\}})$ acts primitively. 
    \end{proof}
    From now on, we shall identify groups of product type with subgroups $G\leq S_\Delta\wr S_I$ of a wreath product  via Lemma \ref{lem:product-type}. We note that the proof of Theorem \ref{thm:main-wreath} for $t\geq 3$ does not use the primitivity of $G\leq S_\Delta\wr S_I$ but rather the other properties in Lemma \ref{lem:product-type}.  
    More specifically, the only used properties are the transitivity of $K=G\cap S_\Delta^I$  on $\Delta^I$ (which follows from that of $Q$), and the transitivity of $G$ on $I$. 
    The last property in Lemma  \ref{lem:product-type}, namely the primitivity of $G\cap (S_\Delta^I\rtimes S_{I\setminus\{i\}})$ on the $i$-th copy of $\Delta$ is used only in the proof for $t=2$, namely in the final steps where Lemmas \ref{lem:hurwitz1} and \ref{lem:primitive} are applied.
     
     The following basic properties then hold: 
    \begin{lem}\label{rem:product-type}
    Let $G\leq S_\Delta\wr S_I$ be a group of product type, let $K:=G\cap S_\Delta^I$ and  $G_i:=G\cap (S_\Delta^I\rtimes S_{I\setminus\{i\}})$ be the kernel of the action on $I$ and the stabilizer of $i\in I$, resp. 
    Let $K_i$ and $H_i$ be point stabilizers in the action of $K$ and $G_i$ on the $i$-th coordinate of $\Delta^I$, $i\in I$, resp., and $Q\subseteq K$ the minimal normal subgroup of $G$. Set $\ell:=\abs\Delta$ and $t:=\abs I$. Then:
    \end{lem}
    \begin{enumerate} 
    \item $K$ is transitive on $\Delta^{I}$ and hence $[G_i:H_i]=[K:K_i]=\ell$, $[K:H\cap K]=[G:H]=\ell^t$, and $[H:H\cap K]=[G:K]$. 
    \item The subgroups $G_i$, $i\in I$ (resp.\ $K_i$, $i\in I$) are conjugate in $G$.
    \item If $t=2$, then $K$ acts on each of the copies of $\Delta$ primitively. 
    \end{enumerate}
    \begin{proof}
    (1) Since $Q\cap S_\Delta^{(i)}$, $i\in I$ is transitive on $\Delta$, the action of $Q$ and hence of $K$ on $\Delta^I$ is transitive. 
    Hence, $[K:K_i]=\ell$, and  $[K:H\cap K]=\ell^t$, for $i\in I$ and $t=\abs I$. The latter also implies $[H:H\cap K]=[G:K]$. 
    
    \noindent (2) Letting $\sigma\in G$ be an element 
    such that $i^\sigma=j$, the subgroups $\sigma^{-1} K_i\sigma$ and $\sigma^{-1} H_i\sigma$ are point stabilizer in the actions of $K$ and $G_j=\sigma^{-1}G_i\sigma$ on the $j$-th coordinate of $\Delta^I$, resp., for $i,j\in I$. As $G$ is transitive on $I$, this yields the assertion. 
    
    (3) If $t=2$, 
    then $G_i=\hat G\cap (S_\Delta^I\rtimes S_{I\setminus\{i\}}) = K$ and Lemma \ref{lem:product-type} implies that the action of $G_i=K$
     on each of the two copies of $\Delta$ is primitive. 
    \end{proof}
    

    We use the following lemma to ensure that a group is of product type:
    \begin{lem}\label{lem:primitive}\label{lem:cyc-exist} 
    Let $G\leq S_\Delta\wr S_I$, $\abs\Delta\geq 5$, act transitively on $\Delta^I$,   
    and assume that the projection of $K:= S_\Delta^I\cap G$ to each coordinate contains $A_\Delta$. Then
    \begin{enumerate}
    \item $G$ is primitive if and only if $K\supseteq A_\Delta^I$ and the $G$-action on $I$ is transitive;
    \item if $\abs I=2$ 
     and $\abs\Delta\neq 6$, 
    then either $K \supseteq A_\Delta^I$ or $K \subseteq \{ (a,v^{-1}av))\,|\,a\in S_\Delta\}$ for some $v\in S_\Delta$. 
    \end{enumerate}
    \end{lem}
    \begin{proof} 
    Write $\Delta = \{1,\ldots, \ell\}$,  $I=\{1,\ldots, t\}$, let $\pi_i: K \ra S_\ell$, $i\in I$ be the natural projections, and let  $Q:=A_\ell^t$.
    
    First assume $K\supseteq A_\ell^t$ and $G$ is transitive on $I$. 
    We claim that there are no nontrivial intermediate subgroups $H\leq M\leq G$, where $H$ is a point stabilizer.  Since $G$ acts transitively on $I$ it acts transitively on the $t$ copies of $A_\ell$. 
    Since in addition $G=HQ$ as $Q$ is transitive, the action of $H$ and hence of $M$ on the $t$ copies is also transitive. It follows by the Lemma \ref{rem:product-type} (with $G$ replaced by $M$ and $K$ replaced by $Q\cap M$) that $\pi_i(Q\cap M)$, $i\in I$ are all isomorphic.  Since $Q\cap H\cong A_{\ell-1}^t$, we have $\pi_i(Q\cap M) = A_{\ell-1}$ or $A_\ell$ for all $i\in I$. 
    Thus $Q\cap M$ is a subdirect product of either $t$ copies of $A_\ell$ or of $t$ copies of  $A_{\ell-1}$. Since in addition $Q\cap M$ contains $Q\cap H\cong A_{\ell-1}^t$, it follows that $Q\cap M =Q$ or $Q\cap H$.  
    Since $Q$ is transitive, we also have $[Q\cap M:Q\cap H] = [M:H]$ and hence either $[M:H]=1$ or $[M:H] = [Q:Q\cap H] = [G:H]$ yielding the claim.
    
    
    Conversely, if the $G$-action on $I$ has an  orbit $\theta$ which is properly contained in $I$, then the $G$-action on $\Delta^I$ is imprimitive since the blocks $U_{\delta_0} = \{\delta\in \Delta^I\suchthat \delta_{|\theta} = \delta_0\}$, $\delta_0\in\Delta^{\theta}$ form a nontrivial partition. 
    Assuming $G$ is primitive, we let $R\leq K$ be a minimal normal subgroup of $G$ and claim that $R=Q$. As $R\lhd K$, one has $\pi_i(R)\lhd \pi_i(K)$ for all $i\in I$. Since $G$ acts transitively on $I$ and $\pi_i(K)\supseteq A_\ell$, the normality  $\pi_i(R)\lhd \pi_i(K)$ implies that $\pi_i(R) = A_\ell$ for every $i\in I$. 
    Thus $R$ is a subdirect product of $A_\ell^t=Q$ and hence is isomorphic via the natural projection to  $A_\ell^U$ for some $U\subseteq I$. 
    Since $R\lhd G$ and $G$ is primitive,  $R$ acts transitively on $\Delta^I$ by Remark \ref{rem:normal}, and hence $U=I$, proving that $R=Q$, completing (1). 

    Now assume $t=2$. 
    If $\ker\pi_1\neq 1$, then $\pi_2(\ker\pi_1)$ is a nontrivial normal subgroup of $\pi_2(K)$. Since $\pi_2(K)\supseteq A_\ell$, this implies $\pi_2(\ker\pi_1)\supseteq A_\ell$. Thus $K\supseteq \ker\pi_1\supseteq 1\times A_\ell$. 
    Since in addition $\pi_1(K) \supseteq  A_\ell$, we deduce that  $K\supseteq  A_\ell^2$. 
    Henceforth assume both projections $\pi_1,\pi_2$ are injective. Then $\pi_2\circ\pi_1^{-1}:\pi_1(K) \ra \pi_2(K)$ is an isomorphism.  Identifying $\pi_1(K)$ and $\pi_2(K)$ with subgroups of $S_\ell$, we get that $\pi_2\circ \pi_1^{-1}$ is an automorphism of $A_\ell \leq \pi_1(K) \leq S_\ell$, and hence is given by conjugation by some $v\in S_\ell$, completing the proof of (2). 
    \end{proof} 
 Note that in Part (1) and its proof, $A_\Delta$ can be replaced by any primitive nonabelian simple group $L\leq S_\Delta$. For Part (2), one additionally needs the assumption that $\Aut(L')$ embeds in $S_\Delta$ for $L\leq L'\leq \Aut(L)$, or equivalently that the action of $L$ extends to $\Aut(L)$. We  use only the case $L=A_\Delta$ and hence restrict to this case for simplicity.    


    \subsection{Setup}\label{sec:setup}
    The following setup is used throughout the proof of Theorem \ref{thm:main-wreath} and may be assumed by Lemma \ref{lem:product-type}. 
    
    Let  $f:X\ra X_0$ be an indecomposable covering with monodromy group $G\leq S_\Delta\wr S_I$ of product type and let $\tilde f:\tilde X\ra X_0$ be its Galois closure. As in Section~\ref{sec:product-groups}, we let $\Omega:=\Delta^I$, let $H\leq G$ be a point stabilizer in the $G$-action on $\Omega$, let $K:= G\cap S_\Delta^I$, and let $K_i$ be a point stabilizer in the action of $K$ on the $i$-th copy of $\Delta$ in $\Omega$, for $i\in I$. 
    As sketched in Section \ref{sec:method}, we let $Z:=\tilde X/(H\cap K)$,  $Y:=\tilde X/K$, and $Y_i:=\tilde X/K_i$ for $i\in I$. We then have the following diagram of natural projections: 
    $$
    \xymatrix{
     Z \ar[r]^{\pi}\ar@/_1pc/[dd]_{h} \ar[d] & X\ar[dd]^{f} \\
     Y_i \ar[d]^{h_i} & \\
     Y \ar[r]^{\pi_0} & X_0,
    }
    $$
    for $i\in I$. Throughout the paper we let $t:=\abs I$ be a fixed integer which is at least~$2$. 
    Also, put $m:=\deg \pi_0$, and $\ell:=\abs \Delta$. 
    
    \begin{rem}\label{rem:setup}
    \begin{enumerate}
    \item  By Lemma \ref{rem:product-type}.(1), 
    $\deg f=\ell^t$, $\deg h = [K:H\cap K]=\ell^t$, $\deg h_i = [K:K_i]= \ell$ for $i\in I$, and $\deg\pi = m$. In particular the fiber product of $h_i$ and $h_j$ is irreducible for every pair of distinct $i,j\in\{1,\ldots,t\}$.
    \item Fix  $1\in I$. By Lemma~\ref{rem:product-type}.(2),  there exists $\sigma_i\in G$ such that $K_1^{\sigma_i} = K_i$, $i\in I$. 
    Such $\sigma_i$ induces an automorphism $\oline \sigma_i:Y\ra Y$ sending the orbit of $K$ on $\tilde P\in \tilde X$ to the orbit of $K = \sigma_i^{-1}K\sigma_i$ on $\sigma_i^{-1}(\tilde P)$. 
    Similarly, $\sigma_i$ induces an isomorphism $\sigma_i^*:Y_1\ra Y_i$, which sends the orbit of $K_1$ on $\tilde P\in \tilde X$ to the orbit of $K_i = \sigma_i^{-1}K_1\sigma_i$ on $\sigma_i^{-1}(\tilde P)$, so that $h_i\circ \sigma_i^* = \oline\sigma_i\circ h_1$. Hence, $E_{h_1}(P)=E_{h_i}(\oline \sigma_i(P))$ for $P\in Y$.  
    \item For $t=2$,  the action of $K$ on each copy of $\Delta$ in $\Omega$ is primitive by Lemma \ref{rem:product-type}.(3), and hence $h_i$ is indecomposable for $i\in I$. 
    \end{enumerate}
    \end{rem}

    Note that for the purpose of computing orbits, indices, and the Riemann--Hurwitz contribution, a branch cycle $y\in G$ over $P$ can always be replaced by $x\in S_\Delta\wr S_I$  of reduced form (even though $x$ might not be in $G$) since $G$ inherits its action from $S_\Delta\wr S_I$ and the above quantities do not change under conjugation in $S_\Delta\wr S_I$. 
    
    For $t=2$, we shall specify the ramification type of $f$ by writing the conjugacy class in $S_\ell \wr S_2$ of a branch cycle over each branch point of $f$.  For $s\in S_2$, and $\alpha_i\in S_\ell$ with cycle structure $A_i$, $i=1,2$, we write the conjugacy class of $(\alpha_1,1)s$ (resp.\ $(\alpha_1,\alpha_2)$) in $S_\ell \wr S_2$ as $(A_1,1)s$ (resp.\ $(A_1,A_2)$) and note that every element in $S_\ell\wr S_2\setminus S_\ell^2$ (resp.\ in $S_\ell^2$ is lies in such a conjugacy class by Lemma \ref{lem:conjugate-1}. 
    In this notation we describe the relation between the ramification of $f$ and that of $\pi_0$ and $h=(h_1,h_2)$:
    \begin{rem}\label{rem:t=2-ramification}
    Fix $P\in X_0(\K)$. If $E_f(P) = (A_1,1)s$, then a branch cycle of $f$ at $P$ is conjugate to  $x\in S_\ell \wr S_2$ 
    of the form $x=(\alpha_1,1)s$ where $\alpha_1\in S_\ell$ is of cycle structure $A_1$. By Lemma~\ref{lem:f-to-h},  $x^2=(u,u)$ is conjugate in $S_\ell \wr S_2$ and hence also in $S_\ell^2$ to a branch cycle of $h$, so that $\pi^{-1}(P)=\{Q\}$ and $E_h(Q)=(A_1,A_1)$. 
    
    If on the other hand $E_f(P)= (A_1,A_2)$, then a branch cycle of $f$ at $P$ is conjugate to $x\in S_\ell^2$ of the form $x=(a,b)\in S_\ell^2$ where $\alpha_1,\alpha_2\in S_\ell$ are of cycle structures $A_1,A_2$, resp. By Lemma \ref{lem:f-to-h}, $(a,b)$ and $(b,a)$ are conjugate in $S_\ell^2$ to branch cycles of $h$ over $Q_1,Q_2$, resp., where $\pi_0^{-1}(P) = \{Q_1,Q_2\}$.
    
    Conversely, the ramification of $f$ is determined by the ramification of $h_1,h_2$ and $\pi_0$: 
    Letting $Q\in \pi_0^{-1}(P)$, it then follows that the ramification type of $f$ at $P$ is $(E_{h_1}(Q),1)s$ if $P$ is a branch point of $\pi_0$, and it is  $(E_{h_1}(Q),E_{h_2}(Q))$ otherwise. 
    \end{rem}

    \section{The ramification types in Theorem \ref{thm:main-wreath}}\label{sec:prelim-t=2}    
    In Table~\ref{table:wreath} below we list the ramification types in Theorem~\ref{thm:main-wreath} corresponding to indecomposable coverings $f:X\ra \mP^1$ with monodromy group $A_\ell^2\lneq G\leq S_\ell \wr S_2$ of product type. We list the {ramification types more concisely by listing the conjugacy classes of branch cycles $x_1,\ldots, x_r$, over the branch points $P_1,\ldots,P_r$ of $f$, in the symmetric normalizer $S_\ell \wr S_2 = N_{S_{\ell^2}}(G)$ of $G$ in $S_{\ell^2}$, see \cite[Lemma 4.5(A)]{DM}. The ramification type of $f$ is then  easily read from these conjugacy classes since $E_f(P_i)$ coincides with the tuple of cardinalities of $\Orb_{\{1,2\}}(x_i)$.}
    As in Setup~\ref{sec:setup}, we represent elements (resp.~conjugacy classes) in $S_\ell \wr S_2$ by $(a_1,a_2)\s^j$ (resp.~$(u_1,1)\s$ and $(u_1,u_2)$) where $a_1,a_2\in S_\ell$ (resp.~$u_1,u_2$ are partitions of $\ell$),  $\s$ is the swap in $S_2$, and $j\in\{0,1\}$.

    We note that if $f:X\ra \mP^1$ is an indecomposable covering with monodromy group $G\leq S_\ell \wr S_2$ such that the ramification of $f$ appears in Table~\ref{table:wreath}, then $G$ is of product type with a minimal normal subgroup $A_\ell^2$. For this, as in Section \ref{sec:RET}, it suffices to show: 
    \begin{lem}\label{lem:All2}
     Let $G\leq S_\ell \wr S_2$, $\ell\geq 9$, be a group of product type, and $x_1,\ldots,x_r$ a product-$1$ tuple generating $G$ whose ramification type appears in Table~\ref{table:wreath}. Then $G\supseteq A_\ell^2$.
    \end{lem}

    \begin{rem}\label{rem:jordan} 
    The proof relies on Jordan's theorem \cite[Theorem 3.3.E]{DM} and its consequence \cite[Example 3.3.1]{DM}. These assert that 
     a primitive subgroup of $S_\ell$ containing either a $p$-cycle for a prime $p<\ell-2$, or a product of two $2$-cycles if $\ell\geq 9$,  contains all of $A_\ell$. 
    \end{rem}
    \begin{proof}[Proof of Lemma \ref{lem:All2}]
    For each ramification type in Table~\ref{table:wreath}, by considering the powers $x_j^a$, $a\leq 6$, for some $j\in\{1,\ldots,r\}$, we see that  $K:=G\cap S_\ell^2$ contains an element $(a_1,a_2)$ such that one of its coordinates $a_i$ is a $3$-cyclc, or a $2$-cycle, or a product of two disjoint $2$-cycles. Since $G$ is of product type, the projection of $K$ to both coordinates is primitive by Lemma \ref{rem:product-type}.(3). As the image of one of the projections is primitive and contains $a_i$, it contains $A_\ell$ by Remark \ref{rem:jordan}. Conjugating the preimage of $A_\ell$ by  an element in $G\setminus K$, we get that the images of projections on both coordinates contain $A_\ell$.
    As $G$ is primitive,  Lemma \ref{lem:cyc-exist} implies that $G\supseteq A_\ell^2$. 
     \end{proof}
     The existence of coverings with such ramification types and the determination of their monodromy groups is given in Section \ref{sec:h1-branches}. 
    \begin{rem}\label{rem:ramification-genus}
    If a covering $f:X\ra \mP^1$ has one of the ramification types I1A.1-I1A.3 or F4.1-F4.5, then  $g_X=1$, and for any other ramification type in Table~\ref{table:wreath} one has $g_X=0$. The genus is computed in each case using Lemma \ref{lem:t=2-gXY1}. Here, the ramification of $h_1$ and $h_2$ is found by Remark \ref{rem:t=2-ramification} and $g_{Y_1}$ is computed using the Riemann--Hurwitz formula for $h_1$. 
    \end{rem}

    \begin{table}
    \caption{Ramification types of coverings $f:X\ra\mP^1$ with monodromy group 
    $A_\ell^2\leq G\leq S_\ell \wr S_2$ of product type.  
    We follow the notation of Setup \ref{sec:setup}. 
    Here $\ell\geq 7$, $s$ is the swap in $S_2$, and $a$ is any integer satisfying $0<a<\ell/2$ and $(a,\ell)=1$. The notation $u^v$ means that $u$ appears $v$ times in the multiset.  }\label{table:wreath}
    $$\begin{array}{| l | l |}
    \hline
    I1.1 & ([\ell],[1^\ell])s, ([a,\ell-a],[1^\ell])s, ([2, 1^{\ell-2}], [1^\ell]) \\
    
    \hline
    I1A.1 &  ([\ell],[\ell]),  ([3,1^{\ell-3}],[1^\ell]),  s,  s  \\ 
    I1A.2a &  ([\ell],[\ell]),  ([2,1^{\ell-2}],[1^\ell]) s,  ([2,1^{\ell-2}],[1^\ell]) s \\ 
    I1A.2b & ([\ell],[\ell]),  ([2,1^{\ell-2}],[1^\ell]),  ([2,1^{\ell-2}],[1^\ell]) s,  s  \\ 
    I1A.2c & ([\ell],[\ell]),  ([2,1^{\ell-2}],[1^\ell]),  ([2,1^{\ell-2}],[1^\ell]),  s,  s  \\ 
    I1A.3 & ([\ell],[\ell]),  ([2^2,1^{\ell-4}],[1^\ell]),  s,  s  \\ 
    I1A.4 & ([a,\ell-a],[a,\ell-a]),  ([3,1^{\ell-3}],[1^\ell]), s,  s  \\ 
    I1A.5a & ([a,\ell-a],[a,\ell-a]),  ([2,1^{\ell-2}],[1^\ell]) s,  ([2,1^{\ell-2}],[1^\ell]) s \\ 
    I1A.5b & ([a,\ell-a],[a,\ell-a]),  ([2,1^{\ell-2}],[1^\ell]),  ([2,1^{\ell-2}],[1^\ell]) s,  s  \\ 
    I1A.5c & ([a,\ell-a],[a,\ell-a]),  ([2,1^{\ell-2}],[1^\ell]),  ([2,1^{\ell-2}],[1^\ell]),  s,  s  \\ 
    I1A.6 & ([a,\ell-a],[a,\ell-a]),  ([2^2,1^{\ell-4}],[1^\ell]), s,  s  \\ 
    I1A.7a & ([\ell],[a,\ell-a]),  ([2,1^{\ell-2}],[1^\ell]),  s,  s  \\ 
    I1A.7b & ([\ell],[a,\ell-a]),  ([2,1^{\ell-2}],[1^\ell]) s,  s \\ 
    \hline
    I2.1a  & ([\ell],[1^\ell])s, s, ([1^3,2^{(\ell-3)/2}],[1,2^{(\ell-1)/2}]), ([1^{\ell-2},2],[1^\ell]) \\
    I2.1b & ([\ell],[1^\ell])s, ([1^{\ell-2},2],[1^\ell])s, ([1^3,2^{(\ell-3)/2}],[1,2^{(\ell-1)/2}]) \\
    I2.2a & ([\ell],[1^\ell])s, ([1^{\ell-2},2],[1^\ell])s, ([1^2,2^{\ell/2-1}],[1^2,2^{\ell/2-1}])  \\
    I2.2b & ([\ell],[1^\ell])s, s, ([1^2,2^{\ell/2-1}],[1^2,2^{\ell/2-1}]), ([1^{\ell-2},2],[1^\ell]) \\
    I2.3 & ([\ell],[1^\ell])s, s, ([2^{(\ell-3)/2},3],[1^3,2^{(\ell-3)/2}]) \\
    I2.4 & ([\ell],[1^\ell])s, s, ([1,2^{\ell/2-2},3],[1^2,2^{\ell/2-1}]) \\
    I2.5 & ([\ell],[1^\ell])s, s, ([1^2,2^{(\ell-5)/2},3],[1,2^{(\ell-1)/2}]) \\
    I2.6 & ([\ell],[1^\ell])s, s, ([1,2^{(\ell-5)/2},4],[1^3,2^{(\ell-3)/2}]) \\
    I2.7 & ([\ell],[1^\ell])s, s, ([1^2,2^{(\ell-6)/2)},4],[1^2,2^{(\ell-2)/2}]) \\
    I2.8 & ([\ell],[1^\ell])s, s, ([1^3,2^{(\ell-7)/2},4],[1,2^{(\ell-1)/2}]) \\
    I2.9a & ([a,\ell-a],[1^\ell])s, ([1^{\ell-2},2],[1^\ell])s, ([1^2,2^{(\ell-2)/2}],[2^{\ell/2}]) \\
    I2.9b & ([a,\ell-a],[1^\ell])s, s, ([1^2,2^{(\ell-2)/2}],[2^{\ell/2}]), ([1^{\ell-2},2],[1^\ell]) \\
    I2.10a & ([a,\ell-a],[1^\ell])s, s, ([1,2^{(\ell-1)/2)}],[1,2^{(\ell-1)/2}]), ([1^{\ell-2},2],[1^\ell]) \\
    I2.10b & ([a,\ell-a],[1^\ell])s, ([1^{\ell-2},2],[1^\ell])s, ([1,2^{(\ell-1)/2}],[1,2^{(\ell-1)/2}]) \\
    I2.11 & ([a,\ell-a],[1^\ell])s, s, ([1^2,2^{(\ell-6)/2},4],[2^{\ell/2}]) \\
    I2.12 & ([a,\ell-a],[1^\ell])s, s, ([1,2^{(\ell-5)/2},4],[1,2^{(\ell-1)/2}]) \\
    I2.13 & ([a,\ell-a],[1^\ell])s, s, ([2^{(\ell-4)/2},4],[1^2,2^{(\ell-2)/2}]) \\
    I2.14 & ([a,\ell-a],[1^\ell])s, s, ([2^{(\ell-3)/2},3],[1,2^{(\ell-1)/2}]) \\
    I2.15 & ([a,\ell-a],[1^\ell])s, s, ([1,2^{(\ell-4)/2},3],[2^{\ell/2}]) \\
    \hline
    F2.1 & ([1,2,3^{(\ell-3)/3}],[3^{\ell/3}]),  ([1,2,3^{(\ell-3)/3}],[1^\ell]) s,  s \\ 
     F2.2 & ([2,3^{(\ell-2)/3}],[1^2,3^{(\ell-2)/3}]),  ([2,3^{(\ell-2)/3}],[1^\ell]) s,  s  \\ 
     F2.3 & ([1,3^{(\ell-1)/3}],[2^2,3^{(\ell-4)/3}]),  ([1,3^{(\ell-1)/3}],[1^\ell]) s,  s  \\ 
     \hline
    F3.1 & ([1,3, 4^{(\ell-4)/4}], [4^{\ell/4}]), ([1^2,2^{(\ell-2)/2}], [1^\ell])s, s \\
    F3.2 & ([2,3,4^{(\ell-5)/4}],[1,4^{(\ell-1)/4}]),  ([1,2^{(\ell-1)/2}],[1^\ell]) s,  s  \\ 
    F3.3 & ([3,4^{(\ell-3)/4}],[1,2,4^{(\ell-3)/4}]),  ([1,2^{(\ell-1)/2}],[1^\ell]) s,  s  \\ 
    \hline
    \end{array}$$
    \end{table}
    \setcounter{table}{0}
    \begin{table}
    \caption{Continued. }
    \label{table:wreath2}\label{table:wreathY=2}\label{table:wreathY=2N}
    $$\begin{array}{| l | l |}
    \hline
    F1A.1a & ([1^2,2^{(\ell-2)/2}],[2^{\ell/2}]), ([1^2,2^{(\ell-2)/2}],[1^2,2^{(\ell-2)/2}]), ([2,1^{\ell-2}],[1^\ell]) s, s \\ 
    F1A.1b & ([1^2,2^{(\ell-2)/2}],[2^{\ell/2}]), ([1^2,2^{(\ell-2)/2}],[1^2,2^{(\ell-2)/2}]), ([2,1^{\ell-2}],[1^\ell]),  s, s \\
    F1A.2a & ([1,2^{(\ell-1)/2}],[1^3,2^{(\ell-3)/2}]), ([1,2^{(\ell-1)/2}],[1,2^{(\ell-1)/2}]), ([2,1^{\ell-2}],[1^\ell]) s, s \\ 
    F1A.2b & ([1,2^{(\ell-1)/2}],[1^3,2^{(\ell-3)/2}]), ([1,2^{(\ell-1)/2}],[1,2^{(\ell-1)/2}]), ([2,1^{\ell-2}],[1^\ell]), s, s \\
    F1A.3a & ([3,2^{(\ell-3)/2}],[1^3,2^{(\ell-3)/2}]), ([1,2^{(\ell-1)/2}],[1,2^{(\ell-1)/2}]), s, s  \\ 
    F1A.3b & ([3,2^{(\ell-3)/2}],[1,2^{(\ell-1)/2}]), ([1^3,2^{(\ell-3)/2}],[1,2^{(\ell-1)/2}]), s, s  \\ 
    F1A.4a & ([1,3,2^{(\ell-4)/2}],[2^{\ell/2}]), ([1^2,2^{(\ell-2)/2}],[1^2,2^{(\ell-2)/2}]), s, s \\
    F1A.4b  & ([1,3,2^{(\ell-4)/2}],[1^2,2^{(\ell-2)/2}]), ([2^{\ell/2}],[1^2,2^{(\ell-2)/2}]), s, s  \\ 
    F1A.5 & ([1,2^{(\ell-1)/2}],[1^2,3,2^{(\ell-5)/2}]), ([1,2^{(\ell-1)/2}],[1,2^{(\ell-1)/2}]), s, s  \\ 
    F1A.6a & ([1,4,2^{(\ell-5)/2}],[1^3,2^{(\ell-3)/2}]), ([1,2^{(\ell-1)/2}],[1,2^{(\ell-1)/2}]), s, s  \\ 
    F1A.6b & ([1,4,2^{(\ell-5)/2}],[1,2^{(\ell-1)/2}]), ([1^3,2^{(\ell-3)/2}],[1,2^{(\ell-1)/2}]), s, s  \\ 
    F1A.7a & ([4,1^2,2^{(\ell-6)/2}],[2^{\ell/2}]), ([1^2,2^{(\ell-2)/2}],[1^2,2^{(\ell-2)/2}]), s, s \\
    F1A.7b & ([4,1^2,2^{(\ell-6)/2}],[1^2,2^{(\ell-2)/2}]), ([2^{\ell/2}],[1^2,2^{(\ell-2)/2}]), s, s  \\ 
    F1A.8 & ([1,2^{(\ell-1)/2}],[1^3,4,2^{(\ell-7)/2}]), ([1,2^{(\ell-1)/2}],[1,2^{(\ell-1)/2}]), s, s  \\ 
    F1A.9 & ([1^2,2^{(\ell-2)/2}],[4,2^{(\ell-4)/2}]), ([1^2,2^{(\ell-2)/2}],[1^2,2^{(\ell-2)/2}]), s, s  \\ 
    \hline
    F4.1 &  \s, \s, ([1^{\ell-2},2],1) \s, ([1^{\ell-2},2],1) \s \\
    F4.2 & \s, \s, \s, ([1^{\ell-2},2],1) \s, ([1^{\ell-2},2],1) \\
    F4.3 & \s, \s, \s, \s, ([1^{\ell-2},2],1), ([1^{\ell-2},2],1) \\
    F4.4 & \s, \s, \s, \s, ([1^{\ell-4},2^2],1) \\
    F4.5 & \s, \s, \s, \s, ([1^{\ell-3},3],1) \\
    \hline
    \end{array}$$ 
    \end{table}

    \section{Almost Galois ramification}\label{sec:almost-Gal}
    
    A key ingredient in the proof of Theorem \ref{thm:main-wreath} is proving that an ``almost Galois/regular" behavior occurs. 
    For a covering $h_1:Y_1\ra Y$, this means that similarly to Galois coverings the ramification multiset $E_{h_1}(P)$ has equal entries with a bounded amount of exceptions.  We use the following proposition to guarantee such a property. This proposition is based on an  idea from Do--Zieve \cite{DZ}. 
    \begin{prop} \label{lem:DZ}
    For every integer $\alpha>0$, there exists a constant $L_{1,\alpha}$ depending only on $\alpha$ that satisfies the following property. 
    Let $h_1:Y_1\ra Y$ and $h_2:Y_2\ra Y$ be two coverings of degree $\ell\geq L_{1,\alpha}$, 
    and let $p:Z\ra Y$ be a minimal covering that factors through $h_1$ and $h_2$. 
    Assume $g_Z<\alpha\ell$ and $\deg p = \deg h_1\cdot \deg h_2$. Then for every point $P$ of $Y$ one of the following holds:
    \begin{enumerate}
    \item[$(a)$] there exists some $k\le 6$ such that, under each of the maps $h_1$ and $h_2$, the number of preimages of $P$ which have ramification index $k$ is at least
    $$\frac{\ell}{k} - 2(\alpha+1)(k+1) - \frac{2}{3}(\alpha+1)(k^2-1);$$
    \item[$(b)$] for each $k\le 6$, under each of the maps $h_1$ and $h_2$ the number of preimages of $P$
    which have ramification index $k$ is at most $2(\alpha+1)(k+1)$.
    \end{enumerate}
    \end{prop}
    
    \begin{proof}
    Write $E_{h_1}(P)=[e_1,\ldots,e_n]$ and $E_{h_2}(P)=[f_1,\ldots,f_m]$. 
    Since $g_Y\geq 0$ and $g_Z<\alpha\ell$,  the Riemann--Hurwitz formula for the natural projection $p_2:Z\ra Y_2$, Lemma~\ref{lem:abh} yield
    \begin{equation} \label{rh0}
    2\ell(\alpha+1)-2 \ge 2g_Z+2\ell-2\ge R_{p_2}(h_2^{-1}(p))
    = \sum_{i=1}^{\ell} \sum_{j=1}^m \bigl(e_i-(e_i,f_j)\bigr).
    \end{equation}
    
    We prove the assertion by induction on $k$, with the base case $k=0$ being vacuous.  So fix a positive integer $k\le 6$,
    and assume that for any $0<k_0<k$ the number of copies of $k_0$ amongst the $e_i$'s and $f_j$'s is at most $2(\alpha+1)(k_0+1)$.
    Let $r$ be the number of $i$'s for which $e_i$ equals $k$, and let $R=\sum_{i: e_i<k} e_i$,
     and likewise let $s$ be the number of $j$'s for which $f_j$ equals $k$, and $S=\sum_{j: f_j<k} f_j$.  Then
    \[
    R \le \sum_{k_0=1}^{k-1} 2(\alpha+1)(k_0+1)k_0 = 2(\alpha+1)\Bigl(\frac{(k-1)k(2k-1)}6 + \frac{(k-1)k}2\Bigr)
    = 2(\alpha+1)\frac{k^3-k}3.
    \]
    Putting $d:=2(\alpha+1)(k^3-k)/3$, and considering the contribution to the right side of (\ref{rh0}) coming
    from the $s$ distinct $j$'s for which $f_j$ equals $k$, we get
    \begin{equation} \label{rhs}
    2(\alpha+1)\ell > s\sum_{i=1}^{n} (e_i-(k,e_i)) \ge s\sum_{e_i> k} (e_i-k) \ge s\sum_{e_i>k} \frac{e_i}{k+1} =
    s\frac{\ell-kr-R}{k+1}\ge s\frac{\ell-kr-d}{k+1}.
    \end{equation}
    Likewise we get $S\le d$ and
    \begin{equation} \label{rhr}
    2(\alpha+1)\ell > r\frac{\ell-ks-S}{k+1} \ge r\frac{\ell-ks-d}{k+1}.
    \end{equation}
First assume  $s \le (\ell-d)/(2k)$.  Then (\ref{rhr}) gives
    \[
    2(\alpha+1)\ell > r\frac{\ell-d}{2(k+1)},
    \text{ so that }
    r<\frac{4(\alpha+1)\ell(k+1)}{\ell-d}
    \]
    and thus 
    $r<8(\alpha+1)(k+1)$, for every $\ell>280(\alpha+1)\geq 4(\alpha+1)(k^3-k)/3 = 2d$. 
    Substituting this into (\ref{rhs}) gives
    \[
    2(\alpha+1)\ell > s\frac{\ell-kr-d}{k+1}>s\frac{\ell-8(\alpha+1)k(k+1)-d}{k+1},
    \]
    so that
    \begin{equation}\label{equ:rounding}
    s<\frac{2(\alpha+1)\ell(k+1)}{\ell-8(\alpha+1)k(k+1)-d},
    \end{equation}
    for $\ell>476(\alpha+1)\geq 8(\alpha+1)k(k+1)+d$. 
    Since $s$ is an integer, there exists a constant $L'_{1,\alpha}$ such that 
    for every $\ell\geq L'_{1,\alpha}$, \eqref{equ:rounding}  forces 
    $s \le 2(\alpha+1)(k+1).$
    Since we have shown that $r<8(\alpha+1)(k+1)$, it follows that $r\le (\ell-d)/(2k)$, for $\ell>812(\alpha+1)\geq 16(\alpha+1)k(k+1)+ d$
    Interchanging the roles of $r$ and $s$, the same argument yields $r\le 2(\alpha+1)(k+1)$.
    
    Next assume that $s>(\ell-d)/(2k)$.
    Then (\ref{rhs}) gives
    \[
    2(\alpha+1)\ell > s\frac{\ell-kr-R}{k+1} \ge \frac{\ell-d}{2k}\cdot\frac{\ell-kr-R}{k+1},
    \]
    so that
    \[
    \ell-kr-R<\frac{4(\alpha+1)\ell k(k+1)}{\ell-d},
    \]
    for $\ell>140(\alpha+1)\geq d$. Hence
    $\ell-kr-R<8(\alpha+1)k(k+1)$
    for $\ell>280(\alpha+1)\geq 2d$, so that
    \[
    r > \frac{\ell-R-8(\alpha+1)k(k+1)}{k} \ge \frac{\ell-d-8(\alpha+1)k(k+1)}k.
    \]
    Substituting this into (\ref{rhr}) gives
    \[
    2(\alpha+1)\ell > r\frac{\ell-ks-S}{k+1} \ge \frac{\ell-d-8(\alpha+1)k(k+1)}{k}\cdot\frac{\ell-ks-S}{k+1},
    \]
    so that
    \begin{equation}\label{equ:bound-other}
    \ell-ks-S < \frac{2(\alpha+1)\ell k(k+1)}{\ell-d-8(\alpha+1)k(k+1)},
    \end{equation}
    for $\ell>476(\alpha+1)\geq 8(\alpha+1)k(k+1)+d$. As the left hand side of \eqref{equ:bound-other} is an integer,  
    \[
    \ell-ks-S \le 2(\alpha+1)k(k+1)\text{ for $\ell\geq L'_{1,\alpha}$,}
    \]
    so that
    \[
    s \ge \frac{\ell-S-2(\alpha+1)k(k+1)}k \ge \frac{\ell-d-2(\alpha+1)k(k+1)}k.
    \]
    Since we showed that $r>(\ell-d-8(\alpha+1)k(k+1))/k$, in particular $r>(\ell-d)/(2k)$, for $\ell>812(\alpha+1)\geq 16(\alpha+1)k(k+1)+ d$. Interchanging
    the roles of $r$ and $s$, we obtain $r\ge (\ell-d-2(\alpha+1)k(k+1))/k$ for every $\ell\geq L_{1,\alpha}:=\max\{L'_{1,\alpha}, 812(\alpha+1)\}$. 
    \end{proof}
    \begin{defn}\label{def:almost-galois-type}
    Fix an integer $\alpha>0$. Let $h_1:Y_1\ra Y$ be  a covering of degree $\ell$.
    We say that a point $P$ of $Y$ is of {\it almost Galois type $m_{h_1}(P):=k$} if $\ell \geq L_{1,\alpha}$ and $1\leq k\leq 6$ is the minimal integer for which $E_{h_1}(P)$ contains at least
    $\ell/k - 2(\alpha+1)(k+1) - 2(\alpha+1)(k^2-1)/3$ entries that are equal to $k$, as in case (a) of Proposition \ref{lem:DZ}. 
    We say that $P$ is of {\it almost Galois type} $m_{h_1}(P):=\infty$ if   $E_{h_1}(P)$ has at most $2(\alpha+1)(k+1)$ entries equal to $k$  (as in  Proposition \ref{lem:DZ}.(b)), and $P$ is not of almost Galois type $k$ for $1\leq k\leq 6$, and $\ell \geq L_{1,\alpha}$. 
    
    For a point $P$ of almost Galois type $m_{h_1}(P)=k<\infty$, we define the {\it error} $\eps=\eps_{\alpha,k}$ of $h_1$ at $P$  to be $\eps:= 2(\alpha+1)k(k+1 + (k^2-1)/3)$, so that $\eps$ bounds the sum of entries in $E_{h_1}(P)$ which are different from $k$. 
    \end{defn}
    \begin{rem}\label{lem:same-m_P} 
    (1) Under its assumptions, Proposition \ref{lem:DZ} also shows that $m_{h_1}(P) = m_{h_2}(P)$ for every point $P$ of $Y$. 
    
    (2) We note that if $P$ is a point of almost Galois type, the value $m_{h_1}(P)$ is independent of the constant $\alpha$. Indeed,  if $m_P:=m_{h_1}(P)$ is finite, the requirement
    $\ell\geq L_{1,\alpha}\geq 812(\alpha+1)$ implies that $\ell> 2\eps_{\alpha,m_P}$ and hence that the sum of entries equal to $m_P$ in $E_{h_1}(P)$ is at least $\ell/2$. If $m_P=\infty$, the same requirement implies that the sum of entries which are at most $6$ in $E_{h_1}(P)$ is less than $\ell/2$. Thus $P$ cannot be a point of distinct almost Galois types for two values of $\alpha$. 
    \end{rem}

    \begin{cor} \label{cor:DZ}\label{cor:DZ1}
    For every integer $\alpha>0$, there exists a constant $L_{2,\alpha}\geq L_{1,\alpha}$ satisfying the following property. Let $h_i:Y_i\ra Y$, $i=1,2$ be coverings of degree $\ell\geq L_{2,\alpha}$, and let $p:Z\ra Y$ be a minimal covering factoring through $h_1$ and $h_2$. Assume $g_Z<\alpha\ell$ and $\deg p=\deg h_1\cdot\deg h_2$, so that
     every point $P$ of $Y$ is of almost Galois type by Proposition \ref{lem:DZ}, for $i=1,2$. 
    Then all of the following hold:
\begin{enumerate}
\item $g_Y\leq 1$;
\item the multiset 
    $M_{h_1}:=\{m_{h_1}(P)>1\,|\,P\in Y(\K)\} $
    is empty if $g_Y=1$;
\item $M_{h_1}$ is one of the following if $g_Y=0$:
    \begin{itemize}
    \item $\infty,\infty$
    \item $\infty,2,2$
    \item $3,3,3$
    \item $2,3,6$
    \item $2,4,4$
    \item $2,2,2,2$;
    \end{itemize}
\item the number of branch points $P$ of $Y$ with $m_{h_1}(P)=1$ is bounded by $2(\alpha+1)$;
\item there exists an integer $N_\alpha$, depending only on $\alpha$, such  the number of points of $Y_1$ lying over $P$ is at most $N_\alpha$  for every  $P\in Y(\K)$ with $m_{h_1}(P)=\infty$;
\item $g_{Y_i}\leq \alpha+1$ for $i\in I$. 
\end{enumerate}
    \end{cor}

\begin{rem}\label{rem:almost-Galois-NZ} The proof relies on the proof of \cite[Corollary 7.3]{DZ}. Namely, repeating this proof gives a constant $M_\alpha$, depending only on $\alpha$, which satisfies the following property. If $h_1:Y_1\ra Y$ is a covering of degree $\ell\geq M_\alpha$ such that every $P\in Y(\K)$ is of almost Galois type $m_{h_1}(P)$ with constant $\alpha$, and if (4) and (6) hold, then (1)-(3) and (5) hold as well.
    \end{rem}
    \begin{proof}
    The Riemann--Hurwitz formula for the natural projection $p_2:Z\ra Y_2$ gives:
    \begin{equation}\label{equ:DZ-h2}
    2(\alpha+1)\ell-2 > 2(g_Z-1+\ell) = \sum_{P\in Y(\K)} R_{p_2}(h_1^{-1}(P)).
    \end{equation} 
    For a fixed point $P$ of $Y$, 
    we have
    \[
    R_{p_2}(h_1^{-1}(P)) = \sum_{e\in E_{h_1}(P),\, f\in E_{h_2}(P)} (e-(e,f)),
    \]
    by Abhyankar's lemma \ref{lem:abh}. We next bound $\vert  R_{p_2}(h_1^{-1}(P)) - \ell R_{h_1}(P)\vert$ in terms of $\alpha$, for a point $P$ with $m_{h_1}(P)=1$.
    Let $m_P:=m_{h_1}(P)$ and let $\eps_P:=\eps(\alpha,m_P)$ be the error at $P$, for a point $P$ of $Y$.
    Since the number of entries of $E_{h_1}(P)$ which equal $m_P$ is at least $(\ell-\eps_P)/m_P$ for every point $P$, for $m_P=1$ the sum
    \[
    \sum_{\substack{e\in E_{h_1}(P),\, f\in E_{h_2}(P) \\ f=1}}(e-(e,f)) 
    = \abs{f\in E_{h_1}(P)\suchthat f=1}\sum_{e\in E_{h_1}(P)}(e-1) 
    \]
    is at least $(\ell-\eps_P)\sum_{e\in E_{h_1}(P)}(e-1)$ and at most $\ell\sum_{e\in E_{h_1}(P)}(e-1)$.
    Since the sum and hence number of entries different from $m_P$ is at most $\eps_P$, one also has 
    \[
    \sum_{\substack{e\in E_{h_1}(P),\, f\in E_{h_2}(P) \\ e\neq 1,f\neq 1}}(e-(e,f))<\eps_P(\eps_P-1),
    \]
    for every point $P$ of $Y$.
    Combining these inequalities for a point $P$ with $m_P=1$, we have:
    \begin{equation}\label{equ:h2h1}
    \begin{split}
    \bigl\vert  R_{p_2}(h_1^{-1}(P)) - \ell R_{h_1}(P) \bigr\vert & =
    \bigl\vert \sum_{e\in E_{h_1}(P),\, f\in E_{h_2}(P)}\bigl(e-(e,f)\bigr) - \ell\sum_{e\in E_{h_1}(P)}(e-1)\bigr\vert  \\
      & 
      \leq\bigl\vert \sum_{\substack{e\in E_{h_1}(P),\, f\in E_{h_2}(P) \\ f= 1}}\bigl(e-(e,f)\bigr) - \ell\sum_{e\in E_{h_1}(P)}(e-1)\bigr\vert \\
      & 
      \qquad + \bigl\vert \sum_{\substack{e\in E_{h_1}(P),\, f\in E_{h_2}(P) \\ e\neq 1, f\neq 1}}\bigl(e-(e,f)\bigr) \bigr\vert \\
      & < \eps_P\sum_{e\in E_{h_1}(P)}(e-1) + \eps_P(\eps_P-1) \leq 2\eps_P(\eps_P-1) < \delta_\alpha. 
    \end{split}
    \end{equation}
    for some constant $\delta_\alpha>1$, depending only on $\alpha$. 
    In particular,  $R_{p_2}(h_1^{-1}(P)) \ge \ell-\delta_\alpha$  if $P$ ramifies under $h_1$.
    Hence \eqref{equ:DZ-h2} and \eqref{equ:h2h1} give 
    $
    2(\alpha+1)\ell>B(\ell-\delta_\alpha)$ where $B$ is the number of branch points of $h_1$ with $m_P=1$. For $\ell\geq \delta_\alpha(2\alpha+3)$, this implies that 
     $B\leq 2(\alpha+1)$,
     proving part (4). 
     Since $\ell(g_{Y_1}-1)\leq g_Z-1$ by the Riemann--Hurwitz formula and since by assumption $g_Z< \alpha\ell$, we also have $g_{Y_1}<\alpha+1$, proving (6). Parts (1)-(3) and (5) then follow by Remark \ref{rem:almost-Galois-NZ} for $\ell$ at least $L_{2,\alpha}:=\max\{\delta_\alpha(2\alpha+3),M_\alpha\}$. 
    \end{proof}

    \begin{defn} \label{def:almost-Gal}
    Fix  integers $\alpha>0$ and $t\geq 2$.  Let $h_i:Y_i\ra Y$, $i=1,\ldots,t$ be coverings of degree $\ell$. 
    We call the sum of errors  $\eps=\eps_\alpha:=2(\alpha+1)\eps_{\alpha,1} + \eps_{\alpha,2}+\eps_{\alpha,3}+\eps_{\alpha,6}$ the {\it total error} and the constant $N=N_\alpha$, from Corollary \ref{cor:DZ},  the {\it entry bound}. 
    \end{defn}
    Note that for coverings $h_i$, $i=1,\ldots,t$, as in the definition, and sufficiently large $\ell$, 
    Proposition \ref{lem:DZ} implies that every point $P$ is of almost Galois type under $h_i$, $i=1,\ldots, t$. Moreover, all of the consequences of Corollary \ref{cor:DZ} hold, in particular giving the possibilities for $m_{h_i}(P)$, $P\in Y(\K)$, for every $i=1,\ldots,t$. 
    The total error $\eps$ is chosen here to bound the sum of all errors over all branch points $P\in Y(\K)$ of $h_1$ with finite $m_{h_1}(P)$, when $M_{h_1}$ is as in Corollary \ref{cor:DZ}. Note that indeed the proof of the corollary shows that the number of branch points $P$ with $m_{h_1}(P)=1$ is at most $2(\alpha+1)$. 
    Also note that for coverings $h_i:Y_i\ra Y$, $i=1,\ldots,t$ with irreducible fiber products $Y_{i}\#_Y Y_j$, 
 $i\neq j$ of genus  $<\alpha\ell$, the values $m_{h_i}(P)$ are independent of the constant $\alpha$, and independent of $i$ for every point $P$ of $Y$, by Remark \ref{lem:same-m_P}.

    \begin{lem}\label{lem:normal-closure}
    Let $t\geq 2$, and let $h_i:Y_i\ra Y$, $i=1,\ldots,t$ be coverings with irreducible fiber products $Y_i\#_Y Y_j$, $i\neq j$.  Suppose that $h_i:Y_i\ra Y$, $i=1,\ldots,t$ admit irreducible fiber products $Y_{i}\#_Y Y_j$, $i\neq j$ of genus  $<\alpha\ell$. Assume that all entries $r\in E_{h_j}(P)$ divide $m_{h_j}(P)$ for all points $P$ of $Y$ with finite $m_{h_j}(P)$, for some $1\leq j\leq t$. Then $h_j$, $j=1,\ldots,t$ are of genus $\leq 1$ and hence their monodromy groups are solvable of derived length at most $2$. 
    \end{lem}
    \begin{proof}
    Let $\tilde h_j:\tilde Y_j\ra Y$ be the Galois closure of $h_j$. 
    We claim that either $g_{\tilde Y_j} = 1$ or $g_{\tilde Y_j}=0$ and the monodromy group $G_j$  of $h_j$ is cyclic or dihedral. The lemma then follows from the claim since the monodromy group of a Galois covering $\tilde Y_j\ra  Y$ of genus $g_{\tilde Y_j}=1$ is  known to be a subgroup of a semidirect product of two abelian groups  \cite[Lemma 9.6]{NZ}. 
    
    By Corollary \ref{cor:DZ} we have $g_Y\leq 1$. 
    Moreover in the case $g_Y=1$, the corollary implies $m_{h_j}(P)=1$ for all points $P$ of $Y$. Thus in this case, $h_j$ is unramified by our assumption. 
    It follows that $h_j$ is a morphism of elliptic curves and hence 
    is Galois by \cite[Theorem III.4.8]{Sil}, so that $\tilde Y_j= Y_j$ is of genus $1$, proving the claim in this case.
    Henceforth assume~$g_Y=0$.
      
    By Corollary \ref{cor:DZ}, the multiset $M_{h_j}:=\{m_{h_j}(P)\,|\,P\in Y(\K), m_{h_j}(P)> 1\}$ is one of $\{\infty,\infty\}, \{\infty,2,2\}, \{2,2,2,2\}, \{3,3,3\}, \{2,4,4\}, \{2,3,6\}$. Put $u:=\deg \tilde h_j$. Since the ramification index of $\tilde h_j$ over a point $P$ is the least common multiple of all $r\in E_{h_j}(P)$ by Abhyankar's lemma, see Remark \ref{rem:abh}, our assumption implies that in the latter four cases, the ramification of $\tilde h_j$ is $[2^{u/2}]$ four times, or $[3^{u/3}]$ three times, or $[2^{u/2}], [4^{u/4}], [4^{u/4}]$, or $[2^{u/2}],[3^{u/3}], [6^{u/6}]$. 
    Thus, the Riemann--Hurwitz formula for $\tilde h_j$ implies that $g_{\tilde Y_j}=1$ in the latter four possibilities for $M_{h_j}$, proving the claim in these cases.
    
    If $M_{h_j} = \{\infty,\infty\}$, our assumption implies that $h_j$ has no branch points $P$ 
    with finite $m_{h_j}(P)$, and hence only two branch points in total. The only possibility for the ramification of a covering $h_j$ which satisfies the Riemann--Hurwitz formula and has at most two branch points is $[\ell],[\ell]$, where $\ell := \deg h_j$. In this case, as in Section \ref{sec:RET}, $G_j$ is generated by two elements with product $1$, and hence $G_j\cong C_\ell$. 
    
    It remains to consider the case where $m_{h_j}(P_0)=\infty$, $m_{h_j}(P_1) = m_{h_j}(P_2) = 2$ for three points $P_0,P_1,P_2$ of $Y$, 
    and $m_{h_j}(P)=1$ for every other point $P\in Y(\K)$. 
    By assumption $r=1$ for all $r\in E_{h_j}(P)$, $P\neq P_j$, $j=0,1,2$ and $r=1$ or $2$ for $r\in E_{h_j}(P_j)$, $j=1,2$. 
    Since $h_j$ has such ramification, as in Section \ref{sec:RET} there exists a product-$1$ tuple $x_1,x_2,x_3$ for $G_j$ where $x_2$ and $x_3$ are elements of order $2$. Since $G_j$ is generated by the order two elements $x_2$ and $x_3$, the group $G_j$ must be cyclic or dihedral, proving the claim.
    \end{proof}

    \section{The genus of a product type covering}\label{sec:genus}
    Fix $t\geq 2$, and consider indecomposable coverings $f:X\ra X_0$ with Galois closure $\tilde X$, and monodromy group $G\leq S_\Delta \wr S_I$ 
    such that $\abs I=t$. As in Setup \S\ref{sec:setup}, 
    Let $H\leq G$ be a point stabilizer, $K:=G\cap S_\Delta^I$,  assume $G=H\cdot K$, and put $\ell:=\abs{\Delta}$ and $m:=[G:K]$. 
    $$
    \xymatrix{
    \tilde X \ar[d] \ar[rd] & \\
     Z =\tilde X/(H\cap K)  \ar[r]^>>>>>>{\pi} \ar[d] & X=\tilde X/H \ar[d]^{f} \\
     \tilde X/K \ar[r] & X_0
    }
    $$
    The following proposition bounds the 
    Riemann--Hurwitz contribution $R_\pi$ of the natural projection $\pi:Z\ra X$, where $Z:=\tilde X/(H\cap K)$.
    \begin{prop}\label{prop:genus-analysis}\label{lem:V-orbits}\label{lem:gZ-gX-t=2} \label{lem:genus-Z-X}  
    There exists a constant $E_0=E_{0,t}>0$, depending only on $t$, satisfying the following property. 
     Let $f:X\ra X_0$ and $\pi:Z\ra X$ be coverings as above, and $x$ a branch cycle of $f$ over  a point $P$ of $X_0$. Let $O$ be the set of orbits of $x$ on $I$, and $E\subseteq O$ the subset of even length orbits. Let $y=a\sigma \in S_\Delta\wr S_I$, $a\in S_\Delta^I$, $\sigma\in S_I$ be a reduced form of  $x$  with representatives $\iota_\theta,\theta\in O$. 
    \begin{enumerate} 
    \item Then $R_\pi(f^{-1}(P)) < m\ell^{t-1}/2 + E_0\ell^{t-2}$;
    \item If $\abs E\neq 1$ or $E=\{\mu\}$ with $\abs \mu>2$, then $R_\pi(f^{-1}(P)) < E_0\ell^{t-2}$;
    \item If $t=2$ and $\sigma\neq 1$, then 
    $R_{\pi}(f^{-1}(P))$ is the number of odd orbits of $a(\iota_I)$.
\end{enumerate}
    \end{prop}

    \begin{proof}
     {\bf Step I:} {\it Expressing $R_\pi(f^{-1}(P))$ in terms of orbits of $y$.}
    Let $M$ be the $G$-set $K\backslash G$ so that $\abs M=[G:K]=m$. 
    Since $H\cap K$ (resp.~$H$) is a point stabilizer in the action of $G$ on $\Delta^I\times M$ (resp.~$\Delta^I$), the points 
    of $Z$ (resp.~of $X$) lying over $P$ are in one to one correspondence with the orbits of $x$ acting on $(H\cap K)\backslash G\cong \Delta^I\times M$ (resp., $H\backslash G\cong \Delta^I$) as $G$-sets. 
    Note that since $G=H K$, we have $\deg\pi=[H:H\cap K] = [G:K]=m$. 
      
     By the chain rule  \eqref{equ:chain},    $R_{f\circ \pi}(P) = mR_f(P) + R_\pi(f^{-1}(P))$. Hence, by the previous paragraph 
    \begin{equation}\label{equ:grr} 
    \begin{split}
    R_{\pi}(f^{-1}(P)) & = R_{f\circ \pi}(P) -mR_f(P) = m\ell^t - \abs{\Orb_{\Delta^I\times M}(x)} - m (\ell^t - \abs{\Orb_{\Delta^I}(x)}) \\
    & = m\cdot \abs{\Orb_{\Delta^I}(x)} - \abs{\Orb_{\Delta^I\times M}(x)}.
    \end{split}
    \end{equation}
    
    Let $\kappa\in S_\Delta^I$ be such that $x^\kappa = y$. Note that $y$ acts on $\hat M:=S_\Delta^I\backslash (S_\Delta^I\cdot G)$. 
    We claim that  
    \begin{equation}\label{equ:Rpi-primitive} 
    R_\pi(f^{-1}(P)) = m\cdot \abs{\Orb_{\Delta^I}(y)} - \abs{\Orb_{\Delta^I\times \hat M}(y)}.
    \end{equation} 
    Let $\hat G:=S_\Delta^I\cdot G$, and let $\hat H$ be a point stabilizer of $\hat G$ such that $\hat H\cap G=H$. 
    As $x$ and $y$ are conjugate in $\hat G$, we have
    $\abs{\Orb_{\hat H\backslash \hat G}(x)} = \abs{\Orb_{\hat H\backslash \hat G}(y)},$ 
    as well as
    $\abs{\Orb_{(\hat H\backslash \hat G) \times \hat M}(x)}=\abs{\Orb_{(\hat H\backslash \hat G) \times \hat M}(y)}$. 
    Since $\hat H\backslash \hat G\cong  \Delta^I$ and $(\hat H\backslash \hat G)\times \hat M\cong  \Delta^I\times \hat M$ as $\hat G$-sets,  and $\hat M\cong M$ as $G$-sets,
    these equalities give
    $$\abs{\Orb_{\Delta^I}(x)} = \abs{\Orb_{\Delta^I}(y)}\text{ and }\abs{\Orb_{\Delta^I\times M}(x)} = \abs{\Orb_{\Delta^I\times \hat M}(x)} =  \abs{\Orb_{\Delta^I\times \hat M}(y)}. $$ 
    Plugging these equalities into \eqref{equ:grr}, gives the claim \eqref{equ:Rpi-primitive}. 

    {\bf Step II:} {\it Computing the length of an orbit in the action of $y$ on  $\Delta^I\times \hat M$. }
    Consider the the action of $y$ on elements $\delta\in R = \prod_{i\in I}R_i$, for orbits $R_i$ of $a(\iota_\theta)$, for all $i\in \theta$, $\theta\in O$. 
    Denote by $\mathcal R=\mathcal R_y$ the set of all
     such products  $\prod_{i\in I}R_i$. 

    Throughout the proof,  $R_i\subseteq \Delta$ denotes the $i$-th component of $R\in\mathcal R$.
    Also put  $r_i:=\abs{R_i}$, for every $R\in \mathcal R, i\in I$. 
    Moreover for $\theta\in O$, $\delta\in R$ and $R\in\mathcal R$, denote by $\delta_\theta$  its restriction to $\theta$, that is, the tuple $(\delta(i))_{i\in \theta}\in \prod_{i\in\theta}R_i$, let $R_\theta:=\{\delta_\theta\suchthat \delta\in R\}$ denote the restrictions of elements in $R$ to $\theta$,
    and let $r_\theta:=\lcm_{i\in\theta}(r_i)$ be the length of the orbits of $a(\iota_\theta)$  on $R_\theta$. 
     
    First note that as $\hat M=S_\Delta^I\backslash (S_\Delta^I\cdot G)$, the action of $y$ on $\hat M$ is through its image $\sigma$. As the latter acts regularly, the orbits of $y$ on $\hat M$ are of length $e:=\abs{\langle\sigma\rangle}=\lcm_{\theta\in O}(\abs\theta)$. 
    Fix $R\in\mathcal R$ and $\theta \in O$. 
    Let $\hat k_\theta$ denote the length of the orbit of $(\delta_\theta,m) \in R_\theta\times \hat M$ under~$y$.  
    Since $\abs{\theta}$ divides $e$, it also divides $\hat k_\theta$. Write
    $z:=aa^{\sigma^{-1}}\cdots a^{\sigma^{1-\abs \theta}}$, so that $y^{\hat k_\theta} = (z\sigma^{\abs\theta})^{\hat k_\theta/\abs \theta}$. 
    It follows that $\hat k_\theta$ is the smallest multiple of $\abs{\theta}$ for which
    $y^{\hat k_\theta} = (z\sigma^{\abs\theta})^{\hat k_\theta/\abs \theta}$ acts as the identity on $R_\theta$, or equivalently $z^{\hat k_\theta/\abs \theta}$ acts as the identity on $R_\theta$.
    Since 
    $z(i)=a(\iota_\theta)$ 
    for $i\in \theta$, 
    and each orbit of $a(\iota_\theta)$ on $R_\theta$ has length $r_\theta$, 
    we see that the tuple $z_\theta = (z(i))_{i\in \theta} \in R_\theta$ is of order $r_\theta$.
    Hence $\hat k_\theta =  \lcm(e,\abs{\theta}\cdot r_\theta)$. 
    
    The length $\hat k$ of the orbit of  $(\delta,m)\in R\times \hat M$ is the least common multiple of the lengths $\hat k_\theta$ of the orbits of $(\delta_\theta,m), \theta\in O$. As $e=\lcm_{\theta\in O}(\abs \theta)$, we get 
    \[
    \hat k=\lcm_{\theta\in O}(\hat k_\theta) =  \lcm_{\theta\in O}(\lcm(e,\abs \theta\cdot r_\theta)) =\lcm_{\theta\in O}(\abs \theta \cdot r_\theta).
    \]

     {\bf Step III:}
    {\it Describing the length of an  orbit on $\Delta^\theta$, for  $\theta\in O$. }
    Fix $R\in\mathcal R$ and $\theta\in O$. 
     Let $k=k_{\theta}$ 
     denote the length of some orbit of $y$ on $R_\theta$. 
    Since $y^{\abs \theta r_\theta}$ acts trivially on $R_\theta$ by step II, we have $k\divides \abs \theta \cdot r_\theta$. As in step II, the orbits of the element $y^{\abs\theta} = z\sigma^{\abs\theta}$ on $R_\theta$ are of length $r_\theta$, and hence $k= \abs{\theta}\cdot r_\theta/v_\theta$ for some integer $v_\theta$ which divides $\abs{\theta}$ and is coprime\footnote{This statement follows from the following elementary fact. If $x^u$ is an element of order $v$, then $x$ is an element of order $uv/w$ for some $w$ dividing $u$ which is prime to $v$.} to $r_\theta$. 

    The length of the orbit of $\delta\in R_\theta$ is then $k$ if and only if $k$ is minimal positive integer
    such that $\delta(i) = \delta^{y^{k}}(i)=\delta^{\sigma^kb_k}(i)$ for all $i\in\theta$, where  $b_k:=\sigma^{-k}y^{k}\in S_\Delta^\theta$. Equivalently one has $\delta^{y^kb_k^{-1}}(i)=\delta^{b_k^{-1}}(i)$,   and hence one has:
    \begin{equation}\label{equ:determined} 
     \delta(i^{\sigma^{-k}})=\delta^{\sigma^k}(i)=\delta^{y^kb_k^{-1}}(i)=\delta^{b_k^{-1}}(i)=\delta(i)^{b_k^{-1}(i)}\text{ for all $i\in\theta$. } 
    \end{equation}
     
    Since $(r_\theta,v_\theta)=1$, and $k=\abs{\theta}r_\theta/v_\theta$, and since $\sigma^{\abs{\theta}/v_\theta}$ has orbits of length $v_\theta$ on $\theta$, the elements $ \sigma^{-k}$ and $\sigma^{\abs\theta/v_\theta}$ have the same orbits on $\theta$. 
  Fixing  $i_0\in \theta$ with orbit $\theta'$ under $\sigma^{\abs\theta/v_\theta}$, and $\alpha\in R_{i_0}$, it follows that the values on  $\theta'$ of  all $\delta\in R_\theta$, whose orbit is of length $k$ with $\delta(i_0)=\delta_0$, 
 are determined uniquely by \eqref{equ:determined}. 
    
     {\bf Step IV:} 
    {\it Estimating the number of elements in $\Delta^I$ with a fixed orbit length. }
    Let $V$ be the set of tuples $v=(v_\theta)_{\theta\in O}$, where $v_\theta \divides \abs\theta$ is a positive integer for all $\theta\in O$. 
    For $v\in V$, let $\Omega_v$ be the set of all $\delta\in \Delta^I$ such that the length of the orbit of $\delta_\theta$ is $\abs\theta \cdot r_\theta/v_\theta$, for all $\theta\in O$.  
    By the claim in step III, $\Delta^I$ is a (disjoint) union of the sets $\Omega_{v}$, $v\in V$. 
    
    Since an element $\delta\in \Omega_{v}$ satisfies \eqref{equ:determined} with $k=\abs\theta\cdot r_\theta/v_\theta$ for each $\theta\in O$, the values of $\delta$ on each orbit $\theta'\subseteq \theta$ of $\sigma^{\abs\theta/v_\theta}$ are determined by its value on a single representative $\iota_{\theta'}\in\theta'$ by Step III. Since there are $\abs\theta/v_\theta$ such orbits $\theta'$ and at most $\ell$ possible values $\delta(\iota_{\theta'})$ for each representative, we get the estimate
    $\abs{\Omega_v}  
    \leq \ell^{\sum_{\theta\in O}\abs\theta/v_\theta}$ for $v\in V$.

     {\bf Step V:} {\it Isolating the main term of $R_\pi(f^{-1}(P))$ and deriving parts (1)-(2). }
    For $v\in V$, let $\mathcal R_v$ denote the set of tuples $R\in\mathcal R$ such that $\Omega_v\cap R$ is nonempty. 
    
    The element $y$ acts on the set $\Omega_v\cap R$  and  its orbits in this action are all of length $k_{v,R} := \lcm_{\theta\in O}(\abs\theta\cdot r_\theta/v_\theta)$, for  $v\in V, R\in \mathcal R_v$. 
    Similarly by step II, the length of each orbit on $(\Omega_v\cap R) \times\hat M$ is $\hat k_{v,R} := \lcm_{\theta\in O}(\abs\theta\cdot r_\theta)$. In particular, $y$ has $\abs{\Omega_v\cap R}/k_{v,R}$ orbits on $\Omega_v\cap R$ and $\abs{\Omega_v\cap R}\cdot m/\hat k_{v,R}$ orbits on $(\Omega_v\cap R)\times \hat M$. 
    Thus \eqref{equ:Rpi-primitive} gives
    \begin{equation}\label{equ:estimateRpi} 
     R_\pi(f^{-1}(P))=\sum_{v\in V}\Bigl(m\frac{\abs{\Omega_v\cap R}}{k_{v,R}}-m\frac{\abs{\Omega_v\cap R}}{\hat k_{v,R}}\Bigr)=m\sum_{v\in V}T_v,
     \end{equation}
      where $T_v:=\sum_{R\in \mathcal R_v}\bigl(1/k_{v,R} - 1/\hat k_{v,R}\bigr)\abs{\Omega_v\cap R}$.
    Since $T_v< \abs{\Omega_v}$ and since $\abs{\Omega_v}\leq \ell^{\sum_{\theta\in O}\abs\theta/v_\theta}$ by step IV, 
    we have $T_v< \ell^{t-2}$ for every $v\in V$ satisfying $\sum_{\theta\in O}\abs\theta/v_\theta\leq t-2$.  
    Moreover, if $\sum_{\theta\in O}\abs\theta/v_\theta = t$ then $v_\theta=1$ for all $\theta \in O$, implying that $k_{v,R}=\hat k_{v,R}$ for all $R\in \mathcal R_v$, and hence $T_{v}=0$. 
    Let $U$ be the set of tuples $v\in V$ such that $\sum_{\theta\in O}\abs\theta/v_\theta = t-1$.
    Since $\abs V$ is bounded by a constant $E'_{0}=E'_{0,t}$ depending only on $t$, and $T_v\leq \ell^{t-2}$ for $v\in V\setminus U$,  
     \eqref{equ:estimateRpi} gives:
    \begin{equation}\label{equ:final?} R_\pi(f^{-1}(P)) < m\sum_{v\in U}T_v +mE'_0\ell^{t-2} \leq m\sum_{v\in U}T_v +E_0\ell^{t-2},
    \end{equation}
    where $E_0:=t!E'_0$, and the latter inequality follows since $m\leq t!$, as $G/K\leq S_t$.

    It remains to estimate $\abs{U}$. 
    Note that as $t=\sum_{\theta\in O}\abs\theta$, one has $v\in U$ if and only if $\sum_{\theta\in O}(\abs\theta - \abs
    \theta/v_\theta)=1$. Hence $v\in U$ if and only if $v_\theta=1$ for all but a single orbit $\mu\in O$ such that $v_\mu=2$ and $\abs\mu=2$. 
    In this case,  $(r_\mu,2)=(r_\mu,v_\mu)=1$ for all $R\in \mathcal R_v$ by step III. 
    As $k_{v,R} = \lcm_{\theta\in O}(r_\theta\cdot \abs\theta/v_\theta)$, and $\hat k_{v,R} = \lcm_{\theta\in O}(r_\theta\cdot \abs\theta)$   for $v\in U$, we have
    \begin{equation}\label{equ:kvR}
    k_{v,R}=\hat k_{v,R}\text{ if $ r_\theta\cdot \abs\theta$ is even for some $\theta\in O\setminus \{ \mu\}$, and $\hat k_{v,R}=2k_{v,R}$ otherwise},
    \end{equation}
    for all $R\in \mathcal R_v$. 
    In particular,  if $\abs E>1$ or $E=\{\mu\}$ with $\abs\mu\neq 2$, 
    then $T_v=0$ for all $v\in U$, so that \eqref{equ:final?} gives part (2). 
    Moreover, if $E=\{\mu\}$ with $\abs\mu=2$, we have $U=\{v\}$, where $v_\mu=2$ and $v_\theta=1$ for all $\theta\in O\setminus \{\mu\}$. Since $\abs{\Omega_v}\leq \ell^{t-1}$ by step IV, and $1/k_{v,R} - 1/\hat k_{v,R}\leq 1/2$ by \eqref{equ:kvR}, we deduce from \eqref{equ:final?} that
    $$ R_{\pi}(f^{-1}(P)) < mT_v + E_0\ell^{t-2}\leq \frac{m}{2}\sum_{R\in \mathcal R_v}\abs{\Omega\cap R} + E_0\ell^{t-2}
    \leq \frac{m\ell^{t-1}}{2} + E_0\ell^{t-2}, $$
    proving part (1). 
    
     {\bf Step VI:} {\it Proving part (3).}  Assume  $t=2$ and $\sigma\neq 1$.  
     We compute $R_\pi(f^{-1}(P))$ using \eqref{equ:estimateRpi}.
     Note that $I=\mu$ is the only orbit of $\sigma$. Write $I=\{1,2\}$ with $1=\iota_\mu$, and $y = (a(1),1)\sigma$ in the notation of Section \ref{sec:setup}. Note that $V=\{u,v\}$ where $v_\mu=2$ and $u_\mu=1$. 
     As in step V,  
     $T_u=0$, so that it remains to find $T_v$. 
     
    We first claim that for every $R\in\mathcal R_v$ and $\delta\in \Omega_v\cap R$, one has $R_1=R_2$ with odd $r_1$, and $\delta(2) = \delta(1)^{b^{-1}(1)}$ where $b=\sigma y^{r_1}$. 
     By step III, $r_\mu$ and hence $r_1$ are odd. 
    By  \eqref{equ:determined},  
    $\delta(2) = \delta(1)^{b^{-1}(1)}$, where 
     $b:=\sigma^{-r_\mu}y^{r_\mu} = \sigma y^{r_\mu}\in S_\Delta^I$. 
    Since $y=(a(1),1)\sigma$, one has $b(1)\in \langle a(1)\rangle$, and hence  $\delta(2)\in R_1$. Thus, $R_1=R_2$, and $r_\mu=r_1$ is odd,  proving the claim. In this case we have $k_{v,R}=r_1$ and $\hat k_{v,R}=2r_1$. 
     
    Conversely, we claim that for an odd length orbit $R_1$ of $a(1)$, and $\alpha\in R_1$, there exists a unique $\delta \in R_1^I\cap \Omega_v$ such that $\delta(1)=\alpha$. Indeed, letting $\delta:=(\alpha,\alpha^{b(1)})\in R_1^I$ where $b:=y^{r_1}\sigma$, a direct computation shows that the orbit of $\delta$ is of length $r_1$, and hence $\delta\in \Omega_v\cap R_1^I$. The uniqueness of $\delta$ follows from \eqref{equ:determined}, proving the claim.

    The two claims in this step imply that $\abs{\Omega_v\cap R}$ is $r_1$ if $R_1=R_2$ and $r_1$ is odd, and is $0$ otherwise. Since in the former case we got $k_{v,R}=r_1$ and $\hat k_{v,R}=2r_1$, we deduce that  $T_v$ is half the number of odd length orbits of $a(1)$. 
    Since  $R_\pi(f^{-1}(P))=mT_v=2T_v$ by \eqref{equ:estimateRpi}, this proves part (3). 
\end{proof}

We will need the following further bound on the Riemann--Hurwitz contribution $R_\pi$ when the entries $a(\iota_\theta)$ are of almost Galois type, where $y=a\sigma$ is in reduced form. 
For this we shall use the following definition of almost Galois permutations: 
\begin{defn}\label{def:almost-Gal-branch-cycles} We say that $x\in S_\ell$ is {\it a permutation of almost Galois type} $m<\infty$ (resp., $m=\infty$) with {\it error} at most $\eps$ (resp., {\it entry bound} $N$)
if the number of length $m$ orbits of $x$ is at least $(\ell-\eps)/m$ (resp., the number of orbits of $x$ is at most $N$). 
\end{defn}
Thus for $m<\infty$, the error $\eps$ bounds the sum of the lengths $r$ of orbits of $x$ with $r\neq m$.
    %
    \begin{cor}\label{cor:rpi-Galois}
    For every $\eps>0$,  there exists a constant $E_1=E_{1,\eps,t}$, depending only on $\eps$ and $t$ with the following property. 
     Let $f:X\ra X_0$ and $\pi:Z\ra X$ be coverings as above, and $y=a\sigma \in S_\Delta\wr S_I$ be a reduced form of a branch cycle of $f$ over a point $P$, where $a\in S_\Delta^I$, $\sigma\in S_I$. Let $\iota_\theta$ denote the representative of $y$ for an orbit $\theta\subseteq I$ of $\sigma$. If $a(\iota_\theta)$ 
    is of almost Galois type with error at most $\eps$ for every orbit $\theta\subseteq I$ of $\sigma$, and $m(a(\iota_\eta))$ is either infinite or even for some orbit $\eta\subseteq I$ of $\sigma$,
    then $R_\pi(f^{-1}(P)) < E_1\ell^{t-2}$.
    \end{cor}
    \begin{proof}
    Let $\eta\in O$ be an orbit for which $a(\iota_\eta)$ is of almost Galois type with error at most $\eps$ such that $m(a(\iota_\eta))$ is even or infinite. 
    We use the estimates and notation of the proof of Proposition \ref{prop:genus-analysis}.
    Let $O$ be the set of orbits of $\sigma$ on $I$.
    By choosing $E_1\geq E_0$ and applying Proposition \ref{prop:genus-analysis}.(3),  we reduce to the case where $\sigma$ has a single orbit $\mu\in O$ of even length, and this orbit is of length $\abs\mu=2$.
    As in Step V, $R_\pi(f^{-1}(P))$ breaks into the sum $m\sum_{v\in V}T_v$, and as in \eqref{equ:final?}, it suffices to estimate $T_v$ for the unique $v\in U$, that is, for the tuple $v=(v_\theta)_{\theta\in O}$  where  $v_\theta =1$ for all $\theta \in O\setminus\{\mu\}$ and  $v_\mu =2$. 

    As in Proposition \ref{prop:genus-analysis}, let $\mathcal R$ be the set of all   $R=\prod_{i\in I}R_i\subseteq\Delta^I$, where $R_i$ is an orbit $a(\iota_\theta)$ for all $i\in \theta$, $\theta\in O$, let and $r_i:=\#R_i$ and 
    $r_\theta := \lcm_{i\in \theta}(\#R_i)$ for $\theta\in O$.
    Also let  $\Omega_v$ be the set of all $\delta\in R$ such that the length of the orbit of $\delta_\theta:=(\delta(i))_{i\in\theta}$ is $\abs\theta \cdot r_\theta/v_\theta$, for all $\theta\in O$. Thus the length $k_{v,R}:=\lcm_{\theta\in O}(r_\theta\cdot \abs{\theta}/v_\theta)$ of an orbit of $\delta\in \Omega_v\cap R$  is smaller than the length $\hat k_{v,R} := \lcm_{\theta\in O}(r_\theta\cdot \abs{\theta})$ of orbits in $R\times (G/K)$. 
    By step III of the proposition, if  $\Omega_v\cap R\neq \emptyset$, 
    then  $r_\mu$ is odd. 
    Moreover by \eqref{equ:kvR}, $T_v=0$ if $r_i$ is even for some $i\in I\setminus \mu$.  
    Henceforce assume all $r_i,i\in I$ are odd, so that \eqref{equ:kvR} gives $\hat k_{v,R} = 2k_{v,R}\geq r_{\iota_\eta}$. Thus, 
    \begin{equation*}
    T_v = \sum_{R\in \mathcal R}\Bigl(\frac{1}{k_{v,R}} - \frac{1}{\hat k_{v,R}}\Bigr)\abs{\Omega_v\cap R} = \sum_{R\in \mathcal R}\frac{\abs{\Omega_v\cap R}}{\hat k_{v,R}} \leq \sum_{R\in \mathcal R}\frac{\abs{\Omega_v\cap R}}{r_{\iota_\eta}}. 
    \end{equation*}
    Write $\mu=\{\iota_\mu,\iota_\mu'\}$. 
    As in Step III, since the values of $\delta\in \Omega_v\cap R$ on $\mu$ are determined by its value on $\iota_\mu$ by \eqref{equ:determined}, 
    and hence the last inequality 
    gives
    \begin{equation}\label{equ:bound-almost-Galois} 
    T_v \leq \sum_{(R_i)_{i\in I\setminus\{\iota_{\mu}'\}}}\frac{r_{\iota_\mu}\prod_{i\in I\setminus\mu}r_i}{r_{\iota_\eta}}, 
    \end{equation}
    where $R_i$ runs through odd length orbits of $a(\iota_\theta)$ for every $i\in \theta\setminus\{\iota_{\mu}'\}$. 
    Cancelling out $r_{\iota_\eta}$ with $r_{\iota_\mu}$ if $\mu=\eta$ and with $r_i$ for $i=\iota_\eta\not\in\mu$ otherwise, \eqref{equ:bound-almost-Galois} gives: 
    \begin{equation}\label{equ:almost-final} 
    T_v \leq 
    \bigl\vert\{R_{\iota_\eta}\in \Orb_\Delta(a(\iota_\eta))\suchthat r_{\iota_\eta}\text{ is odd}\}\bigr\vert 
    \prod_{i\in I\setminus\{\iota_\mu',\iota_\eta\}}\sum_{R_i\in\Orb_\Delta(a(\iota_i))}r_i
    \end{equation}
    Since $a(\iota_\eta)$ is of even or infinite almost Galois type, there is an upper  bound  $E'_1=E'_{1,\eps,t}$, depending only on $\eps$ and $t$,  on the number of odd length orbits of $a(\iota_\eta)$. Thus, as the sum over orbits of an element in $S_\Delta$ is at most $\ell$, \eqref{equ:almost-final} gives:
    \begin{equation}\label{equ:beta1-bound}
    T_v\leq E'_1\prod_{i\in I\setminus\{\iota_\mu',\iota_\eta\}}\sum_{R_i\in\orb_\Delta(a(\iota_i))}r_i \leq E'_1\ell^{t-2}. 
    \end{equation}
    Setting $E_1:=E_{1,t}:=t!E'_1+E_0$ and noting that $m\leq t!$ as $G/K\leq S_t$, \eqref{equ:final?} and \eqref{equ:beta1-bound} give:
    $$ R_{\pi}(f^{-1}(P))< mT_v + E_0\ell^{t-2} \leq (t!E_1'+E_0)\ell^{t-2}= E_1\ell^{t-2}.  $$
    \end{proof}

    \section{Ramification of genus $\leq 1$ with a transposition}\label{sec:transp}
    Fix $t\geq 2$ and consider indecomposable coverings $f:X\ra X_0$ with monodromy group $G\leq S_\Delta \wr S_I$ with $\abs{I}=t$. Letting $K:=G\cap S_\Delta^I$ and $H$ a  point stabilizer, as in Section \ref{sec:genus} we assume $G=H\cdot K$. Assume further that the image $\oline G\cong G/K$ of $G$ in $S_I$ is transitive on $I$.  
    
    The following proposition gives an upper bound on the total Riemann--Hurwitz contribution of the natural projection $\pi:Z\ra X$, where $Z$ is the quotient by $H\cap K$. This bound is then used to prove that the ramification is of almost Galois type,  as described in Section \ref{sec:intro}. Let $Y$ be the quotient by $K$,  put $\ell:=\abs{\Delta}$ and $m:=[G:K]$, and let $E_0$ be the constant from Proposition \ref{prop:genus-analysis}. 
    \begin{prop}\label{prop:asymp}
    Assume $g_Y\leq 1$ and let $s$ be the number of branch points of $f$ whose branch cycle acts on $I$ as a transposition. 
    Let
    $R_\pi:=\sum_P R_\pi(f^{-1}(P))$ where $P$ runs over all points of $X_0$. 
    Then $R_\pi\leq sm\ell^{t-1}/2 + 4E_0\ell^{t-2}$ if $s>0$ and $t\geq 3$;  and $R_\pi\leq 2tE_0\ell^{t-2}$ if $s=0$ and $t\geq 3$;
    and $R_\pi\leq s\ell$ if $t=2$. Moreover, $s\leq 2$ if $t\geq 3$, and $s\leq 4$ if $t=2$. 
    \end{prop}
    The proof follows from Proposition \ref{prop:genus-analysis} and the following lemma: 
    
    \begin{lem}\label{lem:no-trans}
    Let $\oline \pi_0:\oline Y \ra \mP^1$ be a covering with Galois closure of genus at most~$1$. 
    Assume that the ramification type of $\oline\pi_0$ contains a multiset of the form $[2, k_1,\ldots, k_u]$ with odd $k_i$ for $i=1,\ldots, u$. 
    Then $k_i=1$ for $i=1,\ldots, u$, and the ramification type of $\oline\pi_0$ appears in Table~\ref{table:trans}. 
    \begin{table}
    $$\begin{array}{| c | c | c | c | c |}
    \hline
    \deg\oline\pi_0 & \text{Ramification for }\oline{\pi}_0 & \text{Ramification for }\pi_0 & \Mon(\oline{\pi}_0) & g_{Y}\\
    \hline
    2 & [2]^2 & [2]^2 & C_2 & 0 \\
    2 & [2]^4 & [2]^4 & C_2 & 1 \\
    3 & [3], [2, 1]^2 & [3^2], [2^3]^2 & D_6 & 0 \\
    3 & [2, 1]^4 & [2^3]^4 & D_6 & 1 \\
    4 & [4], [2^2], [2, 1^2]  & [4^2], [2^4]^3 & D_8 & 0 \\
    4 & [2, 1^2]^2, [2^2]^2 & [2^4]^4 & D_8 & 1\\
    4 & [4], [3, 1], [2, 1^2] & [4^6], [3^8], [2^{12}] &  S_4 & 0 \\
    6 & [6], [3, 3], [2, 1^4] & [6^4], [3^8], [2^{12}] & C_2 \times AGL(1,4) & 1 \\
    \hline
    \end{array}$$
    \caption{Ramification types containing $[2, 1^{t-2}]$, for coverings $\oline \pi_0$ 
     of $\mP^1$ with  Galois closure $Y$  of genus $g_Y\leq 1$.} \label{table:trans}
    \end{table}
    \end{lem}
    \begin{proof}[Proof of Proposition \ref{prop:asymp}] 
    Letting $\oline H$ be a point stabilizer in the action of $G$ on $I$, we let  $\pi_0:Y\ra X_0$ and $\oline\pi_0:\oline Y\ra X_0$ be the natural projections, where $\oline Y:=\tilde X/\oline H$. 
    Since  $G$ acts  transitively on $I$, the action of $\oline G:=\Mon(\oline\pi_0)$ on $G/\oline H$ is equivalent to its action on $I$. 
    Since $K$ is the kernel of this action,  $\oline G\cong G/K$, and hence $\pi_0$ is the Galois closure of $\oline\pi_0$. 
    In particular, $\pi_0$ and $\oline\pi_0$ have the same branch points by Remark \ref{rem:abh}.
    
     Let $S$ be the subset of branch points of $\oline\pi_0$ whose ramification type contains a tuple
    $[2,k_1,\ldots,k_r]$ for some odd $k_1,\ldots,k_r$, $r\in\mathbb N$. 
    By Lemma \ref{lem:no-trans}, the number of branch points of $\pi_0$ is at most~$4$, and $\abs S\leq 2$ for $t\geq 3$ (resp.~$\abs S\leq 4$ for $t=2$). 
    Moreover, Lemma \ref{lem:no-trans} shows that $S$ consists of points $p$ in $X_0$ such that $E_{\oline\pi_0}(p)= [2,1^{t-2}].$  
    Since $E_{\oline\pi_0}(p)$ coincides with the multiset of orbits of branch cycle over $p$, and the  action of $\oline G$ on $\oline H\backslash \oline G$ is equivalent to its action on $I$, it follows that $s=\abs S$. The desired bounds on $s$ therefore follow from the above bounds from Lemma \ref{lem:no-trans}.  
    
    By Abhyankar's Lemma \ref{lem:abh}, if $p$ is not a branch point of $\pi_0$, then $\pi$ is unramified over every preimage in $f^{-1}(p)$ and hence $R_{\pi}(f^{-1}(p))=0$. 
    For a branch point $p\not\in S$ of $\pi_0$, one has  $R_{\pi}(f^{-1}(p))  < E_0\ell^{t-2}$ if $t\geq 3$ by Proposition \ref{prop:genus-analysis}.(2). 
    For $p\in S$, one has  $R_\pi(p) \leq m\ell^{t-1}/2 + E_0\ell^{t-2}$ if $t\geq 3$, and $R_\pi(f^{-1}(p))\leq \ell$  if $t=2$, by Proposition \ref{prop:genus-analysis}, parts (1) and (3), resp. 
    Since $g_Y\leq 1$, the Riemann--Hurwitz formula implies that $\oline\pi_0$ and hence $\pi_0$ have at most $2t$ branch points. Thus, we have $R_\pi\leq 2tE_0\ell^{t-2}$ if $s=0$.
    Now assume $s\geq 1$, so that $\pi_0$ has at most four branch points by Lemma \ref{lem:no-trans}. The bounds on $R_\pi(f^{-1}(p))$ in this paragraph then give $R_\pi\leq ms\ell^{t-1}/2 + 4E_0\ell^{t-2}$ for $t\geq 3$ and $R_\pi\leq s\ell$ for $t=2$. 
    \end{proof}
    \begin{proof}[Proof of Lemma \ref{lem:no-trans}]
    
    Let $\pi_0:Y\ra \mP^1$ denote the Galois closure of $\oline\pi_0$, so that both maps have the same branch points $p_1,\ldots, p_r$. Assume $p_1,\ldots, p_{r-1}$ are ordered by decreasing Riemann--Hurwitz contributions 
    $R_{\oline \pi_0}(p_1)\geq \cdots \geq R_{\oline\pi_0}(p_{r-1}),$ and that $E_{\oline\pi_0}(p_r)=[2,k_1,\ldots, k_r]$ for odd $k_i$, $i=1,\ldots,r$.
    Let $\oline G$ be the monodromy group of $\oline\pi_0$ equipped with an action on a set $I$ of cardinality $t:=\deg\oline\pi_0$.
    
    We first claim that either $g_{\oline Y}=0$ or the ramification of $\oline\pi_0$ is $[2]$ four times. 
    Clearly,  $g_{\oline Y}\leq g_Y$ and $g_Y\leq 1$ by hypothesis.  
    If $t>2$  and $g_{\oline Y}=1$, then the natural projection $Y\ra \oline Y$ is unramified of degree $m/t$, and hence $[2^{m/t},k_1^{m/t},\ldots, k_r^{m/t}]$ appears in the ramification type of $\pi_0$, contradicting the fact that $\pi_0$ is Galois.
    If $t=2$ and $g_{\oline Y}=1$, then $\oline\pi_0$ is Galois, and the ramification of $\oline\pi_0$ is $[2]^4$, proving the claim. 
    
    Henceforth assume $g_{\oline Y}=0$. As  $\pi_0$ is the Galois closure of $\oline\pi_0$, Abhyankar's Lemma and Remark \ref{rem:abh} 
    imply that
    \begin{equation}\label{equ:lcm-cond}
    \text{ $e_{\pi_0}(p_i)$  is the least common multiple of entries in
    $E_{\oline{\pi}_0}(p_i)$, for $i=1,\ldots, r$.}
    \end{equation}
    
    The ramification and monodromy groups of Galois coverings $\pi_0:Y\ra\mP^1$ with genus $g_Y\leq 1$ is well-known  and appears in Table~\ref{table:Galois}, see \cite[Proposition~2.4]{GT} and proof of \cite[Proposition 9.5]{NZ}. 
    \begin{table}
    $$
    \begin{array}{|l|l|}
    \hline
    \text{(A) $[m],[m]$ with $\oline G\cong~C_m$; } & 
    \text{(F) $[2^{m/2}]^4$; } \\ 
    \text{(B)~$[m/2,m/2],[2^{m/2}],[2^{m/2}]$ with $\oline G\cong D_{m}$; } & 
    \text{(G)~$[2^{m/2}], [4^{m/4}]^2$; } \\ 
    \text{(C) $[2^6], [3^4], [3^4]$ with $\oline G\cong A_4$; } & 
    \text{(H) $[2^{m/2}], [3^{m/3}], [6^{m/6}]$; } \\ 
    \text{(D) $[2^{12}], [3^{8}], [4^6]$ with $\oline G\cong S_4$;} & 
    \text{(I) $[3^{m/3}]^3$.} \\  
    \text{(E) $[2^{30}], [3^{20}], [5^{12}]$ with $\oline G\cong~A_5$;} & \\
    \hline
    \end{array}
    $$
    \caption{Ramification types for Galois coverings $Y\ra \mP^1$ of genus $g_Y\leq 1$.} \label{table:Galois}
    \end{table}
    Moreover, $\oline G$ is solvable in cases (F)-(I). 
    Running over each of the cases (A)-(I),  \eqref{equ:lcm-cond} implies  that  either $k_i=1$ for all $i$, or we are in case (H), and $e_{\pi_0}(p_r) = 6$, and  $k_i=1$ or $3$ for all $i$. We separate case (H) accordingly into case (H1) $e_{\pi_0}(p_r)=2$, and case (H2) $e_{\pi_0}(p_r)=6$. 
    In particular, as $e_{\pi_0}(p_r)$ is $2$ or $6$, 
     case (I) does not occur. 
    
    Since $R_{\pi_0}(p_i)\geq R_{\oline\pi_0}(p_i)\cdot m/t$ by the chain rule \eqref{equ:chain},  and $R_{\pi_0}(p_i) = m(1-1/e_{\pi_0}(p_i))$, we have
    \begin{equation*}\label{equ:ineq1}
    R_{\oline\pi_0}(p_i)\leq \frac{t}{m}R_{\pi_0}(p_i) = t\cdot \frac{e_{\pi_0}(p_i) - 1}{e_{\pi_0}(p_i)} \text{ for $i=1,\ldots, r-1$. }
    \end{equation*}
    As $g_{\oline Y}=0$ and $E_{\oline \pi_0}(p_r)=[2,1^{t-2}]$, combining the last inequality with the Riemann--Hurwitz formula for $\oline \pi_0$ gives in cases (A)-(G) and case (H1):
    \begin{equation}\label{equ:RH-oline-pi_0}
    \sum_{i=1}^{r-1}\frac{e_{\pi_0}(p_i) - 1}{e_{\pi_0}(p_i)}\geq \frac{1}{t}\sum_{i=1}^{r-1}R_{\oline\pi_0}(p_i) = \frac{2t-2 -R_{\oline\pi_0}(p_r)}{t} = 2-\frac{3}{t},
    \end{equation}
    and similarly in case (H2): 
    \begin{equation}\label{equ:caseH2}
    \sum_{i=1}^{r-1}\frac{e_{\pi_0}(p_i) - 1}{e_{\pi_0}(p_i)}\geq \frac{2t-2 -R_{\oline\pi_0}(p_r)}{t} \geq \frac{2t-2 -1 - (2/3)(t-2)}{t}.
    \end{equation}
    Applying \eqref{equ:RH-oline-pi_0} gives 
    $t=2$  in case (A);
    $t\leq 4$  in case (B); 
    $t<5$ in case (C);
    $t<6$ in case (D);
    $t<6$ in case (E);
    and $t\leq 6$ in cases (F), (G), and (H1); 
    where the last is an equality  if and only if $E_{\oline\pi_0}(p_i) = [e^{t/e}]$ for $e=e_{\pi_0}(p_i)$,  $i<r$.
    Similarly \eqref{equ:caseH2} gives $t\leq 10$ in case (H2). Moreover if $t=10$ in the latter case,  then $E_{\oline \pi_0}(p_r) = [2,3^{(t-2)/3}]$ and $E_{\oline\pi_0}(p_i)=[3^{t/3}]$ for some $i<r$. As both equalities cannot hold simultaneously with integral exponents, we can assume $t\leq 9$. 
    %
    
    Since $\oline G$ acts transitively and faithfully on $I$, these conditions force the regular action with $t=m=2$ if $\oline G$ is cyclic (case (A)); 
    the standard action of $D_{2t}$ on $t$ elements with $m=2t=6$ or $8$ if $\oline G$ is dihedral (case (B)); 
    the standard action on a set of size $t=4$ if $\oline G\cong A_4$ or $S_4$ (cases (C)-(D)); 
    and the standard action on a set of size $t=5$ if $\oline G\cong A_5$ (case (E)). Since the groups $A_4$ and $A_5$ do not contain transpositions in their standard action,
    we also have $\oline G\neq A_4,A_5$. 
    
    It follows that the ramification of $\pi_0$ is $[2]$ twice (Case (A)), or $[t,t],[2^{t}]$ twice with $t=3$ or $4$ (Case (B)), or $[2^{12}], [3^8], [4^6]$ (Case D), or as in cases (F)-(H).
    Given the bounds on $t$ in the above paragraph and these possibilities for the ramification of $\pi_0$, a straightforward computation determines which ramification data for $\oline\pi_0$ contains $[2,1^{t-2}]$ in cases (A)-(G) and (H1)
    (resp. $[2,3^{u_1},1^{v_1}]$ for some integers $u_1\geq 1$, $v_1\geq 0$ in case (H2)), satisfies  \eqref{equ:lcm-cond}, and the Riemann--Hurwitz formula for $\oline\pi_0$. These are the ramification types in Table~\ref{table:trans} and the two ramification types (F.N1) $[2,1^4], [2^3]$ three times, and (H2.N1) $[2^3],[3^2],[1,2,3]$. 
    Using multiplicativity of ramification indices, it is straightforward to check that every covering with ramification index F.N1 or H2.N1 is indecomposable. There is no indecomposable covering with ramification type F.N1 by Lemma \ref{lem:hurwitz1}.(a). 
    Since $A_4$ contains no element of order $6$ and $\oline G$ contains a branch cycle of order $6$ over $p_r$ in case H2, $\oline G$ is not a quotient $A_4$, and hence there is no indecomposable covering with ramification H2.N1 by Lemma \ref{lem:hurwitz1}.(c).
    
    For the ramification types with $g_Y=0$, the corresponding group $\oline G$  is read off cases (A)-(E) above. For the ramification types with $g_Y=1$ and $t\leq 4$, the subgroup  $\oline G\leq S_t$, $t=3,4$, is of derived length at most $2$, is transitive, and is not contained in $A_t$, forcing $\oline G \cong D_{2t}$.  
    In case the ramification of $\oline\pi_0$ is  $[6], [3, 3], [2, 1^4]$,  
    as in Section \ref{sec:RET} its monodromy group $\oline G\leq S_6$ is a transitive subgroup generated by a product-$1$ tuple $x_1,x_2,x_3\in S_6$ with cycle structures~$[6],[3,3],[2,1^4]$, respectively. A direct computation with Magma shows that for such tuples $\oline G = \langle x_1,x_2,x_3\rangle \cong C_2\times \AGL(1,4)$. 
    \end{proof}
    

    \section{Theorem \ref{thm:main-wreath}: reduction to almost Galois types and the case $t\geq 3$}
    \label{sec:G/K}
    In this section we deduce Theorem \ref{thm:main-wreath} from Theorem \ref{thm:main-Gal}, which assumes genus bounds that ensure ramification of almost Galois type by Section \ref{sec:almost-Gal}. 
    Moreover, we prove Theorem \ref{thm:main-Gal} for groups of product type $G\leq S_\Delta \wr S_I$ with $\abs I\geq 3$ in the absence of branch cycles that act on $I$ as a transposition. 
    
    We assume the setup of Section \ref{sec:setup}, so that to a covering $f:X\ra X_0$ with monodromy $G\leq S_\Delta \wr S_I$ of product type and Galois closure $\tilde X$, one associates the natural projections  $\pi:Z\ra X$, $h:Z\ra Y$, and $h_i:Y_i\ra Y$, where $Y:=\tilde X/K$, $Z:=\tilde X/(H\cap K)$, $H$ is a point stabilizer of $G$, $K:=G\cap S_\Delta^I$ and $K_i$ is a point stabilizer in the action of $K$ on the $i$-th coordinate of $\Delta^I$. Put $t:=\abs I$, $\ell:=\abs \Delta$, and $m:=[G:K]$. Recall that by Remark \ref{rem:setup} the fiber product of $h_i$ and $h_j$ is irreducible for distinct $i,j$ in $\{1,\ldots,t\}$. 
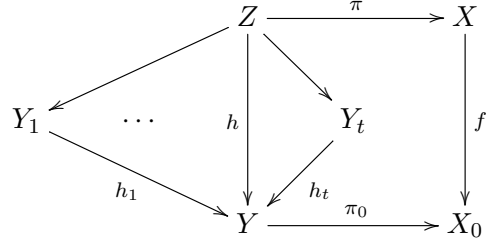
    $$
    \xymatrix{
    & & Z \ar[rr]^{\pi}\ar[dd]_{h} \ar[dr]\ar[dll]  & & X\ar[dd]^{f} \\
    Y_1 \ar[drr]_{h_1}  &  \cdots & & Y_t \ar[dl]^{h_t}  & \\
    & &  Y \ar[rr]^{\pi_0} & &   X_0
    }
    $$
    

    \begin{thm}\label{thm:main-Gal}
    Fix  $t\geq 2$ and $\alpha>0$. Then there exist constants $c_2=c_{2,t}>0$ and $d_2=d_{2,\alpha,t}$, depending only on $\alpha$ and $t$, such that for every indecomposable covering $f:X\ra \mP^1$ with monodromy group $G\leq S_\ell \wr S_t$ of product type, whose corresponding coverings $h_i:Y_i\ra Y$, $i=1,\ldots,t$ admit irreducible fiber products $Y_i\#_YY_j$, $i\neq j$ of genus  $<\alpha\ell$,
    either $g_X >  c_2\ell^{t-1}-d_2\ell^{t-2},$ or
    $t=2$ and the ramification of $f$ appears in Table~\ref{table:wreath}. 
    \end{thm}
    Note that the proof shows that one can choose $c_2=1/(3m)$ for all $t\geq 3$. Note that since $t\leq m\leq t!$,  constants depending on $m$ are bounded by constants depending only on $t$.
    We first deduce Theorem \ref{thm:main-wreath} from Theorem~\ref{thm:main-Gal}:
    \begin{proof}[Proof of Theorem \ref{thm:main-wreath}]
    Let $c_2,d_2$ be the constants\footnote{The reader who is not interested in the dependence between constants is encouraged to skip the first paragraph and replace all terms of the form $E\ell^{t-2}$ where $E$ depends on $\alpha$ and $t$ by an expression of the form $O_{\alpha,t}(\ell^{t-2})$.}
     from Theorem \ref{thm:main-Gal} with respect to a constant $\alpha$, depending only on $t$, such that $\alpha\geq 2m+1/2$. 
    Note that we may choose such $\alpha$ to  depend only on $t$ since $m$ is bounded by $t!$. 
    Let $E_0$, and $L_{2,\alpha}$ be the constants from Proposition \ref{prop:genus-analysis} and Corollary \ref{cor:DZ}, respectively.
    Pick $c\leq \min\{1/2,5/m,c_2\}$ and let $d> \max\{d_2,5E_0/m\}$ be sufficiently large so that $d/c\geq \max\{6+m,4E_0,L_{2,\alpha}\}$. 
    Since the assertion is trivial if $c\ell-d<0$, we may assume that $\ell\geq d/c\geq \max\{m+6,4E_0,L_{2,\alpha}\}$. Since $c\ell^{t-1}-d\ell^{t-2}<\ell^{t-1}$, it suffices to prove the claim when $g_X<\ell^{t-1}$. 
    
    As explained in \S\ref{sec:setup} we  assume the setup therein and also recalled at the beginning of the section. 
    We first prove the theorem under the assumption that the set $S$ of branch points of $\pi_0$ is of cardinality $\abs S\geq 5$. 
    Consider the covering $\tilde f:=f\circ \pi = \pi_0\circ h$.
    Since $\pi_0$ is Galois and $e_{\pi_0}(P)\geq 2$ for $P\in S$, we have $e_{\tilde f}(Q)\geq 2$ for every $Q\in \tilde f^{-1}(S)$. Thus, 
    \begin{equation}\label{equ:tilde-f}
    R_{\tilde f}(P)\geq \frac{\deg \tilde f}{2}= \frac{m}{2}\ell^t
    \end{equation} for every $P\in S$. 
    By Proposition \ref{prop:genus-analysis}.(1), we also have \begin{equation}\label{equ:usual-pi-bound}
    R_{\pi}(f^{-1}(P)) \leq m\ell^{t-1}/2 + E_0\ell^{t-2}
    \end{equation} for every $P\in S$. 
    Since $R_{\tilde f}(P) = mR_f(P) + R_\pi(f^{-1}(P))$ by the chain rule \eqref{equ:chain}, we deduce from \eqref{equ:tilde-f} and \eqref{equ:usual-pi-bound} that 
    $$R_f(P) = \frac{R_{\tilde f}(P)-R_{\pi}(f^{-1}(P))}{m}\geq \frac{\ell^t-\ell^{t-1}}{2} - \frac{E_0}{m}\ell^{t-2}.$$ As $g_{X_0}\geq 0$, together with the Riemann--Hurwitz formula for $f$ this gives 
    \begin{equation}\label{equ:usingS}
    \begin{split}
    2(g_X-1)  & \geq 2\ell^t(g_{X_0}-1) + \sum_{P\in S}R_f(P) \geq
    -2\ell^t  + \abs{S}\bigl(\frac{\ell^t-\ell^{t-1}}{2}-\frac{E_0}{m}\ell^{t-2}\bigr) \\
    & = (\abs{S}-4)\bigl(\frac{\ell^t-\ell^{t-1}}{2}-\frac{E_0}{m}\ell^{t-2}\bigr) - 2\ell^{t-1}-\frac{4E_0}{m}\ell^{t-2}.
    \end{split}
    \end{equation}
    As $\ell\geq \max\{6,E_0\}$, one has $(\ell^t- \ell^{t-1})/2-E_0\ell^{t-2}/m>0$. 
    As in addition $\ell\geq 6$, and $c\leq c_2\leq 1/2$, and $d>d_2> 5E_0/m$, \eqref{equ:usingS} in the case $\abs S\geq 5$ gives
    \begin{align*}
    2(g_X-1)\geq \frac{\ell^t-5\ell^{t-1}}{2}-\frac{5E_0}{m}\ell^{t-2} \geq
    \frac{\ell^{t-1}}{2}-\frac{5E_0}{m}\ell^{t-2}>c\ell^{t-1}-d\ell^{t-2}.
    \end{align*}
    
    Henceforth, assume $\abs S\leq 4$. We next prove the theorem under the assumption $g_Y>1$. By the Riemann--Hurwitz formula for  $\pi$ and $h$, one has
    \begin{equation}\label{equ:gXYZ} 2m(g_X-1) + R_\pi = 2(g_Z-1) \geq 2\ell^t(g_Y-1), 
    \end{equation}
    where $R_\pi:=\sum_{P\in X_0(\K)}R_\pi(f^{-1}(P))$. 
    Since $\abs S\leq 4$ and $\ell>2E_0/m$, Proposition \ref{prop:genus-analysis}.(1) implies 
    \[
    R_\pi\leq \abs S \bigl(\frac{m}{2}\ell^{t-1}- E_0\ell^{t-2}\bigr)\leq 2m\ell^{t-1}-4E_0\ell^{t-2}.
    \]
    Hence for $g_Y>1$, as $\ell>5+m$, and $c\leq 5/m$, and $d>2E_0/m$, \eqref{equ:gXYZ} gives 
    \[
    g_X> \frac{\ell^t - m\ell^{t-1}}{m} - \frac{4E_0}{m}\ell^{t-2} >  \frac{5}{m}\ell^{t-1}-\frac{2E_0}{m}\ell^{t-2} > c\ell^{t-1}-d\ell^{t-2}. 
    \] 
    
    Henceforth assume that $g_Y\leq 1$. 
    Note that this implies $g_{X_0}=0$ as follows. 
    Clearly $g_{X_0}\leq g_Y\leq 1$. If moreover $g_{X_0}=1$, then $\pi_0$ is unramified. In this  case, all branch cycles of $f$ are contained in $S_\Delta^I$ and hence $G\leq S_\Delta^I$, which contradicts the transitive action of $G$ on $I$ by Lemma \ref{lem:product-type}.  
     
    Since $g_{Y}\leq 1$, Proposition \ref{prop:asymp} gives $R_\pi\leq \max\{2m\ell^{t-1}+4E_0\ell^{t-2},2tE_0\ell^{t-2}\}.$
    As in addition $g_X<\ell^{t-1}$, and $\ell\geq 4E_0$, and $m\geq t$, and $\alpha\geq  2m+1/2$, this bound on $R_\pi$ and the Riemann--Hurwitz formula for $\pi$ give
    \[
    \begin{split} 2(g_Z-1)  = 2m(g_X-1) + R_\pi & < \max\{4m\ell^{t-1} + 4E_0\ell^{t-2},2m\ell^{t-1}+2tE_0\ell^{t-2}\} \\
    & \leq (4m+1)\ell^{t-1}\leq 2\alpha\ell^{t-1}.
    \end{split}
    \]
    Thus, letting $Y_{i,j}:=\tilde X/(K_i \cap K_j)$ for distinct $i,j\in I$,   the Riemann--Hurwitz formula for the natural projection $Z\ra Y_{i,j}$ gives
    $g_{Y_{i,j}}-1 \leq \frac{1}{\ell^{t-2}}(g_Z -1) <  \alpha\ell. $
    The natural projection $Y_{i,j}\ra Y$ is a minimal covering which factors through $h_i$ and $h_j$, and its degree is $\deg h_i\cdot\deg h_j=\ell^2$, for distinct $i,j\in I$. By assumption we also have $\ell\geq L_{2,\alpha}$.
    We can therefore apply Theorem \ref{thm:main-Gal}, 
    and deduce that either $g_X >  c_2\ell^{t-1} - d_2\ell^{t-2}$, 
    or $t=2$ and the ramification of $f$ appears in Table~\ref{table:wreath}. 
    In the latter case, $G$ contains $A_\ell^2$ by Lemma \ref{rem:ramification-genus}, and $g_X\leq 1$ by Remark~\ref{rem:ramification-genus}. 
    \end{proof} 
    We next show that Theorem \ref{thm:main-Gal} for $t\geq 3$ follows from: 
    \begin{prop}\label{cor:main-transposition}\label{prop:main-transposition}
    Let $t\geq 3$ and $\alpha>0$. There exists a constant $E_2=E_{2,t,\alpha}> 0$, depending only on $\alpha$ and $t$, which satisfies the following property.  
    For every covering $f:X\ra \mP^1$ with monodromy group $G\leq S_\ell \wr S_t$ of product type, such that its associated coverings $h_i:Y_i\ra Y$ admit irreducible fiber  products $Y_i\#_Y Y_j$, $i,j\in\{1,\ldots,t\}$ of genus  $<\alpha\ell$,
    the Riemann--Hurwitz contribution  of the associated covering $\pi$ is bounded by $\sum_{P\in\mP^1(\K)} R_\pi(f^{-1}(P))  < E_2\ell^{t-2}$. 
    \end{prop}     
    \begin{proof}[Proof of Theorem \ref{thm:main-Gal} for $t\geq 3$]
    We claim that $g_Z-1> \ell^{t-1}/6 -d'_2\ell^{t-2}$ for a constant $d'_2$ depending only on $t$ and $\alpha$. 
    Since the contribution $R_\pi:=\sum_{P\in X(\K)}R_\pi(f^{-1}(P))$ is bounded by $E_2\ell^{t-2}$ by Proposition \ref{cor:main-transposition}, the theorem follows from the claim by applying the Riemann--Hurwitz formula for~$\pi$:
    $$ 
    2(g_X-1) =  \frac{1}{m}(2(g_Z-1)-R_\pi) > \frac{1}{3m}\ell^{t-1}-\frac{2d'_2+E_2}{m}\ell^{t-2}, 
    $$
    by choosing $c_2\leq 1/(3m)$ and $d_2\geq (2d'_2+E_2)/m$, depending only on $t$. Note this is possible since as usual $m$ is bounded by a constant $t!$ depending only on $t$. 
    
    To prove the claim, we apply the Riemann--Hurwitz formula to the natural projection $h_1':Z\ra Y_1'$, 
    where  
    $Y_1':=\tilde X/(K_2\cap \cdots \cap K_t)$. 
    Recall that as $G$ is of product type, 
    Lemma \ref{rem:product-type}.(1) implies  $[K:\bigcap_{j=2}^t K_j]=\ell^{t-1}$ and hence that $\deg h_1'=[\bigcap_{j=2}^t K_j:\bigcap_{j=1}^t K_j]=\ell$. 
    
    Note that as $G$ is of product type and hence nonsolvable, and since $h_i$, $i=1,\ldots,t$ admit irreducible fiber products $Y_i\#_Y Y_j$, $i\neq j$ of genus $< \alpha\ell$, Lemma~\ref{lem:normal-closure} implies that there exists $P_0\in Y(\K)$ of type $m_0=m_{h_1}(P_0)<\infty$, and $r\in E_{h_1}(P_0)$ that does not divide $m_0$. 
    $$
    
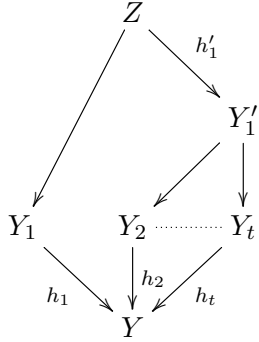
\xymatrix{
     & Z \ar[ddl] \ar[dr]^{h_1'} & \\
     & & Y_1' \ar[dl]\ar[d]\\
     Y_1 \ar[dr]_{h_1} & Y_2 \ar@{..}[r] \ar[d]^{h_2} & Y_t \ar[dl]^{h_t} \\
     & Y &
    }
    $$
    Let $U$ be the set of preimages $Q_Z\in h^{-1}(P_0)$ whose image $Q_i$ in $Y_i$ has ramification index $e_{h_i}(Q_i)=r$ if $i=1$ and $e_{h_i}(Q_i)=m_0$ if $i=2,\ldots,t$. 
    Since $h$ is a minimal covering which factors through $h_i$, $i\in I$, and $\deg h= \deg h_1\cdots\deg h_t$ by Remark~\ref{rem:setup}.(1), Abyahnkar's Lemma \ref{lem:abh} implies that for each choice of preimages $Q_i\in h_i^{-1}(P_0)$, with $e_{h_i}(Q_i)=r$ if $i=1$  and $e_{h_i}(Q_i)=m_0$ if $i=2,\ldots, t$, there exist $(r,m_0)\cdot m_0^{t-2}$ points of $Z$ with image $Q_i$ in $Y_i$, $i=1,\ldots,t$. Let $\eps=\eps_{\alpha,m_0}$ be the error from Definition \ref{def:almost-galois-type}. 
    Since $m_0 = m_{h_i}(P_0)<\infty$ for all $i\in I$, 
    there are at least $(\ell-\eps)/m_0$ preimages $Q_i\in h_i^{-1}(P_0)$  such that $e_{h_i}(Q_i)=m_0$, for each $i=2,\ldots, t$. Thus in total 
    \begin{equation}\label{equ:U}
    \abs U\geq (r,m_0)\cdot m_0^{t-2}\left(\frac{\ell-\eps}{m_0}\right)^{t-1} \geq \frac{(r,m_0)}{m_0}\ell^{t-1} - E_3\ell^{t-2},
    \end{equation}
    for some constant $E_3$ depending only on $t$ and $\alpha$. 
    Moreover, Abhyankar's Lemma \ref{lem:abh} implies $e_{h_1'}(Q_Z)=r/(r,m_0)$ for all $Q_Z\in U$. 
    Hence the Riemann--Hurwitz formula for $h_1':Z\ra Y_1'$ and \eqref{equ:U} give:
    \begin{equation}\label{equ:y1'}
    \begin{split}
    2(g_Z-1) & \geq  2\ell(g_{Y_1'}-1) + \sum_{Q_Z\in U}(e_{h_1'}(Q_Z)-1)
    \geq -2\ell + \abs U\bigl(\frac{r}{(r,m_0)}-1\bigr) \\
     &  > -2\ell +\frac{(r,m_0)}{m_0}\bigl(\frac{r}{(r,m_0)}-1\bigr)\ell^{t-1}  -E_3\bigl(\frac{r}{(r,m_0)}-1\bigr)\ell^{t-2} \\
     & >  \frac{r-(r,m_0)}{m_0}\ell^{t-1} - 2d'_2\ell^{t-2}, 
     \end{split}
    \end{equation}
    for some constant $d'_2$, depending only on $t$, such that $2d_2'>2/\ell^{t-3}+(r-(r,m_0))E_3/(r,m_0)$. Note that such a $d'_2$ exists since $r$ is bounded by the error $\eps$ which depends only on $\alpha$ and $m_0\leq 6$. 
    As $m_0\leq 6$, as $m_0\neq 5$, and as $r$ does not divide $m_0$, we have $(r-(r,m_0))/m_0\geq 1/3$.  
    Thus \eqref{equ:y1'} gives $g_Z-1> \ell^{t-1}/6 - d'_2\ell^{t-1}$, proving the claim, and hence~the~theorem.
    \end{proof}
    The proof of Theorem \ref{thm:main-Gal} for $t=2$ is given in Section \ref{subsec:t=2}.
    Proposition \ref{cor:main-transposition} follows from Proposition \ref{prop:asymp} in case none of the branch cycles acts on $I$ as a transposition. In the remaining cases, where  $t\geq 3$ and some branch cycle acts as on $I$ as a transposition, Proposition \ref{prop:main-transposition} is proved in Section \ref{sec:reduced}, completing the proof of Theorem \ref{thm:main-Gal}. 


    \section{Theorem \ref{thm:main-Gal}: The case $t=2$}\label{sec:t=2}
    \subsection{Relating $g_X$ and $g_{Y_1}$}
    We use the setup of \S\ref{sec:setup} with $t=2$, $X_0=\mP^1$, and write $I=\{1,2\}$, so that $f:X\ra \mP^1$ is an indecomposable covering with monodromy group $G\leq S_\Delta \wr S_2$ of product type, and $\tilde X$ is its Galois closure. As in \S\ref{sec:setup}, we associate to $f$ the coverings $h_i:Y_i\ra Y$, where $Y_i:=\tilde X/K_i$, $Y:=\tilde X/K$,   $K:=G\cap S_\Delta^2$, and $K_i$ a point stabilizer in the action of $K$ on the $i$-th copy of $\Delta$, $i=1,2$. Recall that $h_1$ and $h_2$ are indecomposable and their fiber product  coincides with a curve $Z$ birational to the fiber product of $f$ with the (degree-$2$) natural projection $\pi_0:Y\ra X_0$, see Remark \ref{rem:setup}. 
    
    The following lemma gives an explicit formula for the genus of $X$ in terms of the ramification of $h_i$, $i=1,2$. 
    The proof of Theorem \ref{thm:main-Gal} for $t=2$, given in this section, relies on this formula and  estimates of it. 
Fix $\sigma\in G\setminus K$, and let $\oline \sigma:Y\ra Y$ be the automorphism it induces, cf.~Remark \ref{rem:setup}.(2). Note that $\oline \sigma$ is an involution and  $\pi_0(P^{\oline \sigma}) = \pi_0(P)$, $P\in Y(\K)$. 
    \begin{lem}\label{lem:t=2-gXY1} 
    Write  the points of $Y$ as a disjoint union $R \cup S\cup S^{\oline\sigma}$, where $R$ is the set of ramification points of $\pi_0$. Then
    $$
    4g_X -4  = 2\ell(g_{Y_1}-1) + \sum_{P\in S}S_{h_1,h_2}(P) + \sum_{P\in R} S_{h_1}(P) - \sum_{P \in R} \abs{\{r\in E_{h_1}(P)\,|\, r\text{ is odd}\}}, $$
    where $S_{h_1}(P) := \sum_{r_1,r_2\in E_{h_1}(P)}(r_1-(r_1,r_2))$, and 
    $$S_{h_1,h_2}(P) := \sum_{r\in E_{h_1}(P),s\in E_{h_2}(P)}(r+s-2(r,s)).$$ 
    \end{lem}
    \begin{proof}
    Let  $h_{1}^2: Z \ra Y_2$ be the natural projection, $P$ a point of $Y$, and $p:=\pi_0(P)$. 
    As in Setup \ref{sec:setup}, we have  $E_{h_2}(P^{\oline \sigma}) = E_{h_1}(P)$ for every point $P$ of $Y$.  Note that $P^{\oline\sigma}=P$ if and only if $P\in R$. 
     Thus, if $P\in R$,  Abhyankar's Lemma \ref{lem:abh}  implies:
    $$R_{h_{1}^2}((\pi_0\circ h_2)^{-1}(p)) = R_{h_{1}^2}(h_2^{-1}(P)) = \sum_{r\in E_{h_1}(P), s\in E_{h_2}(P)}(r-(r,s)) = \sum_{r_1,r_2\in E_{h_1}(P)}(r_1-(r_1,r_2)),  $$ 
     and, if $P\in S$, 
    \begin{align*}
    R_{h_{1}^2}((\pi_0\circ h_2)^{-1}(p))  = & R_{h_{1}^2}(h_2^{-1}(P)) + R_{h_{1}^2}(h_2^{-1}(P^{\oline\sigma})) \\ = & \sum_{r\in E_{h_1}(P), s\in E_{h_2}(P)}(r-(r,s))+  \sum_{s\in E_{h_1}(P^\oline\sigma), r\in E_{h_2}(P^{\oline\sigma})}(s-(r,s)) & \\ 
    = & \sum_{r\in E_{h_1}(P), s\in E_{h_2}(P)}(r+s-2(r,s)). &
    \end{align*} 
     Thus the Riemann--Hurwitz formula  for $h_{1}^2$ gives:
    $$2(g_Z-1) = 2\ell (g_{Y_2}-1) + \sum_{P\in R}S_{h_1}(P) + \sum_{P\in S}S_{h_1,h_2}(P) $$ 
    Substituting 
     $g_{Y_2}=g_{Y_1}$ as $Y_2\cong Y_1$ by Remark \ref{rem:setup}.(2), and  
     replacing the left hand side using the Riemann--Hurwitz formula for the natural projection $\pi:Z\ra X$ gives
    $$4(g_X-1) + \sum_{p\in X_0(\K)}R_\pi(f^{-1}(p)) = 2\ell (g_{Y_1}-1) + \sum_{P\in R}S_{h_1}(P) + \sum_{P\in S}S_{h_1,h_2}(P). $$ 
    The claim follows from the latter equality since  $R_\pi(f^{-1}(p)) = \abs{\{r\in E_{h_1}(p)\,|\, r\text{ is odd}\}}$ by Proposition~\ref{lem:gZ-gX-t=2}.(3) and Remark \ref{rem:t=2-ramification}, for every point $p$ of $X_0$.
    \end{proof}

    We use the following estimates on $S_{h_1}$ and $S_{h_1,h_2}$. We denote by $O_\alpha(1)$ a constant depending only on $\alpha$, and write $X=Y+O_\alpha(1)$ to denote that $\abs{X-Y}$ is bounded by a constant depending only on $\alpha$. 
    \begin{lem}\label{lem:RS-estimates} 
    For an integer $\alpha>0$, assume the fiber product $Y_1\#_Y Y_2$ is irreducible of  genus at most $\alpha\ell$. Let $P$ be a point of $Y$ of almost Galois type $m_P := m_{h_1}(P) = m_{h_2}(P)$. 
    If $m_P<\infty$, then $S_{h_1,h_2}(P) = S_{h_1}(P) + S_{h_2}(P) + O_\alpha(1)$,~and 
    \[ 
    S_{h_1}(P) = \begin{cases} 
     \ell R_{h_1}(P) + O_\alpha(1)  \text{ if }m_P=1; \\
    \ell[\frac{\ell}{2} - \abs{E_{h_1}(P)} + \abs{\{r\in E_{h_1}(P)\,|\,\text{ $r$ is odd}\}}] + O_\alpha(1)  \text{ if }m_P=2; \\
     \ell[\frac{\ell}{3} - \abs{E_{h_1}(P)} + \frac{4}{3}\abs{\{r\in E_{h_1}(P)\,|\,(r,3)=1\}}] + O_\alpha(1)  \text{ if }m_P=3; \\
      \ell[\frac{\ell}{4} - \abs{E_{h_1}(P)} + \abs{\{r\in E_{h_1}(P)\,|\,r\equiv 2\,(4)\}}
       + \frac{3}{2}\abs{\{\text{odd }r\in E_{h_1}(P)]} \\ + O_\alpha(1)  \text{ if }m_P=4; \\
    \ell\left[\frac{\ell}{6} - \abs{E_{h_1}(P)} + \abs{\{r\in E_{h_1}(P)\,|\,r\equiv 3\,(6)\}} + \frac{4}{3}\abs{\{r\in E_{h_1}(P)\,|\,r\equiv \pm 2\,(6)\}}\right. \\
    \left. + \frac{5}{3}\abs{\{r\in E_{h_1}(P)\,|\,r\equiv \pm 1\,(6)\}}\right] + O_\alpha(1)  \text{ if }m_P=6.  \\
    \end{cases}
    \]

    If $m_P=\infty$,  then 
    $S_{h_1,h_2}(P) \geq  (u+v)\ell - 2\ell\min(u,v)$, where $u:=\abs{E_{h_1}(P)}$ and $v:=\abs{E_{h_2}(P)}$. 
    Moreover, if $u\leq v\leq 3$, and the greatest common divisor of all entries in $E_{h_1}(P)$ and $E_{h_2}(P)$ is $1$, then $S_{h_1,h_2}(P)\geq (u+v-2\min(u,v)+\xi)\ell$ for an absolute constant $\xi>0$. 
    \end{lem}
    
    The proof relies on the following lemma: 
    \begin{lem}\label{lem:t=2-mini}\label{lem:minimize-gcd} \label{lem:bound-rs}
    Let  $r_1,...,r_u, s_1,...,s_v$  be positive integers such that  $\sum_{i=1}^u r_i = \sum_{j=1}^v s_j =~\ell$, and $u\leq v$.
    Let $S:=\sum_{i=1}^u\sum_{j=1}^v (r_i+s_j-2(r_i,s_j))$. Then
    \begin{enumerate} 
    \item  $S\geq \ell(v-u)$;
    \item There exists an absolute constant $\xi>0$ such that if $v\leq 3$, and $\ell\geq (42v)^v$, and $(r_1,\ldots, r_u, s_1,\ldots,s_v)=1$, then $S\geq \ell(v-u+\xi)$. 
    \item If $u=v=2$,  $\{r_1,r_2\}\cap \{s_1,s_2\} = \emptyset$, and 
    $(r_1,r_2,s_1,s_2)=1$, then $S> 5\ell/2 - 5$. 
    \end{enumerate}
    \end{lem}
    \begin{proof}[Proof of Lemma \ref{lem:t=2-mini}]
    (1) Since $(r_i,s_j)\leq s_j$ for all $i,j$, since $\sum_{i=1}^u\sum_{j=1}^v s_j =  u\ell$ and $\sum_{i=1}^u\sum_{j=1}^v r_i =  v\ell$, we have $S\geq \ell(v+u)-2\sum_{i=1}^u\sum_{j=1}^vs_j=\ell(v-u), $ as desired. 
    
    (2) Now assume $v\leq 3$, and $(r_1,\ldots, r_u,s_1,\ldots,s_v)=1$. 
    Putting $D=1/42$, a straightforward check  
     shows that for any $w$ with $0\leq w\leq 3$, the sum of $w$ divisors of $\ell$ does not lie in the interval  
    $( (1-D)\ell, \ell )$.  
    Note that  the sum 
    $$T:=\frac{S-\ell(v-u)}{2} = \sum_{i=1}^u\sum_{j=1}^v\bigl(s_j - (r_i,s_j)\bigr)$$   is at least
    $s_j - (r_i,s_j)$  for any prescribed $i,j$.  
    Put $\xi_v:=D/v$, and $\xi:=\xi_3\leq \xi_v$. If there exist $i,j$ such that  $s_j \geq D\ell/v$  and  $s_j$ doesn't divide $r_i$,  then $(r_i,s_j)$ is a proper divisor of $s_j$  and hence is at most $s_j/2$,  so that
    $$T \geq s_j - (r_i,s_j) \geq \frac{s_j}{2} >  \frac{D}{2v}\ell>\frac{\xi}{2}\ell,
    \text{ and hence }S\geq \ell(v-u+\xi).$$   
    
    Now suppose that 
    $\ell > (v/D)^v$ and that every $s_j$ which is at least $D\ell/v$ divides all $r_i$, $i=1,\ldots,u$.  
    Let $J:=\{1,\ldots,v\}$, let  $U_1$ be the multiset $\{s_j\suchthat j\in J,\text{ and } s_j \geq D\ell/v\}$,  and let $U_2$  be
    the multiset $\{s_j\suchthat j\in J,\text{ and }s_j< D\ell/v\}$.  Since $s_1+\ldots+s_v=\ell$, the biggest $s_j$ is at least $\ell/v$, which is at least $D\ell/v$ 
    since $D< 1$, so that $U_1$ contains an element which is at least $\ell/v$.  
    Each  $s \in U_1$  divides every $r_i$, and hence  divides
    $\ell=\sum_{i=1}^u r_i$.  
    Thus the elements of $U_1$ are divisors of $\ell$, and
    $$\gcd(U_1) = \frac{\ell}{\lcm(\{\ell/s \suchthat s \in U_1\})} \geq \frac{\ell}{\prod_{s \in U_1} \frac{\ell}{s}} \geq \frac{\ell}{\prod_{s \in U_1} \frac{v}{D}} \geq \frac{\ell}{(v/D)^{v}}>1.$$
    Since  $\gcd(U_1)$ is nontrivial and divides $(r_1,...,r_u)$, the assumption $(r_1,\ldots,r_u,s_1,\ldots,s_v)=1$ implies that $U_2$ is nonempty.  Thus  $U_1$  is a multiset of
    divisors of $\ell$ whose sum is less than $\ell$;  since $\abs{U_1} < v$, it follows that the sum of the 
    elements of $U_1$ is at most $(1-D)\ell$, whence the sum of the elements of $U_2$  is at least  $D\ell$.  Here  $U_2$ is a multiset of
    integers $1\leq z<D\ell/v$, and $\abs{U_2} \leq v$,  so the sum of the elements of $U_2$ is less than $D\ell$, contradicting  $\sum_{s\in U_1\cup U_2}s=\ell$. 
    
     
    
    (3) Now let $u=v=2$. Let $S_i:=(r_1,s_i)+(r_2,s_i)$,  for $i=1,2$. 
    Without loss of generality assume $s_2$ is the maximum in $\{r_1,r_2,s_1,s_2\}$. 
    Since $(r_1,r_2,s_1,s_2)=1$ and $s_2=r_1+r_2-s_1$ it follows that $(r_1,r_2,s_1)=1$.  Hence $(r_1,s_1)$ and $(r_2,s_1)$ are coprime. As both divide $s_1$, their product also divides $s_1$.  It follows that either $S_1\leq s_1/2+2$, or $s_1\divides r_j$ for some $j\in\{1,2\}$ in which case $S_1=s_1+1$. As $s_2>r_j$, $j=1,2$, the same argument gives $S_2\leq s_2/2+2$. 
    Thus $S_1+S_2\leq s_1+s_2/2 + 3 = \ell +3 -s_2/2$.  As $s_2>(\ell+3)/2$, we get 
    \[
    S=4\ell - 2(S_1+S_2)\geq  4\ell-2(\ell+3) +s_2 > (5\ell-9)/2.
    \]

    \end{proof}
    \begin{rem}The proof of part (2) generalizes to a statement for arbitrary $v$. Namely, it yields a constant $\xi_v>0$, depending only on $v$, such that for every $\ell$ sufficiently large compared to $v$, and positive integers $r_1,\ldots,r_u$ and $s_1,\ldots,s_v$ for $u\leq v$ such that $\sum_{i=1}^ur_i=\sum_{j=1}^vs_j=\ell$ and $(r_1,\ldots,r_u,s_1,\ldots,s_v)=1$, the above sum $S$ is at least $\ell(v-u+\xi_v)$. 
    The only required adjustment is replacing $D$ by a constant $0<D_v<1$, depending only on $v$, such that every sum of at most $v$ divisors of $\ell$ does not lie in the interval $((1-D_v)\ell,\ell)$. 
    \end{rem}

    \begin{proof}[Proof of Lemma \ref{lem:RS-estimates}]
    If $m_P=\infty$, we may assume without loss of generality that $\abs{E_{h_2}(P)}\geq \abs{E_{h_1}(P)}$. The assertion then follows immediately from Lemma \ref{lem:bound-rs}.(1-2) by writing $E_{h_1}(P)=[r_1,\ldots,r_u]$ and $E_{h_2}(P)=[s_1,\ldots,s_v]$. 
    
    Now assume $m_P$ is finite. Since there are $O_\alpha(1)$ entries of $h_2$ that are different from $m_P$, all of which are bounded by  $O_\alpha(1)$, we have
    \begin{align*}
    \begin{split}
    S_{h_1,h_2}(P)  = &  \displaystyle \sum_{\substack{r\in E_{h_1}(P),
    s\in E_{h_2}(P)\\
    s=m_P}} (r+m_P-2(r,m_P)) \\
     & + \displaystyle \sum_{\substack{r\in E_{h_1}(P),s\in E_{h_2}(P), \\ r=m_P}} (s+m_P-2(s,m_P)) +  \sum_{\substack{r\in E_{h_1}(P),s\in E_{h_2}(P),\\ r,s\neq m_P}} (r+s-2(r,s)) \\
     = & \frac{\ell}{m_P}\sum_{r\in E_{h_1}(P)} (r+m_P-2(r,m_P)) +  \frac{\ell}{m_P} \sum_{s\in E_{h_2}(P)} (s+m_P-2(s,m_P)) +O_\alpha(1) \\
       = &  S_{h_1}(P) + S_{h_2}(P) +  O_\alpha(1).
     \end{split}
    \end{align*}
    
     The estimates for $S_{h_1}(P)$ are proved in \cite[Lemma 10.3]{NZ}.
    \end{proof}

    Finally, we state some estimates from \cite[Proof of Proposition 10.1]{NZ}:
    \begin{lem}\label{lem:Sn-two-stab}
    There exist constants $c_3>0$ and $d_3$ such that for every indecomposable covering $h_1:Y_1\ra \mP^1$ with nonabelian monodromy and for which every point $P$ of $Y$ is of almost Galois type $m_{h_1}(P)$, the following property holds. Let $\ell:=\deg h_1$, and
    $M_{h_1}$ the multiset $\{m_{h_1}(P)>1\,|\, P\in\mP^1(\K)\}$. 
    \begin{enumerate}
    \item[(I1)] 
    If $M_{h_1} = \{\infty,\infty\}$ and $2\ell(g_{Y_1}-1) + \sum_{P\in \mP^1(\K)} S_{h_1}(P) < c_3\ell -d_3$, then 
    the ramification type of $h_1$ is $[\ell]$, $[a,\ell-a], [2,1^{\ell-2}]$ with $(a,\ell)=1$; 
    \item[(I2)] 
    If $M_{h_1}=\{\infty,2,2\}$ and $2\ell(g_{Y_1}-1) + \sum_{P\in \mP^1(\K)} S_{h_1}(P) - \ell < c_3\ell - d_3$, then the ramification type of ${h_1}$ is one of the types I2.1-I2.15 in \cite[Table~4.1]{NZ};
    \item[(F1)] 
    If $M_{h_1}=\{2,2,2,2\}$ and  $2\ell(g_{Y_1}-1) + \sum_{P\in \mP^1(\K)} S_{h_1}(P) - 2\ell < c_3\ell - d_3$, then the ramification type of ${h_1}$ is one of the types F1.1-F1.9 in \cite[Table~4.1]{NZ}; 
    \item[(F3)] 
    If $M_{h_1}=\{2,4,4\}$ and $2\ell(g_{Y_1}-1) + \sum_{P\in \mP^1(\K)} S_{h_1}(P) - \ell < c_3\ell - d_3$, then the ramification type of $h_1$ is one of the types F3.1-F3.3 in \cite[Table~4.1]{NZ}.
    \end{enumerate}
    \end{lem}
    \begin{proof}
    Estimating the terms $S_{h_1}(P)$ in each case using Lemma \ref{lem:RS-estimates}, the inequality 
    \[2\ell(g_{Y_1}-1) + \sum_{P\in \mP^1(\K)} S_{h_1}(P)<c_3\ell-d_3
    \] in case (I1) coincides with  \cite[(10.3)]{NZ} as the left hand side of the latter is bounded by $c_3\ell-d_3$ for some constants $c_3,d_3>0$, depending only on $\alpha$. 
    Similarly, the inequality $2\ell(g_{Y_1}-1) + \sum_{P\in \mP^1(\K)} S_{h_1}(P)-\ell<c_3\ell-d_3$ in case (I2) coincides with \cite[(10.6)]{NZ} combined with $g_{X_2}<c_3\ell-d_3$, the inequality $2\ell(g_{Y_1}-1) + \sum_{P\in \mP^1(\K)} S_{h_1}(P)-2\ell<c_3\ell-d_3$ in case (F1) coincides with 
    \cite[(10.10)]{NZ} with $g_X<c_3\ell-d_3$, and $2\ell(g_{Y_1}-1) + \sum_{P\in \mP^1(\K)} S_{h_1}(P)-\ell<c_3\ell-d_3$ in case (F3) coincides with 
    \cite[(10.14)]{NZ} for $g_X<c_3\ell-d_3$. 
    The proof of \cite[Proposition 10.1]{NZ}  determines all ramification types of indecomposable coverings $h_1$ satisfying these inequalities. 
    \end{proof}
    
    We shall also use the following lemma to show that certain ramification types do not correspond to an indecomposable covering: 
    \begin{lem}\label{cor:inf2-3} 
    There is no product-$1$ tuple that generates a primitive group and whose ramification corresponds to types I1A.N1, I1A.N2, F4.N1,  F4.N2, F4.N3 in Table~\ref{table:I1AN}. 
    \end{lem}
    \begin{proof}
    Assume on the contrary that there exists an indecomposable cover $f$ with monodromy group $G$ whose ramification type appears in Table \ref{table:I1AN}. Letting $x_1,\ldots, x_5$ be a product-$1$ tuple for $G$ corresponding to this ramification type, we see that at least one of the $x_i$'s or its square is in $K:=G\cap S_\ell^2$ and has a transposition, double transposition, or a $3$-cycle as its first coordinate. Since $h_1$ is indecomposable and the projection of $K$ to the first coordinate contains  an element with such cycle structure, this projection contains $A_\ell$ by Remark \ref{rem:jordan}. Since $G$ is primitive, this implies $K\supseteq A_\ell^2$ by Lemma \ref{lem:cyc-exist}.(1). Thus, the inequality    $g_{Y_1}-g_Y\geq g_X$ from \cite[Chp.\ 3]{Tali-wreath} holds. However for types I1A.N1 and I1A.N2 (resp.\ F4.N1-3),  $g_{X}=0$ using Lemma \ref{lem:t=2-gXY1}, $g_Y=0$ (resp.\ $g_Y=1$) since $\pi$ has two (resp.\ four) branch points, and $g_{Y_1}=1$ (resp.\ $g_{Y_1}=2$), contradicting this inequality. 
    \end{proof}
    \begin{table}
    \caption{Non-occurring ramification types for primitive groups $A_\ell^2\leq G\leq S_\ell \wr S_2$ of product type.}
    \label{table:I1AN}
    $$\begin{array}{|l|l|}
    \hline
    I1A.N1 & ([\ell],[\ell]),  ([2,1^{\ell-2}],[2,1^{\ell-2}]),  s,  s \\
    I1A.N2 & ([\ell],[\ell]),  ([2^2,1^{\ell-4}],[1^\ell]) s,  s. \\
    \hline
    F4.N1 & \s, \s, \s, ([1^{\ell-4},2^2],1) \s \\
    F4.N2 & \s, \s, \s, ([1^{n-3},3],1) \s \\
    F4.N3 & \s, \s, \s, \s, ([1^{\ell-2},2],[1^{\ell-2},2]) \\
    \hline
    \end{array}
    $$ 
    \end{table}
    Note that the inequality from \cite{Tali-wreath} is derived from a representation theoretic argument. Another proof of Lemma \ref{cor:inf2-3}, of more explicit nature,  is given in Appendix \ref{sec:no-ram}.

    \subsection{Proof of Theorem \ref{thm:main-Gal} for $t=2$}\label{subsec:t=2}
    
    We will show that there exists a sufficiently small $c_2>0$ and sufficiently large $d_2$ such that the condition $g_X<c_2\ell-d_2$ forces the ramification of $f$ to appear in Table~\ref{table:wreath}. Let $K:=G\cap S_\ell^2$ and $H$ be a point stabilizer of $G$. Let $\pi_0:Y\ra\mP^1$ and $\pi:Z\ra X$ be the natural projections from $Y:=\tilde X/K$ and $Z:=\tilde X/(H\cap K)$, where $\tilde X$ is the Galois closure of $f$. Recall that the coverings $h_1$ and $h_2$ are indecomposable by  Remark~\ref{rem:setup}.(3) as $G$ is primitive. 

    Write the branch points of $h$ as a disjoint union of three sets $R\cup S\cup S^\oline\sigma$ for $\oline\sigma \in G\setminus K$, where $R$ is the set of branch points of $\pi_0$. 
    Lemma \ref{lem:t=2-gXY1} 
    then gives 
    \begin{equation}\label{equ:genus-start}
    4g_X-4 =   2\ell(g_{Y_1}-1) + \sum_{P\in R} S_{h_1}(P)  + \sum_{P\in S} S_{h_1,h_2}(P) - \sum_{P\in R} \abs{\{\text{odd }r\in E_{h_1}(P)\}}. 
    \end{equation}
    Note that since $E_{h_2}(P) = E_{h_1}(P^{\oline\sigma})$  by Remark \ref{rem:setup}.(2), we have  $S_{h_2}(P) = S_{h_1}(P^\oline\sigma)$  for all $P\in S$.
    Also note that $S_{h_1,h_2}(P) = S_{h_1}(P) + S_{h_2}(P) + O_\alpha(1)$ for $P\in S$ with $m_{h_i}(P)<\infty$,  
    and that the number of branch points of $h_1$ is bounded by Corollary \ref{cor:DZ}. Thus, \eqref{equ:genus-start} gives
    \begin{equation}\label{equ:main-t=2}
    \begin{split}
    4g_X-4 & =  2\ell(g_{Y_1}-1) + \sum_{P\in R,m_{h_1}(P)=\infty}S_{h_1}(P) + \sum_{P\in S,m_{h_1}(P)=\infty}S_{h_1,h_2}(P) \\
    & + \sum_{P\in Y(\K), m_{h_1}(P)<\infty} S_{h_1}(P) - \sum_{P\in R} \abs{\{\text{odd }r\in E_{h_1}(P)\}} + O_\alpha(1).
    \end{split}
    \end{equation}
    Since $Y_1\#_Y Y_2$ is irreducible of genus  $< \alpha\ell$, Corollary \ref{cor:DZ} implies that $g_{Y}\leq 1$. 
    Hence, $\abs{R}=\sum_{p\in X_0(\K)}R_{\pi_0}(p) = 2(g_Y+1)$ is $2$ if $g_Y=0$ and $4$ if $g_Y=1$, by the Riemann--Hurwitz formula for $\pi_0$. 
    
    First assume $g_Y=1$ and hence that $\abs R=4$,
    and that $m_{h_1}(P) =1$ for every point $P$ of $Y$, by Corollary \ref{cor:DZ}.
    By Lemma \ref{lem:RS-estimates}, we have $S_{h_1}(P) = \ell R_{h_1}(P) + O_\alpha(1)$ for every point $P$ of $Y$. Letting $R_{h_1}:=\sum_{P\in Y(\K)}R_{h_1}(P)$, we have in addition $2(g_{Y_1}-1) = R_{h_1}$  for $P$ in $Y_1$ by the Riemann--Hurwitz formula for $h_1$. Thus  \eqref{equ:main-t=2} gives: 
    $$
    4(g_X-1) \geq 2\ell(g_{Y_1}-1) + \ell R_{h_1} - 4\ell + O_\alpha(1) = 2\ell (R_{h_1}-2) + O_\alpha(1).
    $$
    Thus,  we have $g_X> \ell/2 - O_\alpha(1)$  if $R_{h_1}>2$. Hence by requiring $c_2\leq 1/2$ and taking $d_2$ to be sufficiently large, we may assume $R_{h_1}\leq 2$. Note that  $R_{h_1}$ is even since  $R_{h_1}=2(g_{Y_1}-1)$, so that $R_{h_1}=0$ or $2$. 
    If $R_{h_1}=0$, then $h_1$ is a morphism of genus $1$ curves and hence $\Mon(h_1)$ is abelian \cite[Theorem 4.10(c)]{Sil}, contradicting the nonsolvability of $G$. 
    Thus, $R_{h_1}=2$ and the ramification type of  $h_1$ is either $[2,1^{\ell-2}]^2$, or $[3,1^{\ell-3}]$, or $[2,2,1^{\ell-4}]$. By Remark \ref{rem:t=2-ramification}, the ramification types of $f$ corresponding to such  $h_1$ appear as the F4 types in Tables \ref{table:wreathY=2} and \ref{table:I1AN}. Types I1A.N1, I1A.N2, 
    F4.N1, F4.N2 and F4.N3  in Table~\ref{table:I1AN} do not correspond to an indecomposable covering by Lemma \ref{cor:inf2-3}. 
    
    Henceforth assume $g_Y=0$ and hence that $R=\{P_1,P_2\}$ for two points $P_1,P_2$ of~$Y$. 
    We analyze the possibilities for the multiset 
    $ M_{h_1} := \{m_{h_1}(P)>1\,|\, P\in Y(\K)\}.  $
    Since $\abs R=2$ and since  $m_{h_1}(P) = m_{h_1}(P^\oline \sigma)$ for every $P\in S$ by Remark \ref{lem:same-m_P},  Corollary~\ref{cor:DZ} gives the following possibilities for $M_{h_1}$:
    \begin{enumerate}
    \item[(I1A)] $M_{h_1} = \{\infty,\infty\}$, $m_{h_1}(P_i)=1$, $i=1,2$;  
    \item[(I1B)] $M_{h_1} = \{\infty,\infty\}$, $m_{h_1}(P_i)=\infty$, $i=1,2$; 
    \item[(I2)]  $M_{h_1} = \{\infty, 2,2\}$, $\{m_{h_1}(P_1), m_{h_1}(P_2)\} = \{1,\infty\}$; 
    \item[(F1A)] $M_{h_1} = \{2, 2, 2, 2\}$, $m_{h_1}(P_i)=1$, $i=1,2$; 
    \item[(F1B)] $M_{h_1} = \{2, 2, 2, 2\}$, $m_{h_1}(P_i)=2$, $i=1,2$;
    \item[(F2)]  $M_{h_1} = \{3,3,3\}$, $\{m_{h_1}(P_1), m_{h_1}(P_2)\} = \{1,3\}$; 
    \item[(F3)] $M_{h_1} = \{2,4,4\}$, $\{m_{h_1}(P_1), m_{h_1}(P_2)\} = \{1,2\}$. 
    \end{enumerate}
    We estimate the right hand side of \eqref{equ:main-t=2} in each of these cases using Lemma~\ref{lem:RS-estimates}. In each case we require conditions on $c_2$, and $d_2$. The theorem then follows by taking the minimal (resp., maximal) value of $c_2$ (resp., $d_2$) among all cases.  
    
    {\bf Case I1A}:
    Let $P_3$ be a point of $Y$  such that $m_{h_1}(P_3)= \infty$, and set $P_4=P_3^{\oline\sigma}$, so that $m_{h_1}(P_4) =\infty$. 
    Set $u=u_{P_3}:=\abs{E_{h_1}(P_3)}$ and $v=v_{P_4}:=\abs{E_{h_2}(P_3)}$ which equals $\abs{E_{h_1}(P_4)}$ by Remark \ref{rem:setup}.(2). 
    By Lemma \ref{lem:RS-estimates},  equality \eqref{equ:main-t=2} gives: 
    \begin{equation}\label{equ:RH-XY1-inf-inf}
    4(g_X-1) 
    \geq 2\ell(g_{Y_1}-1) + \ell (u+v) - 2\ell\min(u,v) + \xi_h \ell +
    \ell\sum_{P\in Y(\K)\setminus\{P_3,P_4\}}R_{h_1}(P) -2\ell + O_\alpha(1). 
    \end{equation}
    where we set $\xi_h$ to be the constant $\xi$ from Lemma \ref{lem:RS-estimates} if $\max\{u,v\}\leq 3$ and the greatest common divisors of all entries in $E_{h_1}(P_3)$ and $E_{h_2}(P_3)$ is $1$, and set $\xi_h=0$ otherwise. 
    On the other hand the Riemann--Hurwitz formula for $h_1$ gives
    \begin{equation}\label{equ:RH-Y1Y}
    \sum_{P\in Y(\K)\setminus\{P_3,P_4\}} R_{h_1}(P) = 2(g_{Y_1}-1) + u + v.
    \end{equation}
    Substituting \eqref{equ:RH-Y1Y} into \eqref{equ:RH-XY1-inf-inf} gives:
    $$ 4(g_X-1) \geq 4\ell(g_{Y_1}-1) + 2\ell (u+v) - 2\ell\min(u,v) + \xi_h \ell  -2\ell + O_\alpha(1). $$
    For  $0<c_2<\xi/4$ and  sufficiently large $d_2$, the condition $g_X<c_2\ell-d_2$ forces 
    \begin{equation*} 0 \geq 4(g_{Y_1}-1)  +  2(u+v) - 2\min(u,v) + \xi_h  -2,
    \end{equation*}
    or equivalently 
    \begin{equation}\label{equ:max-bnd}
    \max(u,v)\leq 3-2g_{Y_1}-\xi_h/2.
    \end{equation}
    In particular $g_{Y_1}\leq 1$. First, assume $g_{Y_1}=1$ so that $u=v=1$ by \eqref{equ:max-bnd}. Then \eqref{equ:RH-Y1Y} gives $\sum_{P\in Y(\K)\setminus\{P_3,P_4\}}R_{h_1}(P)=2$.
    Thus the ramification type of $h_1$ is 
    either $[\ell], [\ell], [3, 1^{\ell-3}]$,
     or $[\ell], [\ell], [2, 1^{\ell-2}]$ twice,
     or $[\ell], [\ell], [2, 2, 1^{\ell-4}].$
     By Remark \ref{rem:t=2-ramification}, these correspond to the ramification types of $f$ appearing as types I1A.1-3 and I1A.N1-2 in Tables~\ref{table:wreath} and~\ref{table:I1AN}. 
     By Lemma \ref{cor:inf2-3}, 
     types I1A.N1-2 do not correspond to an indecomposable covering $f$. 
    
    We can therefore assume $g_{Y_1}=0$ and without loss of generality $u\leq v$.
    Then  \eqref{equ:max-bnd} gives $v\leq 3-\xi_h/2$. 
    By Lemma \ref{lem:hurwitz1}, since $h_1$ is indecomposable, the greatest common divisor of all entries in $E_{h_1}(P_3)$ and $E_{h_2}(P_3)=E_{h_1}(P_4)$ is $1$. Since in this case $\xi_h>0$, \eqref{equ:max-bnd} gives $v\leq 2$. 
    If $u=v=1$, $E_{h_1}(P_3) = [\ell]$ and $E_{h_1}(P_4) = E_{h_2}(P_3) = [\ell]$, contradicting the indecomposability of $h_1$ by Lemma~\ref{lem:hurwitz1}.
    
    It remains to consider the case $v=2$. First let $u=1$, so that \eqref{equ:RH-Y1Y} implies that the sum $\sum_{P\in Y(\K)\setminus\{P_3,P_4\}}R_{h_1}(P)$ is  $1$.
    Hence $h_1$ has only one branch point $Q$ other than $P_3,P_4$, and $E_{h_1}(Q)=[2,1^{\ell-2}]$.
    Write  $E_{h_2}(P_3)=[a,\ell-a]$. Since $E_{h_1}(P_3) = [\ell]$ and $h_1$ is indecomposable, Lemma \ref{lem:hurwitz1} implies that 
    $(a,\ell)=1$.
    Hence  the ramification type of $h_1$ over $P_3,P_4,Q$ is type  I1.1 $[\ell], [a,\ell-a], [2, 1^{\ell-2}]$ for some $1\leq a<\ell$ with $(a,\ell)=1$.
    By Remark \ref{rem:t=2-ramification}, the corresponding ramification types for $f$ in Table~\ref{table:wreath} are I1A.7a-I1A.7b. 
    
    It remains to consider the case $u = v =2$ in which case  $\sum_{P\in Y(\K)\setminus\{P_3,P_4\}} R_{h_1}(P) =2$ by \eqref{equ:RH-Y1Y}. 
    Write $E_{h_1}(P_3)=[a_1,a_2]$ and $E_{h_2}(P_3)=~[b_1,b_2]$.
     By \eqref{equ:main-t=2} and Lemma \ref{lem:RS-estimates} 
    $$4g_X = -2\ell+\sum_{1\leq i,j\leq 2}(a_i+b_j-2(a_i,b_j)) +~O_\alpha(1).$$
     If $\{a_1,a_2\}\cap \{b_1,b_2\} = \emptyset$, then  $\sum_{1\leq i,j\leq 2}(a_i+b_j -2(a_i,b_j))\geq 5\ell/2-5$
     by Lemma~\ref{lem:t=2-mini}, and hence $g_X>\ell/8+O_\alpha(1)$. Thus, by choosing $c_2<1/8$ and  sufficiently large $d_2$, the condition $g_X<c_2\ell-d_2$  forces $a_1=b_1$ or $b_2$. Putting  $a:=a_1$, we deduce that $E_{h_1}(P_3)=E_{h_2}(P_3)=[a,\ell-a]$, and hence $E_{h_1}(P_4)=E_{h_2}(P_3)=[a,\ell-a]$.
     Since $h_1$ is indecomposable, Lemma \ref{lem:hurwitz1} implies 
    that $(a,\ell)=1$.
    
    Since $\sum_{P\in Y(\K)\setminus\{P_3,P_4\}} R_{h_1}(P) =2,$  the ramification type of $h_1$ 
    is one of the types I1A.4-6 in Table~\ref{table:I1A}.
    By Remark \ref{rem:t=2-ramification}, the corresponding ramification types for $f$ are types I1A.4-I1A.6 in Table~\ref{table:wreath}\footnote{The ramification data $([a,\ell-a],[a,\ell-a]), ([3,1^{\ell-3}],1)s, s$ and $([a,\ell-a],[a,\ell-a]), ([2^2,1^{\ell-4}],1)s, s$
    are ruled out here as they do not correspond to a covering. Indeed,  if there was a covering $f:X\ra\mP^1$ with such ramification then one would have  $g_X< 0$ by Lemma \ref{lem:t=2-gXY1}.} 

    {\bf Case I1B}:
    Estimating \eqref{equ:main-t=2} using Lemma \ref{lem:RS-estimates} gives:
    \begin{equation}\label{equ:I2-gXY1}
    4(g_X-1) = 
     2\ell(g_{Y_1}-1) + \ell\sum_{P\in Y(\K)\setminus R}R_{h_1}(P)+ \sum_{r_1,r_2\in E_{h_1}(P_i),i=1,2}(r_1-(r_1,r_2)) +O_\alpha(1).
    \end{equation}

    By Lemma \ref{lem:Sn-two-stab}, there exist $c_3>0$ and sufficiently large $d_3$ for which the conditions $g_X<c_3\ell-d_3$ and \eqref{equ:I2-gXY1} force  the ramification type of $h_1$ to be $[\ell], [a,\ell-a], [2,1^{\ell-2}]$  with $(a,\ell)=1$.  
    Since the ramification types $[\ell]$ and $[a,\ell-a]$ appear over $P_1$ and~$P_2$,
    Remark \ref{rem:t=2-ramification} implies that  
    the corresponding ramification type for $f$ is type I1.1 in Table~\ref{table:wreath}. 
    Choosing $c_2\leq c_3$ and $d_2\geq d_3$, the conclusion follows when $g_X<c_2\ell-d_2$. 
    
    {\bf Case I2:}
    Assume without loss of generality $m_{h_1}(P_1) = \infty$ and $m_{h_1}(P_2) = 1$. 
    Let $P_3\in S$ be a point with $m_{h_1}(P_3)=2$, and set $P_4=P_3^{\oline\sigma}$ so that  $m_{h_1}(P_4)=2$.
    Estimating \eqref{equ:main-t=2} using Lemma \ref{lem:RS-estimates} gives:
    \begin{equation}\label{equ:inf-2-2-gXY1}
    \begin{split}
     4g_X  = & 
    2\ell(g_{Y_1}-1) + \sum_{r_1,r_2\in E_{h_1}(P_1)}(r_1-(r_1,r_2)) + \ell\sum_{P\in Y(\K)\setminus\{P_1,P_3,P_4\}}R_{h_1}(P) \\
    & + \ell\sum_{j=3}^4\bigl(\frac{\ell}{2} - \abs{E_{h_1}(P_j)}+\abs{\{\text{odd }r\in E_{h_1}(P_j)\}}\bigr) -\ell +O_\alpha(1).
    \end{split}
    \end{equation}
    By Lemma \ref{lem:Sn-two-stab}, there exist $c_3>0$ and sufficiently large $d_3$ such that the conditions $g_X<c_3\ell - d_3$  and  \eqref{equ:inf-2-2-gXY1} force the ramification of $h_1$ to be one of the types I2.1--I2.15 in \cite[Table~4.1]{NZ}. 
    As $m_{h_1}(P_1) =\infty$, and $m_{h_1}(P_2) =1$, 
    Remark \ref{rem:t=2-ramification} implies that  the corresponding ramification types for $f$ are types  I2.1a--I2.15 in Table~\ref{table:wreath}. 
    Choosing $c_2\leq c_3$ and $d_2\geq d_3$, the conclusion follows when $g_X<c_2\ell-d_2$.

    {\bf Case F1A:} 
    Let  $Q_1,Q_2\in S$ be points with $m_{h_1}(Q_i)=2$, for  $i=1,2$ 
    and set $Q_3 :=Q_1^{\oline\sigma}$, and $Q_4 := Q_2^{\oline\sigma}$, so that $m_{h_1}(Q_i)=2$, for $i=3,4$. 
    Estimating \eqref{equ:main-t=2} using Lemma~\ref{lem:RS-estimates} gives:
    \begin{equation}\label{equ:2222-XY1}
    \begin{split}
     4g_X  = 
      & 2\ell(g_{Y_1}-1)   + \ell\sum_{i=1}^4\bigl(\frac{\ell}{2} - \abs{E_{h_1}(Q_i)} + \abs{\{\text{odd }r\in E_{h_1}(Q_i)\}}\bigr)  \\
     & \displaystyle + \ell \sum_{P\in Y(\K)\setminus \{Q_1,\ldots, Q_4\}}R_{h_1}(P)  -2\ell+O_\alpha(1).
    \end{split}
    \end{equation}
    By Lemma \ref{lem:Sn-two-stab},  there exist $c_3>0$ and sufficiently large $d_3$ such that the condition $g_X<c_3\ell-d_3$ and \eqref{equ:2222-XY1} force the ramification type of $h_1$ to be one of the types (F1.1)-(F1.9) in \cite[Table~4.1]{NZ}.
    As $m_{h_1}(P_1) = m_{h_1}(P_2) = 1$, Remark \ref{rem:t=2-ramification} implies that the corresponding ramification types for $f$ are types F1A.1a-F1A.9 in Table~\ref{table:wreath}. 
    Choosing $c_2\leq c_3$ and $d_2\geq d_3$, the conclusion follows when $g_X<c_2\ell-d_2$.

    {\bf Case F1B:}
    Let  $P_3\in S$ be a point with $m_{h_1}(P_3)=2$, and set $P_4=P_3^{\oline\sigma}$,  so that $m_{h_1}(P_i)=2$ for $i=1,\ldots,4$. 
    Then \eqref{equ:main-t=2} gives: 
    \begin{equation}\label{equ:F1B1}
    \begin{split}
     4g_X  = & 2\ell(g_{Y_1}-1)  
    + \ell\sum_{i=1}^4\bigl(\frac{\ell}{2} -\abs{E_{h_1}(P_i)}+ \abs{\{\text{odd }r\in E_{h_1}(P_i)\}} \bigr) \\
     & 
     +  \ell \sum_{P\in Y(\K)\setminus\{P_1,\ldots, P_4\}} R_{h_1}(P_i) + O_\alpha(1).
     \end{split}
    \end{equation}
    On the other hand by the Riemann--Hurwitz formula for $h_1$, we have 
    \begin{equation}\label{equ:F1B-RH-h1} 
    2(g_{Y_1}-1) = \sum_{P\in Y(\K)\setminus\{P_1,\ldots,P_4\}}R_{h_1}(P_i) + \sum_{i=1}^4(\frac{\ell}{2}-\abs{E_{h_1}(P_i)}).  
    \end{equation}
    Combining the latter equality with \eqref{equ:F1B1} one has
    \begin{equation}\label{equ:F4-gXY1}
    4g_X = 4\ell(g_{Y_1}-1) + \ell\sum_{i=1}^4 \abs{\{\text{odd }r\in E_{h_1}(P_i)\}} + O_\alpha(1).  
    \end{equation}
    Hence for $c_2<1/4$ and sufficiently large $d_2$, the condition $g_X<c_2\ell-d_2$ forces: 
    \begin{equation}\label{equ:F1B-odds} 0 = 4(g_{Y_1}-1)+ \sum_{i=1}^4 \abs{\{\text{odd }r\in E_{h_1}(P_i)\}}. 
    \end{equation}
    In particular, $g_{Y_1}\leq 1$. Furthermore, if $g_{Y_1}=1$, then \eqref{equ:F1B-odds} and \eqref{equ:F1B-RH-h1} imply that the ramification type of $h_1$ is $[2^{\ell/2}]$ four times. 
    In this case Lemma \ref{lem:normal-closure} shows that $\Mon(h_1)$ is solvable, contradicting the nonsolvability of $G$.
    If $g_{Y_1}=0$,
    then \eqref{equ:F1B-odds} and \eqref{equ:F1B-RH-h1} imply that the ramification type of $h_1$ 
    is $[1, 2^{(\ell-1)/2}]$ four times. In this case as well, $\Mon(h_1)$ is solvable by Lemma \ref{lem:normal-closure}, contradicting the nonsolvability of~$G$.

    {\bf Case F2:}
    Assume without loss of generality $m_{h_1}(P_1) = 3$, and $m_{h_1}(P_2) = 1$. Let $P_3\in S$ be a point with $m_{h_1}(P_3)=3$ and set $P_4:=P_3^{\oline\sigma}$, so that $m_{h_1}(P_4) = 3$. 
    Estimating \eqref{equ:main-t=2} using Lemma \ref{lem:RS-estimates} gives:
    \begin{equation}\label{equ:F2main}
    \begin{split}
     4g_X & = 2\ell(g_{Y_1}-1) + \ell\sum_{i\in \{1,3,4\}}\bigl(\frac{\ell}{3}-\abs{E_{h_1}(P_i)} +\frac{4}{3}\abs{\{r\in E_{h_1}(P_i)\,|\,(r,3)=1\}}\bigr)  \\
      &   + \ell\sum_{P\in Y(\K)\setminus\{P_1,P_3, P_4\}}R_{h_1}(P) -4\ell/3 + O_\alpha(1). 
    \end{split}
    \end{equation}
    By the Riemann--Hurwitz formula for $h_1$, we  have
    \begin{equation}\label{equ:F2-RH-h1}
    2(g_{Y_1}-1) =  \sum_{i\in\{1,3,4\}}\bigl(\frac{\ell}{3} - \abs{E_{h_1}(P_i)}\bigr) +  \sum_{P\in Y(\K)\setminus\{P_1,P_3, P_4\}}R_{h_1}(P).
    \end{equation}
    Combining the latter with \eqref{equ:F2main}, one has
    \[
    4g_X= 4\ell(g_{Y_1}-1) + \frac{4}{3}\ell\sum_{i\in\{1,3,4\}}\abs{\{r\in E_{h_1}(P_i)\,|\,(r,3)=1\}} - \frac{4}{3}\ell + O_\alpha(1).
    \]
    Hence for $c_2<1/12$ and sufficiently large $d_2$, the condition $g_X<c_2\ell - d_2$ forces:
    \[
    0 = 4(g_{Y_1}-1)+ \frac{4}{3}\sum_{i\in\{1,3,4\}}\abs{\{r\in E_{h_1}(P_i)\,|\,(r,3)=1\}} - \frac{4}{3},
    \]
    or equivalently, 
    \begin{equation}\label{equ:F2-conseq}
    4-3g_{Y_1}= \sum_{i\in\{1,3,4\}}\abs{\{r\in E_{h_1}(P_i)\,|\,(r,3)=1\}}.
    \end{equation}
    In particular $g_{Y_1}\leq 1$. There are no ramification types for $h_1$ satisfying \eqref{equ:F2-conseq} and \eqref{equ:F2-RH-h1} with $g_{Y_1}=1$. If $g_{Y_1}=0$,  
     at least two of $E_{h_1}(P_1),E_{h_1}(P_3),E_{h_1}(P_4)$  contain an entry $r$ prime to $3$, by Lemma \ref{lem:hurwitz1}. 
    The ramification types for $h_1$ satisfying the latter constraint, \eqref{equ:F2-conseq} and \eqref{equ:F2-RH-h1} are types F2.1-F2.3 in Table~\ref{table:I1A}. As $m_{h_1}(P_1) =3$ and $m_{h_1}(P_2) = 1$, Remark \ref{rem:t=2-ramification} implies that the corresponding ramification types for $f$ are types F2.1-F2.3 in Table~\ref{table:wreath}. 

    {\bf Case F3:}
    Assume  without loss of generality $m_{h_1}(P_1)  = 2$ and $m_{h_1}(P_2) =1$. Let $P_3\in S$ be a  point with $m_{h_1}(P_3)=4$, and set $P_4=P_3^{\oline\sigma}$, so that $m_{h_1}(P_4)=4$.
    Estimating \eqref{equ:main-t=2} using Lemma \ref{lem:RS-estimates} gives:
    \begin{equation}\label{equ:F3-XY1}
    \begin{split}
     4g_X  & =  
     2\ell(g_{Y_1}-1) + \ell\biggl(\frac{\ell}{2}-\abs{E_{h_1}(P_1)}+\abs{\{\text{odd }r\in E_{h_1}(P_1)\}}\biggr)     \\
        &  +  \ell\sum_{i=3}^4\biggl(\frac{\ell}{4} - \abs{E_{h_1}(P_i)} + \abs{\{r\in E_{h_1}(P_i)\,|\,r\equiv 2\,(4)\}} +\frac{3}{2}\abs{\{\text{odd }r\in E_{h_1}(P_i)\}}\biggr)  \\
        & + \ell \sum_{P\in Y(\K)\setminus\{P_1,P_3, P_4\}}R_{h_1}(P)  -\ell +  O_\alpha(1).
    \end{split}
    \end{equation}
    By Lemma \ref{lem:Sn-two-stab}, there exists $c_3>0$ and sufficiently large $d_3$ such that the conditions  $g_X<c_3\ell-d_3$ and \eqref{equ:F3-XY1} force  the ramification type of $h_1$ to be one of  types F3.1-F3.3 in \cite[Table~4.1]{NZ}. As $m_{h_1}(P_1) =2$ and $m_{h_1}(P_2)=1$, Remark \ref{rem:t=2-ramification} implies that the corresponding ramification types for $f$ are types F3.1-F3.3 in Table~\ref{table:wreath}.
    Choosing $c_2\leq c_3$ and $d_2\geq d_3$, the conclusion follows when $g_X<c_2\ell-d_2$.

    \section{Reduced tuples: a proof of Theorem \ref{thm:main-Gal} for $t\geq 3$}\label{sec:reduced}
    
    In this section we complete the proof of Proposition  \ref{cor:main-transposition}, and hence that of Theorem~\ref{thm:main-Gal}. 
    We assume the setup of \S\ref{sec:setup} with $X_0 := \mP^1$, so that $f:X\ra \mP^1$ is a covering with primitive monodromy group $G\leq S_\Delta \wr S_I$ of product type with $t=\abs{I}\geq 3$, Galois closure $\tilde X$, and $K:=G\cap S_\Delta^I$.  
    As outlined in Section \ref{sec:intro}, we shall create from $f$ a new covering $\hat f$ whose monodromy is a subgroup of $S_\Delta$. This is done in the following proposition which creates out of a product-$1$ tuple for $f$, a product-$1$ tuple in $S_\Delta$ consisting of reduced forms of branch cycles of $f$. 
    \begin{defn}
    Let $x_1,\ldots ,x_r\in G$ be a product-$1$ tuple for a group $G\leq S_\Delta\wr S_I$ of product type,  and let $O_j$ be the set of orbits of $x_j$ on $I$. Put $J:=\{1,\ldots,r\}$. 
    We call a multiset $T:=\{y_{\theta,j}\in S_\Delta\suchthat \theta\in O_j, j\in J\}$  a {\it reduced product-$1$ multiset} for $x_1,\ldots,x_r$ if it satisfies the following properties: 
    \begin{enumerate}
    \item there exists a reduced form $y_j = a_j\sigma_j$, $a_j\in S_\Delta^I$, $\sigma_j\in S_I$ of $x_j$ with representatives $\iota_\theta$, $\theta\in O_j$ such that $a_j(\iota_\theta)=y_{\theta,j}$, for all $\theta\in O_j$, $j\in J$; 
    \item the group $\langle T\rangle$ acts transitively on $\Delta$;
    \item for some choice of $\omega_{\theta,j}\in\{1,-1\}$, $\theta\in O_j$, $j\in J$ and in some ordering, the elements $y_{\theta,j}^{\omega_{\theta,j}},$ $\theta\in O_j$, $j\in J$ have product $1$.
    
    \end{enumerate}
    \end{defn}
    
    As in the setup, let $\oline\pi_0:\oline Y \ra \mP^1$ be the natural projection from the quotient $\oline Y$ by a point stabilizer in the action of $G$ on $I$. 
    Since $G$ is primitive of product type, it acts transitively on $I$ by Lemma~\ref{lem:product-type}, and hence  this action is equivalent to the $G$-action on $I\backslash G$. 
    As $K$ is the kernel of this action, the natrual projection $\pi_0:Y\ra \mP^1$ from $Y:=\tilde X/K$, is the Galois closure of $\oline\pi_0$,
    and $\oline G:=\Mon(\oline\pi_0)$ is isomorphic to $G/K$ equipped with its action on $I$. 
    Put  $t:=\deg\oline\pi_0=\abs I$, and $m:=\deg\pi_0 = [G:K]$. 
    \begin{prop}\label{prop:product}\label{lem:product}
    Assume $g_{\overline Y} = 0$, $g_Y\leq 1$, and let $x_1,\ldots, x_r$ be a product-$1$ tuple for $G$ whose images $\sigma_1,\ldots, \sigma_r\in \oline G$ satisfy\footnote{Assuming (LG$_Y$) holds simplifies the proof. Counterexamples in its absence might be possible.}:
    \begin{enumerate} 
    \item[(LG$_{Y}$)] $\sigma_i=1$ if and only if $i>s$ for some $s\in \{3, 4\}$, and $\langle \sigma_1,\sigma_2\rangle = \langle \sigma_3,\sigma_4\rangle = \oline G$ if~$s=4$, and $\sigma_1$ is a $t$-cycle if $s=3$. 
    \end{enumerate} 
    Then there exists a reduced product-$1$ multiset for $x_1,\ldots, x_r$. 
    \end{prop} 
    Since the proof of Proposition \ref{prop:product} is involved,  we first give the following example which illustrates our process of generating a product-$1$ multiset. 
    \begin{exam} \label{exam:reduced}
    Suppose $t=3$, and $x_1,x_2,x_3$ is a product-$1$ tuple for $\langle x_1,x_2,x_3\rangle$ whose images in  $G/K\leq S_3$ are $r,s,rs$, respectively, where $r = (1,2,3)$ and $s = (1,2)$. In this example, we form a reduced product-$1$ multiset for $G$.
    For this, we claim that
    \begin{enumerate}
    \item $x_1,x_2,x_3$ have reduced forms 
    $y_1=(\alpha,1,1)r$,
    $y_2=(1,b_1,b_2)s$, and 
    $y_3=(c_1,c_2,1)rs$, respectively, where $r=(1,2,3), s=(1,2)$, and $a, b_1, b_2, c_1, c_2\in S_\ell§$;
    \item
     $W:=\{a, b_1,b_2,c_1,c_2\}$ is a reduced product-$1$ multiset for $x_1,x_2,x_3$. 
    \end{enumerate}
    \begin{proof}[Proof of claim]
    Since the reduced form of $x_1$ is  $y_1:=(a,1,1)r$ for some $a\in S_\Delta$, there exists $z\in S_\Delta^I$ such that $x_1^z=y_1$, by Lemma \ref{lem:conjugate-1}. 
    For simplicity replace the tuple $x_1,x_2,x_3$ by the tuple $x_1^z,x_2^z,x_3^z$, and $G$ by $G^z$. Then $x_1=y_1$ is in reduced form. 
    Write $x_2=(\beta_1,\beta_2,\beta_3)s$ and $x_3 = (\gamma_1,\gamma_2,\gamma_3)rs$. Since $G$ is transitive on $\Delta^I$, the subgroup $K_0=\langle a,\beta_i,\gamma_i, i=1,2,3\rangle\leq S_\Delta$ is transitive. 
    Note that the orbits of $s$ are $\mu=\{1,2\}$ and $\eta=\{3\}$ and hence $x_2$ has a reduced form $y_2=(1,b_1,b_2)s$ where $b_1 = \beta_2\beta_1$ and $b_2 = \beta_3$, by Lemma \ref{lem:conjugate-1}. Similarly, $x_3$ has a reduced form  $y_3=(c_1,c_2,1)rs$ where $c_2 = \gamma_2\gamma_3$ and $c_1 = \gamma_1$. 
    
    The product-$1$ relation gives:
    \begin{align*}
        1 = x_1x_2x_3 & =   (a,1,1)r\cdot (\beta_1,\beta_2,\beta_3)s\cdot (\gamma_1,\gamma_2,\gamma_3)rs \\ & = (a,1,1)\cdot (\beta_1,\beta_2,\beta_3)^{r^{-1}}\cdot (\gamma_1,\gamma_2,\gamma_3)^{rs} \\
        & = (a,1,1)\cdot (\beta_2,\beta_3,\beta_1)\cdot (\gamma_1,\gamma_3,\gamma_2).
    \end{align*} 
    Since $(\beta_1,\beta_2,\beta_3)^{\sigma^{-1}} = (\beta_{1^\sigma},\beta_{2^\sigma},\beta_{3^\sigma})$ for every $\sigma\in S_t$, the product-$1$ relation amounts to
    the equalities:
    \begin{align*}
        a\beta_2\gamma_1 & = 1   \\ \beta_3\gamma_3 & = 1  \text{ or equivalently }\gamma_3\beta_3 = 1 \\ \beta_1\gamma_2 & = 1. 
    \end{align*}
    By iteratively inserting the equalities into the first one, we get:
    $$ 1= a\beta_2\gamma_1 = a\beta_2\beta_1\gamma_2\gamma_1 = a\cdot (\beta_2\beta_1)\cdot (\gamma_2\gamma_3)\cdot \beta_3\cdot \gamma_1 = a b_1 c_2 b_2 c_1.
    $$
    It remains to show that $W$ generates $K_0$ since then $\langle W\rangle$ is transitive. This holds since by the above equalities $\gamma_3=\beta_3^{-1}\in\langle \beta_3\rangle$ and $\beta_1 = \gamma_2^{-1}\in\langle \gamma_2\rangle$ and hence
    $$ \langle W\rangle = \langle a,\beta_2\beta_1, \beta_3,\gamma_1,\gamma_2\gamma_3\rangle = \langle a,\beta_2\beta_1, \beta_3,\gamma_1,\gamma_2\rangle = \langle a,\beta_2,\beta_3,\gamma_1,\gamma_2\rangle = K_0.
    $$
    This completes the proof of the claim. 
    \end{proof}
    Proposition \ref{prop:product} generalizes this argument under the mere condition (LG$_Y$). 
    \end{exam}
    We first note that if one of the $x_i$'s acts on $I$  as a transposition (and hence the ramification of $\oline\pi_0$ is in Table \ref{table:trans}), the product-$1$ tuple can be modified so that condition  (LG$_Y$)  holds.
    \begin{lem}\label{rem:LGY-holds}
    Assume that $\deg\oline\pi_0>2$ and that the ramification type of $\oline\pi_0$ appears in Table~\ref{table:trans}. Then there exists a product-$1$ tuple $x_1,\ldots,x_r$  for $G$ whose images in $\oline G$ satisfy (LG$_Y$), and whose ramification type coincides with the ramification type of $f$. 
    \end{lem}
    \begin{proof}
    Let $\mathcal R$ be the ramification type of $f$.
    As in Section \ref{sec:RET}, there exists a product-$1$ tuple  $x_1,\ldots,x_r$ for $G$ with ramification type $\mathcal R$. 
    Write $x_j = a_j\sigma_j$, $a_j\in S_\Delta^I$, $\sigma_j\in S_I$, $j=1,\ldots, r$. 
    Note that by swapping $(x_i,x_{i+1})\ra (x_{i+1},x_i^{x_{i+1}})$ for some $i\in\{1,\ldots,r\}$, one obtains another product-$1$ tuple for $G$ with ramification type $\mathcal R$. 
    By making such swaps, we may assume $x_1,\ldots, x_r$ is a product-$1$ tuple for $G$ with ramification type $\mathcal R$, whose images $\sigma_1,\ldots,\sigma_r\in \oline G$ are ordered by decreasing element order. In particular $\sigma_i=1$ if and only if $i>s$ for some $s$. 
    
    The ramification type corresponding to $\sigma_1,\ldots,\sigma_r$ appears in Table~\ref{table:trans}. Hence $s\in \{3,4\}$. 
    Moreover,  every ramification type with $s=3$ contains the multiset $[t]$, and among the multisets in such a ramification type, $[t]$ has the largest least common multiple.
    Hence, if the tuple $x_1,\ldots,x_r$ corresponds to such a ramification type, then the element $\sigma_1$ of highest order among the $\sigma_i$'s is a $t$-cycle, so that (LG$_Y$)~holds. 
    
    Henceforth assume $s=4$. 
    There are two ramification types in Table~\ref{table:trans} with $s=4$ and $\deg\oline\pi_0>2$, namely, $[2,1]$ four times with $t=3$, and $[2,1^2]$ twice, $[2,2]$ twice with $t=4$. In both cases, $\sigma_1,\ldots,\sigma_4$ are  involutions, and $\oline G\cong D_{2t}$. 
    Note that in both cases a direct inspection shows that every product-$1$ tuple for $D_{2t}$, $t\in\{3,4\}$, consisting of four involutions does not contain a rotation, and hence consists of reflections. 
    
    Assume first that $t=3$. As $\oline G=\langle \sigma_1,\ldots,\sigma_4\rangle$,  by possibly swapping $(x_2,x_3)\ra  (x_3,x_2^{x_3})$, we may assume that  $\sigma_1\neq \sigma_2$ and hence that $\sigma_3\neq \sigma_4$ by the product-$1$ relation. 
    Since every two distinct reflections generate $D_6$, we have $\langle \sigma_1,\sigma_2\rangle = \langle \sigma_3,\sigma_4\rangle = D_6$, so that (LG$_Y$) holds. 
    
    Now assume $t=4$.  There are two types of reflections in $D_8$, transpositions and products of two transpositions. Moreover,  two reflections generate $D_8$ if and only if they are of distinct types. The ramification type in Table~\ref{table:trans} with $s=4$ and $t=4$ consists of two elements of each type. By possibly swapping  $(x_2,x_3)\ra  (x_3,x_2^{x_3})$, we may assume $\sigma_1$ and $\sigma_2$ are of distinct type, and hence so are $\sigma_3$ and $\sigma_4$, giving $\langle \sigma_1,\sigma_2\rangle = \langle\sigma_3,\sigma_4\rangle = D_8$, so that (LG$_Y$) holds. 
    \end{proof}
    The proof of Proposition \ref{prop:product} relies on the following lemma:
    \begin{lem}\label{rem:s=4}
    Assume that  (LG$_Y$) holds with $s\in \{3,4\}$, and that  $g_{\oline Y}=0$ and $g_Y\leq 1$.  Then  
    $\abs{\Orb_I(\sigma_{s-1})} + \abs{\Orb_I(\sigma_s)} = t+1$. Moreover,  $\abs{\Orb_I(\sigma_1)} + \abs{\Orb_I(\sigma_2)} = t+1$  if $s=4$. 
    \end{lem}
    \begin{proof} 
    For $s=3$, this follows directly since $\sigma_1$ has one orbit by (LG$_Y$) and since the Riemann--Hurwitz formula for $\oline\pi_0$ gives $\sum_{i=1}^3\abs{\Orb_I(\sigma_i)}=t+2$. 
    Assume $s=4$. Since $g_Y\leq 1$ and $\pi_0:Y\ra \mP^1$ is Galois with four branch points, 
    the Riemann--Hurwitz formula for $\pi_0$ implies that the elements $\sigma_1,\ldots,\sigma_4$ are involutions. 
    Moreover, as (LG$_Y$) holds, the group $\oline G := \Mon(\oline \pi_0)$ is generated by two involutions and hence $\oline G \cong D_{m}$ is a Dihedral group. 
    Since $\oline\pi_0$ is not Galois as $g_{\oline Y}<g_Y$, the action of $\oline G$ is faithful, transitive and nonregular, and hence its the natural action of $\oline G\cong D_m$ on a regular $t$-gon for $t=m/2$. 
    Moreover, since there is no product-$1$ four tuple of involutions generating  $\oline G$ which contains a rotation, the elements $\sigma_i$, $i=1,\ldots, 4$ are reflections. If $t$ is odd, the number of orbits of every reflection is $(t+1)/2$ and the claim follows. 
    If $t$ is even, there are two types of  reflections, those with no fixed point and  those with two fixed points, which have $t/2$ or $t/2 + 1$  orbits, respectively. 
    Since two reflections that generate $D_{2t}$  are necessarily of distinct types, the assumption $\langle \sigma_3,\sigma_4\rangle = D_{2t}$ implies that $\{\abs{\Orb_I(\sigma_3)},\abs{\Orb_I(\sigma_4)}\} = \{t/2,t/2 + 1\}$. Since $g_{\oline Y} = 0$, the Riemann--Hurwitz formula for $\oline\pi_0$ then forces  $\{\abs{\Orb_I(\sigma_1)},\abs{\Orb_I(\sigma_2)}\} = \{t/2,t/2 + 1\}$.
    %
    \end{proof} 
    \begin{proof}[Proof of Proposition \ref{prop:product}]
    For $1\leq j\leq r$, write  $x_j=a_j\sigma_j$ where $a_j\in S_\Delta^I$, $\sigma_j\in S_I$, so that $\oline G \cong G/K$ is generated by $\sigma_1,\ldots,\sigma_r$. 
    
    {\bf Step I:} {\it We first define an element $z\in S_\Delta^I$ such that $x_1^z,x_2^z$ (resp., $x_1^z$) are in reduced form if $s=4$ (resp., $s=3$). }
    We define $z$ on $I$ iteratively. First,  assume $z$ is defined on a subset $I_1\subseteq I$,  and define it on a larger subset of $I$. Put $J_1:=\{1,\ldots,s-2\}$, $J_2:=\{s-1,s\}$, and $J_3:=\{s+1,\ldots,r\}$.
    Initially, set $I_1$ to be a subset of $I$ consisting of a single element  denoted by $1\in I$, set $z(1)=1$, and set $P=\emptyset$. 
    Since $\langle \sigma_j\suchthat j\in J_1\rangle$  acts transitively on $I$,
    there exists some $i\in I\setminus I_1$ and $j\in J_1$ such that $i^{\sigma_j}\in I_1$  if $s=4$ (resp., if $s=3$). 
    Then set 
    $$z(i) := a_j(i)z(i^{\sigma_j}) = (a_jz^{\sigma_j^{-1}})(i).$$ 
    Adding $i$ to $I_1$ and the pair $(j,i)\in P$,  we can repeat this process until $I_1=I$. 
    
    Writing $x_j^z=b_j\sigma_j$ for $b_j\in S_\Delta^I$, we have $b_j=z^{-1}a_jz^{\sigma_j^{-1}}$ and hence
    $$b_j(i) =  (z^{-1}a_jz^{\sigma_j^{-1}})(i) = z^{-1}(i)a_j(i)z(i^{\sigma_j})= 1,\text{ for all $(j,i)\in P$.}$$
    Moreover, it follows that for every $i\in I\setminus \{1\}$, there exists $j\in J_1$ such that $(j,i)\in P$.  
    For $s=3$, this means that $P$ contains all pairs $(1,i)$, $i\in I\setminus\{1\}$, so that $b_1(i)=1$, for $i\in I\setminus\{1\}$, and hence  $x_1^z$ is in reduced form. 
    For $s=4$, one has $\abs P=t-1$ by construction, 
    so that the complement $P'$ of $P$ in $\{1,2\}\times I$ has $t+1$ elements. 
    Since in addition the total number of orbits of $\sigma_1$ and $\sigma_2$ is $t+1$ by Lemma \ref{rem:s=4}, and since  the set $P_j:=\{i\,|\,(j,i)\in P\}$ does not contain any full orbit of $\sigma_j$ for $j\in J_1$ by construction,  
    it follows that $P'$ intersect every orbit $\theta$ of $\sigma_j$, $j=1,2$ at exactly one $\iota_\theta\in \theta$. 
    Thus $b_j(i) = 1$ for  $i\in \theta\setminus\{i_\theta\}$ and $x_j^z$, $j=1,2$ are in reduced form. 
    
    Since a reduced form of $x_j^z$ is also a reduced form of $x_j$, we may replace the elements $x_j$ by  $x_j^z$, for all $j=1,\ldots,r$. Thus we may assume that $x_j$, $j\in J_1$  are in reduced form. 
    We note that since $x_1,\ldots,x_r$ act transitively on $\Delta^I$, the subgroup $K_0 \leq S_\Delta$ generated by the elements $a_j(i)$, $j=1,\ldots,r$, $i\in I$, acts transitively on $\Delta$. Since $a_j(i)=1$ for $(j,i)\in P$ we shall henceforth consider the following multiset of generators for $K_0$:
    $$W := \{ a_j(i)\,|\, j=1,\ldots, r,i\in I,\text{ and } (j,i)\not\in P\}.$$ 
    
    {\bf Step II:} {\it Describing the product-$1$ relation for $t$-tuples.} 
    Consider the product-$1$ relation:  
    $$1 = \prod_{j=1}^r x_j = \prod_{j=1}^r a_j\sigma_j = \prod_{j=1}^r\sigma_j\cdot \prod_{j=1}^r a_j^{\tau_j},$$ 
    where $\tau_j = \prod_{k=j}^r\sigma_k$, $j=1,\ldots, r$.
    Since $\tau_1 = \prod_{j=1}^r\sigma_j=1$, we have $\prod_{j=1}^ra_j^{\tau_j}=~1$. 
    Evaluating the latter expression at $i\in I$, we get: 
    \begin{equation}
    \label{equ:explicit1}
    \prod_{j=1}^r a_j^{\tau_j}(i)=\prod_{j=1}^r a_j(i^{\tau_j^{-1}})=1, \,i\in I.
    \end{equation}

    {\bf Step III: }{\it Describing an iterative procedure to form the product-$1$ relation from the $t$ equalities in \eqref{equ:explicit1}.} 
     Let $O_j$ be the set of orbits of $\sigma_j$ on $I$. By Lemma \ref{lem:conjugate-1}, there exist reduced forms $y_j=c_j\sigma_j$, $c_j\in S_\Delta^I$, $\sigma_j\in S_I$, of $x_j$ such that \[
     c_j(\iota_\theta) = a(\iota_\theta)\cdots a(\iota_\theta^{\sigma_j^{\abs\theta-1}}),\] 
     for $\theta\in O_j$, $j\in J_2$. 
     We order the elements of $\theta\in O_j$ accordingly: $\iota_\theta,\ldots,\iota_\theta^{\sigma_j^{\abs{\theta} -1}}$. 

    At each step of the process we shall modify \eqref{equ:explicit1} with $i=1$, so that it  consists of a product over all elements $a_j^{\tau_j}(i)$, each appearing once with exponent $1$ or $-1$, where $i$ runs through a proper subset $I_2\subseteq I$, and $j\in J_1\cup J_2\cup J_3$, with the exception of pairs $(j,i)$ in a set $Q$ associated with this process. 
    
    (Step A) Initially, set $I_2$ to be the subset of $I$ consisting of the element $1$, and set $Q=\emptyset$.  
    
    Since $I_2$ is a proper subset of $I$, and $\langle \sigma_{s-1},\sigma_s\rangle=\langle \sigma_{s-1}^{\tau_{s-1}}, \sigma_s\rangle = \oline G$ is transitive on $I$,  there exist  $\mu\in I_2$ and $k\in J_2$ such that the orbit $\hat \theta$ of $\mu$ under $\sigma_{k}^{\tau_{k}}$ is not contained in $I_2$.
    For each  $\nu\in \hat\theta\setminus I_2$, equality \eqref{equ:explicit1} for $i=\nu$ gives 
    $$a_{k}(\nu^{\tau_{k}^{-1}}) = F_{k},\text{ where }F_{k} := \prod_{j=k-1}^1 a_{j}(\nu^{\tau_{j}^{-1}})^{-1}\cdot \prod_{j=s}^{k+1}a_{j}(\nu^{\tau_{j}^{-1}})^{-1}$$ is a product of the terms $a_{j}^{\tau_{j}}(\nu)^{-1}$ over all $j\not= k$, each appearing once.
    Note that since $ (i^{\tau_{k}^{-1}})^{\sigma_{k}}=(i^{(\sigma_{k}^{\tau_{k}})})^{\tau_{k}^{-1}}$ for $i\in I$ and $\hat \theta$ is an orbit of $\sigma_k^{\tau_k}$,  the set $\hat\theta^{\tau_{k}^{-1}}:=\{i^{\tau_{k}^{-1}}\suchthat i\in \hat \theta\}$ is an orbit of $\sigma_{k}$. 
    Next, 
    
    (Step B) Iteratively insert the expressions $F_{k}^{-1}\cdot a_{k}(\nu^{\tau_{k}^{-1}})$
    (resp., $a_{k}(\nu^{\tau_{k}^{-1}})F_{k}^{-1}$), which equal $1$, into equality \eqref{equ:explicit1} with $i=1$ at the left (resp., right) of 
    $a_{k}(\mu^{\tau_{k}^{-1}})$, so that the new equality at $i=1$ contains a product  $B_{k,\hat\theta}:=\prod_{i\in (\hat\theta\setminus I_2)\cup\{\mu\}}a_{k}(i^{\tau_{k}^{-1}})$, ordered so that the indices $i^{\tau_{k}^{-1}}$ are increasing with respect to the above ordering of the orbit $\hat\theta^{\tau_{k}^{-1}}$ of $\sigma_k$.  
    
    (Step C) Add $\hat\theta\setminus I_2$ to $I_2$; add the pair $(k,\nu)$ to $Q$ for every $\nu\in \hat\theta\setminus I_2$; 
    and replace in $W$ the elements $a_{k}(i^{\tau_{k}^{-1}})$ , $i\in \{\mu\}\cup (\hat \theta\setminus I_2)$,  by the element $B_{k,\hat \theta}$. 
    
    We repeat steps (B)-(C) until $I_2= I$, in which case the process terminates. 
    
    Note that after preforming step (B)  with an orbit $\hat \theta$ of $\sigma_{k}^{\tau_{k}}$, the factors in equality \eqref{equ:explicit1} contain  $a_{k}^{\tau_{k}}(i)$, $i\in\hat \theta$, and therefore the order of these elements and the products $B_{k,\hat\theta}$ remain unchanged in further steps of the process.
    For each pair $(j,i)$, $j\in J_2$, $i\in I$, for which $a_j^{\tau_j}(i)$ appears at the resulting equality \eqref{equ:explicit1} but not as part of a product $B_{j,\hat \theta}$, denote  $B_{j,\{i\}}:=a_{j}^{\tau_j}(i)$, for $j\in J_2$. 
    We shall see that in such cases $\{i\}$ is a length $1$ orbit of $\sigma_j^{\tau_j}$.

    {\bf Step IV}: {\it We show that the multiset $T$ consisting of the elements $a_j(\iota_\theta)$, $\theta\in O_j$, $j\in J_1\cup J_3$, and $c_j(\iota_\theta)$, $\theta\in O_j$, $j\in J_2$ is a reduced product-$1$ multiset for $x_1,\ldots,x_r$, where the product-$1$ relation is given by the equality \eqref{equ:explicit1}, at  $i=1$, resulting at the end of the process.}
    
    We claim that each element $c_j(\iota_\theta)$, $\theta\in O_j$, $j\in J_2$  appears  in the  resulting equality \eqref{equ:explicit1}, with $i=1$, at the end of Step III. 
    Let $\hat Q$ denote the complement of $Q$ in $\{(j,i)\suchthat j\in J_2, i\in I\}$. 
    Since each $i\in I\setminus\{1\}$ is added to $I_2$ once, we have $\abs Q = t-1$ and hence $\abs{\hat Q}=t+1$.  
    Since (LG$_Y$) holds, Lemma \ref{rem:s=4} implies that
    \begin{equation}\label{equ:J2} 
    \abs{\Orb_I(\sigma_{s-1}^{\tau_{s-1}})} + \abs{\Orb_I(\sigma_s^{\tau_s}))} = t+1=\abs{\hat Q}.
    \end{equation}
    Since by construction, for every orbit $\hat\theta$ of $\sigma_j^{\tau_j}$, $j\in J_2$, there exists at least one $\mu\in \hat\theta$ such that $(j,\mu)\in \hat Q$, \eqref{equ:J2}  implies that every such orbit $\hat\theta$ contains exactly one $\mu\in\hat\theta$ such that $(j,\mu)\in \hat Q$. Thus,  $B_{j,\hat\theta}$ is a product of elements $a_{j}^{\tau_j}(i)$  running over all $i\in\hat\theta$, for every orbit $\hat\theta$ of $\sigma_j^{\tau_j}$, $j\in J_2$. 
    In particular, the elements $B_{j,\{i\}} = a_{j}^{\tau_j}(i)$ defined in the end of Step III are products over length $1$ orbits $\{i\}$ of $\sigma_j^{\tau_j}$, $j\in J_2$, and at each step B we have $\hat\theta\cap I_2 =\{\mu\}$. 
    Since $\theta:=\hat\theta^{\tau_j^{-1}}$ is an orbit of $\sigma_j$ as above, we have $B_{j,\hat\theta} = c_j(\iota_\theta)$ by  definition of $B_{j,\hat\theta}$, for every orbit $\hat\theta$ of $\sigma_j^{\tau_j}$, $j\in J_2$. 
    Since there is a one to one correspondence between orbits of $\sigma_j$ and $\sigma_j^{\tau_j}$, this shows that each element $c_j(\iota_\theta)$, $\theta\in O_j$, $j\in J_2$ appears once in the resulting product \eqref{equ:explicit1} for $i=1$, proving the claim.
    Since each of the equations \eqref{equ:explicit1} for $i\in I$ is inserted exactly once into the equation with $i=1$, it follows that  the resulting equality \eqref{equ:explicit1} is a product-$1$ relation consisting of the elements in $T$, each appearing once either with exponent  $1$ or with exponent $-1$. 
    
    
    It remains to prove the transitivity of $T$. 
    We claim that the resulting set $W$ after all steps C generates the same subgroup $K_0$ and hence is transitive. 
    By definition of $B_{k,\hat\theta}$ and since we have shown that $\hat\theta\cap I_2 =\{\mu\}$ at each step B, the element $a_{k}^{\tau_{k}}(\mu)$ is in the group generated by $B_{k,\hat\theta}$ and the elements $a_{k}^{\tau_{j}}(i)$, $i\in \hat\theta\setminus \{\mu\}$, and hence by replacing it with $B_{k,\hat\theta}$, one still has $\langle W\rangle =K_0$. By equality~\eqref{equ:explicit1},  $a_{k}^{\tau_{k}}(i)$ is in the group generated by the elements $a_{j}^{\tau_{j}}(i)$, $j\neq k$ (which were not removed from $W$),  and hence after removing the elements $a_{k}^{\tau_{k}}(i)$, $i\in \hat\theta \setminus \{\mu\}$ from $W$ in step C,  one still has $\langle W\rangle =K_0$.  In total after each step~C, one  still has $\langle W\rangle=K_0$, proving the claim and hence completing the proof. 
    \end{proof}

    To the covering $f:X\ra\mP^1$, one associates the natural projections $h_i:\tilde X/K_i\ra \tilde X/K$ and $\pi_0:\tilde X/K\ra \mP^1$, where $K_i$ is a point stabilizer in the action of $K$ on the $i$-th copy $\Delta$, for $i\in I$. Recall that by Remark \ref{rem:setup}.(1), the fiber product of $h_i$ and $h_j$ is irreducible for every two distinct $i,j\in\{1,\ldots,t\}$.
    The following proposition shows that if every pair of coverings $h_i,h_j$, $i\neq j$ admits an irreducible fiber products of genus  $<\alpha\ell$, and there is branch cycle of $f$, a power of which acts on $I$ as a transposition, then the branch cycles of $f$ are conjugate to elements $y\sigma\in S_\Delta\wr S_I$, $y\in S_\Delta^I$, $\sigma\in S_I$ such that $y(i)$, $i\in I$ are of almost Galois type with bounded error in the sense of Definition \ref{def:almost-Gal-branch-cycles}.

    \begin{prop}\label{cor:gZ-bound} 
    For every constant $\alpha>0$, there exist constants $\hat \eps:=\hat\eps_{\alpha,t}$ and $N_\alpha$ satisfying the following property. 
    Let $x_j\in S_\Delta\wr S_I$, $j\in J$ be a product-$1$ tuple for $f$, and assume the coverings $h_i:Y_i\ra Y$, $i\in I$ corresponding to $f$ have irreducible fiber products $Y_i\#_Y Y_j$, $i\neq j$ of genus  $< \alpha\ell$.
    Let 
    $O_j$ be the set of orbits of $x_j$ on $I$,
    let $e_j$ be the order of the image of $x_j$ in $S_I$, 
    and put $z_j:=x_j^{e_j}\in S_\Delta^I$. 
    
    Assume that there exists a reduced product-$1$ multiset $y_{\theta,j}$, $\theta\in O_j$, $j\in J$ for the tuple $x_1,\ldots,x_r$. 
    Finally assume that $O_k$ consists of odd length orbits and a single length $2$ orbit for some $k\in J$. 
    Then  $y_{\theta,j}$ and $z_j(\iota_\theta)$  
    are of almost Galois type with error at most $\hat \eps$ for finite types, and with entry bound $N_\alpha$ for infinite types, such that  $m(y_{\theta,j})=m(z_j(\iota_\theta))$ if $x_j$ does not act on $I$ as a transposition or if $\abs{\theta}>1$, and $m(y_{\theta,j})=2m(z_j(\iota_\theta))$ otherwise, for $\theta\in O_j$, $j\in J$. 
    \end{prop}
    In fact, we show that  $N=N_\alpha$ is simply the entry bound corresponding to the constant $\alpha$ in Definition \ref{def:almost-Gal}. 
    Recall that $m=\deg\pi_0$, and let $\eps_j=\eps_{j,\alpha}$ (resp.~$\eps=\eps_\alpha$) denote the error (resp.~total error) of $h_i$, $i\in I$ over the branch point of $x_j$. (Note that the error is independent of $i\in I$.) We show that $z_j(\iota_\theta)$ is of almost Galois type $m_j$ with error at most $\eps_j$, for all $\theta\in O_j$, $j\in J$. We also show the same error occurs for $y_{\theta,j}$ if $x_j$ does not act on $I$ as a transposition, or if $\abs\theta> 1$. Otherwise, we show that $y_{\theta,j}$ is almost Galois with error at most $6m_jm(\eps+N)$,
    for all $\theta\in O_j$, $j\in J$.
    
    Since the proof is subtle as it considers many cases, we first demonstrate its outline in the following example:
    \begin{exam}\label{exam:almost-Gal} 
    Following Example \ref{exam:reduced}, let $x_1,x_2,x_3$ be a product-$1$ tuple for a covering $f$ with monodromy group $G = S_\ell\wr S_3$ with reduced forms $y_1=(a,1,1)r$, $y_2=(b_1,1,b_2)s$, and $y_3=(c_1,c_2,1)rs$, respectively, such that  $W:=\{a,b_1,b_2,c_1,c_2\}$ is a reduced product-$1$ multiset. 

    Assume that the associated coverings $h_i:Y_i\ra Y$, $i\in I$ admit irreducible fiber products $Y_i\#_Y Y_j$, $i\neq j$ of genus  $< \alpha\ell$, and the list of values $m_{h_1}(Q)$, $Q\in Y$ which are greater than $1$, is $\{3,3,3\}$. We then claim that $a,b_1,b_2,c_1,c_2$ are permutations of almost Galois type of  error at most $30\eps$ each, where $\eps:=\eps_{\alpha,3}$ is the error for points of type $3$. Here, a permutation is of ``almost Galois type" in the sense of Definition \ref{def:almost-Gal-branch-cycles}. We note that  the resulting Galois types in $W$ which are greater than $1$ are then $\{2,3,6\}$ and not $\{3,3,3\}$. 
    \begin{proof}[Proof of claim (from Example \ref{exam:almost-Gal})]
    We first describe the ramification of $h_1$. Write $P_1,P_2,P_3$ for the branch points of $f$ with branch cycles $x_1,x_2,x_3$, respectively. Then $\pi_{0}^{-1}(P_1)$ consists of two points, $\pi_0^{-1}(P_i)$ consists of three points for each $i\in\{2,3\}$, and $\pi_0^{-1}(P)$ consists of six points for any other point $P$. Since $m_{h_1}(Q)$ is the same for every point $Q\in \pi_0^{-1}(P)$ over a given point $P$ of $X_0=\mP^1$ by Remark \ref{lem:same-m_P}, and since the list of greater than $1$ values of $m_{h_1}(Q)$ is given to be $\{3,3,3\}$, it follows that $m_{h_1}(Q)>1$ for all $Q\in \pi_0^{-1}(P_0)$ for exactly one point $P_0$ of $\mP^1$, and that point $P_0$ is either $P_2$ or $P_3$. Since the argument is the same in both cases, we shall assume $P_0=P_2$, and denote $\pi_0^{-1}(P_2) = \{Q_1,Q_2,Q_3\}$, so that $m_{h_1}(Q_i) = 3$ for $i=1,2,3$, and $m_{h_1}(Q) = 1$ for any other point $Q$ of $Y$. 
    
    By Lemma \ref{lem:f-to-h}, the points $Q_1,Q_2,Q_3$  correspond to the double cosets $\langle x_2\rangle \backslash (S_\ell\wr S_3)/S_\ell^3$. Since $1,r,r^2$ is a set of representatives for these double cosets, the lemma implies that, up to reordering the points, the branch cycles of $h$ over $Q_1, Q_2,Q_3$ are conjugate in $S_\ell^3$ to $$y_2^2=(b_1,b_1,b_2^2),\,\, (y_2^2)^r=(b_2^2,b_1,b_1),\,\, (y_3^2)^{r^2} = (b_1,b_2^2,b_1),$$ respectively.  Thus the branch cycles of $h_1$ over $Q_1,Q_2,Q_3$ are $b_1,b_2^2,b_1$, respectively, and hence these are of almost Galois type $3$ with error at most $\eps=\eps_{\alpha,3}$. 
    Similarly, the branch cycle of $h_1$ over each of the two preimages in $\pi_0^{-1}(P_1)$ is $a$, and the branch cycles over points in $\pi_0^{-1}(P_3)$ are $c_1,c_1,c_2^2$. Thus  $a,c_1,c_2^2$ are of almost Galois type $1$ with error at most $\eps_{\alpha,1}\leq \eps$. 
    
    The remaining difficulty is the main one: showing that $b_2$ and $c_2$ are of almost Galois type. Since $W$ is a product-$1$ multiset, there is a covering $\hat f:\hat Y\ra\mP^1$ whose branch cycles are the elements of $W$. The Riemann--Hurwitz formula  for $\hat f$ then gives:
    \begin{equation}\label{equ:exam:ree} \sum_{x\in W}\left(\ell-\abs{\orb(x)}\right)\geq 2\ell-2.
    \end{equation}
    On the other hand, since $a,b_1,c_1$ are permutations of almost Galois types $1,1,3$, resp., with error at most $\eps$, we have $\ell-\abs{\orb(a)}<\eps$, and $\ell-\abs{\orb(c_1)}<\eps$, and $\ell-\abs{\orb(b_1)}<2\ell/3+\eps$. To estimate the number of orbits of $b_2$ and $c_2$. 
    Denote  by $u_3$ and $u_6$ the number of orbits of $b_2$ with length $3$ and length $6$, respectively, and by $v_1$ and $v_2$ the number of length $1$ and length $2$ orbits of $c_2$, respectively. 
    Since $c_2^2$ is of almost Galois type $1$ with error at most $\eps$, we have $v_1+2v_2\geq \ell-\eps$ and $\abs{\orb(c_2)}\geq v_1+v_2$, the combination of which gives $\ell-\abs{\orb(c_2)}\leq v_2+ \eps$. 
    Similarly, since $b_2^2$ is of almost Galois type $3$, we have
     $\ell-\abs{\orb(b_2)}\leq  2u_3 + 5u_6 + \eps$. Combining the above bounds with \eqref{equ:exam:ree}, we get:
     \begin{equation}\label{equ:ree-bounded}
     2\ell-2\leq \sum_{x\in W}\left(\ell-\abs{\orb(x)}\right) < 2u_3+5u_6 + v_2 + 2\ell/3 + 5\eps.
     \end{equation}
     Furthermore, one has (A) $v_2\leq \ell/2$ and since $3u_3+6u_6\leq \ell$ and $u_6\leq \ell/6$ we also have (B) $2u_3+5u_6\leq 2(\ell/3-2u_6)+5u_6\leq  5\ell/6$. 
     Plugging  inequalities (A) and (B) into  \eqref{equ:ree-bounded} gives:
     \begin{align*}
         v_2 & > 4\ell/3-(2u_3+5u_6)-5\eps > \ell/2-5\eps  \text{ and } \\
         2u_3 + 5u_6 & > 4\ell/3-v_2-5\eps > 5\ell/6-5\eps. 
     \end{align*}
     The first inequality implies that $c_2$ is of almost Galois type $2$ with error at most $10\eps$. 
     Since $u_3\leq \ell/3 - 2u_6$, the second inequality implies that $2\ell/3 + u_6> 5\ell/6-5\eps$ and hence $u_6> \ell/6-5\eps$. Thus $b_2$ is of almost Galois type $6$ with error at most $30\eps$. 
    \end{proof}
    \end{exam}
    The proof of Proposition \ref{cor:gZ-bound} generalizes the proof of the claim to cases where the list of $m_{h_1}(P)$ values is general and not merely $\{3,3,3\}$.

    \begin{proof}[Proof of Proposition \ref{cor:gZ-bound}]
    {\bf Step I:} {\it Describing the branch cycles of $h_i$, $i\in I$. }
    Let $P_j$ be the point of $\mP^1$ corresponding to the branch cycle $x_j$, for $j\in J$, and $\tilde h:\tilde X\ra\tilde X/K$  the natural projection. 
    The points $Q_{j,\tau}$ in $\pi_0^{-1}(P_j)$ are in one to one correspondence with double cosets $K\tau\langle x_j\rangle$ in $K \backslash G /\langle x_j\rangle$ by Lemma \ref{lem:f-to-h}, so that $(x_j^\tau)^{e_j}=z_j^\tau$ is a branch cycle of $\tilde h$ over $Q_{j,\tau}$  for $j\in J$. 
    Since $K_i$ is a point stablizer in the action of $K$ on the $i$-th copy $\Delta^{(i)}$ of $\Delta$, 
    the monodromy group $G_i$ of $h_i:\tilde X/K_i\ra \tilde X/K$ is isomorphic to the projection of $K\leq S_\Delta^I$ to $S_{\Delta^{(i)}}$, 
    for $i\in I$. 
    Identifying $G_i$ with this subgroup of $S_{\Delta^{(i)}}$,  the image $z_j^\tau(i)\in S_{\Delta^{(i)}}$ of $z_j^\tau\in S_\Delta^I$ is a branch cycle of $h_i$ over $Q_{j,\tau}$, for every coset $K\tau\langle x_j\rangle \in K \backslash G /\langle x_j\rangle$, $j\in J$, and $i\in I$. Since $h_i:Y_i\ra Y$, $i\in I$ admit irreducible fiber products $Y_i\#_Y Y_j$, $i\neq j$ of genus $< \alpha\ell$, their branch cycles  $z_j(i)$, $i\in I$, $j\in J$ are all of almost Galois type with constant $\alpha$ by Corollary \ref{cor:DZ}. As remarked at the end of Definition~\ref{def:almost-Gal}, the types $m_j=m(z_j(i))$ and errors $\eps_j=\eps_{\alpha,m_j}$ (resp., entry bounds $N$) are independent of $i\in I$, for each $j\in J$. 
    
    {\bf Step II:} {\it Describing the sets $O_j$, $j\in J$. }
    As in the beginning of the section, let $\oline\pi_0:\oline Y\ra \mP^1$ be the natural projection from the quotient $\oline Y$ of $Y$ by a point stabilizer in the action on $I$, so that its monodromy group $\oline G$ is isomorphic to $G/K$ equipped with its action on $I$. Let $\sigma_j$ be the image of $x_j$ in $\oline G$, and put $m:=\abs{G/K}=\deg\pi_0$.
    In particular, the ramification multiset $E_{\oline\pi_0}(P_j)$ coincides with the cycle structure of $\sigma_j$, for $j\in J$.  
    Since $g_Y\leq 1$ by Corollary \ref{cor:DZ},
    and since by assumption $O_k$ consists of odd length orbits and a single orbit of length $2$ for some $k$, we may apply Lemma~\ref{lem:no-trans} to 
    obtain the possibilities for the ramification multiset $E_{\oline \pi_0}(P_j)$, and hence for the cycle structure of $\sigma_j$, for $j\in J$.  The lemma implies that the ramification type of $\oline\pi_0$ is in Table~\ref{table:trans}. Letting $S=\{j\suchthat \sigma_j\text{ is a transposition}\}$, letting $F$ be the multiset $\{\theta\in O_j\suchthat j\in S,\abs\theta=1\}$ of fixed points, and $t:=\abs{I}$, it follows that 
    $\abs\theta = e_j$ for every $\theta\in O_j$, $j\in J\setminus S$ in case $\oline G\not\cong S_4$,  and 
    \begin{enumerate}
    \item $\abs F=4$ with $t\leq 6$ if $g_Y=1$;
    \item $\abs F=2$ with $t\in\{3,4\}$ if $g_Y=0$. 
    \end{enumerate} 
    
    {\bf Step III:} {\it Proof for $j\in J\setminus S$ when $\oline G\not\cong S_4$}. Let $y_j=b_j\sigma_j$, $b_j\in S_\Delta$ be a reduced form of $x_j$ with representatives $\iota_\theta$, $\theta\in O_j$, such that $b_j(\iota_\theta) = y_{\theta,j}$, for $\theta\in O_j$, $j\in J$.  
    Since $x_j$ and $y_j$ are conjugate by an element of $S_\Delta^I$, the elements $z_j=x_j^{e_j}$ and $d_j:=y_j^{e_j}\in S_\Delta^I$ are also conjugate in $S_\Delta^I$, and hence $d_j(i)$ is of almost Galois type $m_j$ with error at most $\eps_j$ if $m_j<\infty$ (resp., entry bound $N$ if $m_j=\infty$), for all $i\in I$, $j\in J$. 
    Since  the orbits of $\sigma_j$ are all of length $\abs\theta=e_j$ when $\oline G\not\cong S_4$, we have 
    $$d_j(i)=\prod_{u=0}^{e_j-1}b_j^{\sigma_j^{-u}}(i) = \prod_{u=0}^{e_j-1}b_j(i^{\sigma_j^u}) =  b_j(\iota_\theta) = y_{\theta,j},$$ 
    for every $i\in \theta$, $\theta\in O_j$, $j\in J\setminus S$. 
    Thus, $y_{\theta,j}$ is of almost Galois type $m(y_{\theta,j}) = m(d_j(\iota_\theta)) = m(z_j(\iota_\theta))=m_j$ with error at most $\eps_j$ if $m_j<\infty$ (resp., entry bound $N$ if $m_j=\infty$), for every $\theta\in O_j$, $j\in J\setminus S$.

    {\bf Step IV:} {\it Proof for  $j\in S$ with $m_j$ infinite or even}. Since the ramification of $\oline\pi_0$ is as in Table~\ref{table:trans}, in this case $e_j=2$ and $\abs\theta=1$ or $2$, for every $\theta\in O_j$. If $\abs\theta=2$, then the same argument as in step III shows that $y_{\theta,j}=b_j(\iota_\theta)$ is of almost Galois type $m(y_{\theta,j}) = m(d_j(\iota_\theta)) = m(z_j(\iota_\theta))=m_j$ with error at most $\eps_j$ if $m_j<\infty$ (resp., entry bound $N$ if $m_j=\infty$), for $\theta\in O_j$, $j\in S$. For $\theta\in O_j$, $j\in S$ with $\abs\theta=1$, we have $d_j(\iota_\theta)=b_j(\iota_\theta)b_j(\iota_\theta^{\sigma_j}) = b_j(\iota_\theta)^2$, and hence
     $b_j(\iota_\theta)^2$ is of almost Galois type $m_j$ with error $\eps_j$ if $m_j<\infty$ (resp., entry bound $N$ if $m_j=\infty$). If $m_j=\infty$ this implies that the number of orbits of $b_j(\iota_\theta)$ is at most $N$, proving the assertion in this case.
     
    Since the square of a length-$(u\cdot m_j)$ cycle in $S_\ell$  is a product of cycles of length $um_j/(2,um_j)$,  this square is a product of $m_j$-cycles if and only if either [$u\in\{1,2\}$ and $m_j$ is odd], or [$u=2$ and  $m_j$ is even]. 
    Thus, letting $n_{\theta,u}$ denote the number of length-$(um_j)$ orbits of $b_j(\iota_\theta)$, 
    we have 
    \begin{equation}\label{equ:n-bounds}
        \begin{split}
            (\ell-\eps_j)/m_j & \leq n_{\theta,1}+2n_{\theta,2}\leq \ell/m_j \text{ if $m_j$ is odd, and } \\
            (\ell-\eps_j)/m_j & \leq 2n_{\theta,2}\leq \ell/m_j\text{ if $m_j$ is even.}
        \end{split}
    \end{equation} The assertion then also follows when $m_j$ is even. 
    It remains to treat the cases where $m_j$ is finite and odd. For this we use:
    
    {\bf Step V:} {\it The Riemann--Hurwitz formula for the reduced multiset.}
    Since $y_{\theta,j}$, $\theta\in O_j$, $j\in J$ is reduced product-$1$ multiset, it generates a transitive subgroup of $S_\ell$, and hence By Riemann's existence theorem there exists a covering  $\hat f:\hat Y\ra\mP^1$ with these branch cycles. To prove that  $y_{\theta,j}=b_j(\iota_\theta)$ is of almost Galois type in the remaining case where $m_j$ is odd, $j\in S$ and $\theta\in O_j\cap F$, we bound its Riemann--Hurwitz contribution to $\hat f$ using \eqref{equ:n-bounds}:  
    \begin{equation}\label{equ:reduced-contr}
    \begin{split}
    \ell-\abs{\Orb_\Delta(y_{\theta,j})}& \leq  \ell - n_{\theta,1} - n_{\theta,2} =  (\ell - n_{\theta,1}-2n_{\theta,2}) + n_{\theta,2}  \\
    & \leq  \ell(1-1/m_j) + n_{\theta,2} + \eps_j/m_j.
    \end{split}
    \end{equation}
    Throughout the rest of the proof we use the convention $1/m_j = 0$ in case $m_j = \infty$. 
    When $\oline G\not\cong S_4$,  steps III and IV imply that $y_{\theta_j}$ are of almost Galois type $m_j$ with error $\leq \eps_j$  for $\theta\in O_j$, $j\in J$ such that either $j\in J\setminus S$, or $\abs\theta>1$, or $m_j$ is even or infinite. Hence similarly to \eqref{equ:reduced-contr}, in these cases we have:
    \begin{equation}\label{equ:non-special}
    \ell-\abs{\Orb_\Delta(y_{\theta,j})} \leq  
    \begin{cases}
    \ell(1-1/m_j) + \eps_j/m_j & \text{if }m_j<\infty \\
    \ell(1-1/m_j) & \text{if }m_j=\infty.
    \end{cases}
    \end{equation}
     

    Plugging \eqref{equ:reduced-contr} and \eqref{equ:non-special} to the Riemann--Hurwitz formula for $\hat f$ gives: 
    \begin{equation}\label{equ:indy}
    \begin{split}
     2\ell-2 \leq &  \sum_{j\in J} \sum_{\theta\in O_j}(\ell-\abs{\Orb_\Delta(y_{\theta,j})}) \\
     \leq & \sum_{j\in J} \abs{O_j}\ell(1-\frac{1}{m_j}) + \sum_{\{j\in J\suchthat m_j<\infty\}}\abs{O_j}\frac{\eps_j}{m_j} + \sum_{\{j\in S\suchthat m_j\text{ is odd}\}}\sum_{\theta\in F\cap O_j}n_{\theta,2}  \\ 
     \end{split} 
     \end{equation}
    Set $\delta_{0}:=\sum_{j\in J,\, m_j<\infty}\abs{O_j}\eps_j/m_j$ so that  $\delta_0<t\eps$ (since $\abs{O_j}<t$ and $\sum_j\eps_j\leq \eps$). 
    
    {\bf Step VI:} {\it Separation into cases according to almost Galois types.}
    We claim that $n_{\theta,2}\geq \ell/(2m_j) - 3m(\eps+N)$, as this would show that $y_{\theta,j}$ is of almost Galois type $2m_j$, and error at most $6mm_j(\eps+N)$ for all $j\in S$, $\theta\in O_j\cap F$, completing the proof when $\oline G\not\cong S_4$. 
     
    At first consider the case $g_Y=1$ in which case $\abs F=4$ and $t\leq 6$. 
    Since $g_Y=1$ and since $h_i:Y_i\ra Y$, $i\in I$ admit irreducible fiber products $Y_i\# Y_j$, $i\neq j$ of genus  $< \alpha\ell$,
    we also have $m_j=1$ for $j\in J$ by Corollary \ref{cor:DZ1}. 
    Hence 
    $\sum_{\theta\in F} n_{\theta,2} > 2\ell - \delta_0$
    by \eqref{equ:indy}. 
    Since  $\abs F\leq 4$ and $n_{\theta,2}\leq \ell/2$, we deduce that  
    $n_{\theta,2}> \ell/2-\delta_0> \ell/2 -t\eps$ for $\theta\in F$. 
    As $t\leq m$, 
    this shows that $y_{\theta,j}$ is of almost Galois type $2$ with error at most $2t\eps\leq 2m\eps$ for all $\theta\in F\cap O_j$, $j\in J$, as desired.
    
    Henceforth assume $g_Y=0$. 
    We first modify \eqref{equ:indy} using the Riemann--Hurwitz formula for $h_i$, $i\in I$. 
    Since $Q_{j,\tau}$ is of almost Galois type $m_j$ under $h_i$, the cardinality  $\abs{E_{h_i}(Q_{j,\tau})}$ is at most $(\ell-\eps_j)/m_j + \eps_j$ if $m_j<\infty$ and at most $N$ if $m_j=\infty$, for $i\in I$. Hence 
    \begin{equation}\label{equ:non-special2}
    R_{h_i}(Q_{j,\tau}) = \ell - \abs{E_{h_i}(Q_{j,\tau})} \geq 
    \begin{cases} \ell(1-1/m_j) -\eps_j(1-1/m_j) & \text{if }m_j<\infty \\
    \ell(1-1/m_j) - N & \text{if }m_j=\infty,
    \end{cases}
    \end{equation}
    for every $j\in J$, and coset $\tau\in K\backslash G/\langle x_j\rangle$. 
    Since  $g_{Y_i} \leq \alpha+1$ by Remark \ref{rem:almost-Galois-NZ}, and since $\abs{\pi_0^{-1}(P_j)}=m/e_j$ as $\pi_0$ is Galois,  
    the Riemann--Hurwitz formula for $h_i$ and \eqref{equ:non-special2} give: 
    \begin{equation}\label{equ:h-almost-Galois} 
    \begin{split}
    2(\ell+\alpha) & \geq 2(\ell+g_{Y_i}-1)  = \sum_{j\in J} R_{h_i}(\pi_0^{-1}(P_{j})) \geq \sum_{j\in J} \frac{m}{e_j}\ell(1 - \frac{1}{m_j}) - \delta_1, 
    \text{ where } \\
     \delta_1 & := \sum_{j\in J\suchthat m_j<\infty}\frac{m}{e_j}\eps_j(1-1/m_j) + \sum_{j\in J\suchthat m_j=\infty}\frac{m}{e_j}N. 
    \end{split}
    \end{equation}
    Note that since the ramification of $\pi_0$ is as in Table~\ref{table:trans} and $g_Y=0$, one has $\sum_{j\in J}(1/e_j)< 2$ and hence also $\delta_1< 2m(\eps + N)$. 
    The combination of \eqref{equ:indy} and \eqref{equ:h-almost-Galois} then gives:
    \begin{equation}\label{equ:co} 
    \begin{split}
      \sum_{\{j\in S\suchthat m_j\text{ is odd}\}}\sum_{\theta\in F\cap O_j}n_{\theta,2} & \geq 2\ell-2 - \sum_{j\in J}\abs{O_j}\ell(1-\frac{1}{m_j}) -\delta_0 \\ 
      & \geq -2 -2\alpha - \delta_0 - \delta_1+ \sum_{j\in J}(\frac{m}{e_j}-\abs{O_j})\ell(1-\frac{1}{m_j}) \\
      & = \sum_{j\in J}(\frac{m}{e_j}-\abs{O_j})\ell(1-\frac{1}{m_j}) - \delta_2,
    \end{split}
    \end{equation} 
     where $\delta_2:=2+2\alpha+\delta_0+\delta_1$.
    Note that as $\delta_0<t\eps$ and 
    $\delta_1<2m(\eps+N)$, one has $\delta_2<2(\alpha+1)+t\eps+2m(\eps+N)$.
    As $2(\alpha+1)<\eps$ and $t+1<m$ in all cases in Table~\ref{table:trans} with $t>2$, one further has $\delta_2<2(\alpha+1)+t\eps+2m(\eps+N)<3m(\eps+N)$. 
    
    We estimate \eqref{equ:co} for each of the possibilities for the multiset:
    $$M_{h}  := \{ m_j=m_{h_i}(Q_{j,\tau})  \suchthat m_j>1, \tau\in K \backslash G /\langle x_j\rangle,  j\in J\},$$
    which, as remarked earlier, is independent of the choice of $i\in I$.
    Note that as $m_j=m_{h_i}(Q_{j,\tau})$ is also independent of $\tau$, it appears in $M_h$ at least $\abs{K\backslash G/\langle x_j\rangle}=m/e_j$ times. 
    By going over  ramification types for $\pi_0$  corresponding  to the possibilities for the ramification of $\oline\pi_0$ in Table~\ref{table:trans}, Corollary~\ref{cor:DZ} shows that 
    $M_h = \{m_\eta\suchthat \tau\in K\backslash G/\langle x_\eta\rangle\}$ for some $\eta\in J$, and there are three possibilities:
    \begin{enumerate} 
    \item[(I1)] $M_h=\{\infty,\infty\}$, and $E_{\pi_0}(P_\eta)=[t,t]$, and $E_{\oline\pi_0}(P_\eta)=[t]$ for $t=3$ or $4$;
    \item[(F1)] $M_h = \{3,3,3\}$, and $E_{\pi_0}(P_\eta)=[2^3]$, and $E_{\oline\pi_0}(P_\eta)=[2,1]$ for $t=3$; 
    \item[(F2)] $M_h = \{2,2,2,2\}$, and $E_{\pi_0}(P_\eta)=[2^4]$, and $E_{\oline\pi_0}(P_\eta)$ is $[2,1^2]$ or $[2^2]$ for $t=4$. 
    \end{enumerate}
    In all three cases $\oline G$ is Dihedral of order $2t$ and in particular the (so far excluded) case with $\oline G\cong S_4$ does not occur. Note that as $E_{\oline\pi_0}(P_\eta)$ is the cycle structure of $\sigma_j$, we have  $\eta\not\in S$ in case (I1) and $\eta\in S$ in case (F1).  
    It follows that $m_j=1$ for $j\in J\setminus\{ \eta \}$, and $\abs F=2$ as $\oline G\not\cong S_4$.  Thus \eqref{equ:co} reduces to 
    \begin{equation}\label{equ:co2}
     \sum_{\{j\in S\suchthat m_j\text{ is odd}\}}\sum_{\theta\in F\cap O_j}n_{\theta,2} \geq (\frac{m}{e_\eta}-\abs{O_\eta})(1-\frac{1}{m_\eta})\ell - \delta_2.
     \end{equation}
    
    In case (I1),  $\eta\not\in S$, $m_\eta=\infty$,  $e_\eta=m/2$, and $x_\eta$ has one orbit, giving $\abs{O_\eta}=1$. Hence  $\sum_{\theta\in F}n_{\theta,2}\geq \ell -  \delta_2$ by \eqref{equ:co2}. Since in addition $n_{\theta,2}\leq \ell/2$ for each $\theta\in F$ and $\abs F=2$,  we have $n_{\theta,2}\geq \ell/2-\delta_2$ for all $\theta\in F$, as desired. 
    
    In case (F1),  $\eta\in S$, $m_\eta=3$, $m=6$,  $e_\eta=2$, and $\abs{O_\eta}=2$. Hence \eqref{equ:co2} gives $\sum_{\theta\in F}n_{\theta,2}\geq 2\ell/3 - \delta_2$. As $\abs{S}=2$, write $S=\{\eta,\mu\}$. In this case, $\sigma_\eta$ and $\sigma_\mu$ both have a single  fixed point, denoted by $\theta_1$ and $\theta_2$, respectively. 
    Since $n_{\theta_1,2}\leq \ell/(2m_\eta) = \ell/6$ and $n_{\theta_2,2}\leq \ell/(2m_{\mu})=\ell/2$, 
     the conclusion from \eqref{equ:co2} is that $n_{\theta_1,2} \geq \ell/6-\delta_2=\ell/(2m_\eta) - \delta_2$ and 
     $n_{\theta_2,2} \geq \ell/2-\delta_2= \ell/(2m_{\mu}) - \delta_2$, as desired. 
    
    In case (F2),  $m=8$, $m_\eta=2$, $e_\eta =2$, and $S=\{\mu\}$ for some $\mu\in J$. 
    Moreover, the case $\mu=\eta$ is covered in Step IV, as then $m_\mu=m_\eta$ is even. 
    Hence, we may assume $\mu\neq\eta$, in which case $\abs{O_\eta}=2$. 
    Then  \eqref{equ:co2} reduces to $\sum_{\theta\in F}n_{\theta,2}\geq \ell - \delta_2$ and by definition $n_{\theta,2}\leq \ell/2$ for $\theta\in F$. Thus in this case  as well $n_{\theta,2}\geq \ell/2-\delta_2$ for $\theta\in F$, as desired. 
    \end{proof}

    \begin{proof}[Proof of Proposition \ref{prop:main-transposition}]
    Let $x_1,\ldots, x_r\in G$ be a product-$1$ tuple corresponding to~$f$, and $P_j$ the branch point corresponding to $x_j$. Put $J=\{1,\ldots,r\}$ and  
    let $R_\pi$ be the sum of contributions $R_\pi(f^{-1}(P))$ over all points $P$ of $\mP^1$. 
    We note that since $h_i:Y_i\ra Y$, $i\in I$ admit irreducible fiber products $Y_i\#_Y Y_j$, $i\neq j$ of genus  $< \alpha\ell$, and since the number of branch points of $\pi_0$ is at most four as $g_Y\leq 1$, 
    it follows by Corollary \ref{cor:DZ} that the number of branch points $r$ is bounded by a constant $B=B_{\alpha,t}>0$, depending only on $\alpha$ and $t$. 
    
    Let $S$ be the set of all $j\in J$ such that the action of $x_j$ on $I$ consists of odd length orbits and a single orbit of length $2$. 
    Proposition \ref{prop:genus-analysis}.(2) implies that $R_\pi(f^{-1}(P_j))<E_0\ell^{t-2}$ for all $j\in J\setminus S$, where  $E_0=E_{0,t}>0$ is a constant depending only on $t$. 
    In particular, if $S=\emptyset$, we have $R_\pi<E_3\ell^{t-2}$
    for a constant $E_3:=BE_0$, which depends only on $\alpha$ and $t$. 
    
    Henceforth, assume $S\neq\emptyset$. 
    Recall that the natural projection $\oline\pi_0:\oline Y\ra \mP^1$  from the quotient $\oline Y$ by a point stabilizer in the action of $G$ on $I$, has Galois closure $\pi_0$, and monodromy group $G/K$ equipped with its action on $I$. 
    In particular, the ramification multiset $E_{\oline\pi_0}(P_j)$ coincides with the cycle structure of the image $\sigma_j$ in $G/K$ of $x_j$, for $j\in J$.  Since $S\neq \emptyset$ and $g_Y\leq 1$ by Corollary \ref{cor:DZ},
    the ramification type of $\oline\pi_0$ appears in Table~\ref{table:trans} by Lemma~\ref{lem:no-trans}, and hence $S$ consists of those $j\in J$ for which $x_j$ acts on $I$ as a transposition.   
    Hence, there exists a product-$1$ tuple $x_1',\ldots,x_r'$ which satisfies condition (LG$_Y$) and whose ramification type is the same as that of $f$, by Lemma~\ref{rem:LGY-holds}. 
    We can therefore apply Proposition  \ref{lem:product} and deduce that there exists a reduced tuple $y_{\theta,j}$ where $\theta$ runs through the set $O_j$ of orbits of $x_j'$ on $I$, for $j\in J$. 
    
    Let $e_j$ be the order of the image of $x_j'$ in $G/K$, and $z_j:=(x_j')^{e_j}\in S_\Delta^I$, 
    for $j\in J$.
    By Proposition \ref{cor:gZ-bound}, $y_{\theta,j}$ is of almost Galois type $m(y_{\theta,j}) = 2m(z_j(\iota_\theta))<\infty$ with error at most $\hat\eps_{\alpha,t}$ depending only on $\alpha$ and $t$, if $j\in S$ and $\theta\in O_j$ has length $1$. 
    Since $t\geq 3$ and $x_j'$ acts as a transposition on $I$, each $x_j'$ has an orbit $\tilde \theta\in O_j$ of length $1$ for every $j\in S$, 
    so that $m(y_{\theta,j})$ is even.
    Hence, 
    $R_{\pi}(f^{-1}(P_j)) < E_1\ell^{t-2}$ for all $j\in S$, where $E_1$ is the constant provided by Corollary  \ref{cor:rpi-Galois}. In total we have $R_\pi < E_2\ell^{t-2}$ for $E_2\geq B\cdot \max\{E_3,E_1\}$.  
    \end{proof} 
    This completes the proof of Theorem \ref{thm:main-Gal} and hence also that of Theorem \ref{thm:main-wreath}.

    \section{Realizing the ramification types in Table~ \ref{table:wreath}}\label{sec:h1-branches}
    In this section, we prove the existence of indecomposable coverings $f:X\ra \mP^1$  with the ramification types in Table~\ref{table:wreath}, and a monodromy group of product type. For this, by Riemann's existence theorem it suffices to show:
    \begin{prop}\label{prop:wreath-types} 
    For every ramification type in Table~\ref{table:wreath} with $\ell\geq 9$ 
    there exists a product-$1$ tuple corresponding to the ramification type and generating a primitive group  $G\leq S_\ell \wr S_2$ that contains $A_\ell^2$. 
    \end{prop}
    The proof relies on the following lemma:
    \begin{lem}\label{lem:ram-possible}
    For each $\ell\geq 9$ in the appropriate congruence class modulo $3$ and integer $0<a<\ell/2$ with $(a,\ell)=1$, there exist $x_1,x_2,x_3 \in S_\ell$ such that $x_1x_2x_3=1$, and  $\langle x_1,x_2,x_3\rangle$ contains $A_\ell$,
    and the cycle structures of $x_1,x_2,x_3$ are one of the following:
    \begin{equation}\label{table:I1A}
    \begin{array}{| l | l |}
    \hline
    I1A.1 &  [\ell],  [\ell],  [3,1^{\ell-3}] \\
    I1A.2 &  [\ell],  [\ell],  [2,1^{\ell-2}],  [2,1^{\ell-2}] \\
    I1A.3 &  [\ell],  [\ell],  [2^2,1^{\ell-4}]\\
    I1A.4 &  [a,\ell-a],  [a,\ell-a],  [3,1^{\ell-3}] \\
    I1A.5 & [a,\ell-a],  [a,\ell-a],  [2,1^{\ell-2}],  [2,1^{\ell-2}] \\
    I1A.6 & [a,\ell-a],  [a,\ell-a],  [2^2,1^{\ell-4}]  \\
    \hline
    %
    %
    F2.1 &  [3^{\ell/3}],  [1,2,3^{(\ell-3)/3}]\text{ twice }\\
    F2.2 & [1^2,3^{(\ell-2)/3}],  [2,3^{(\ell-2)/3}]\text{ twice }\\
    F2.3 & [2^2,3^{(\ell-4)/3}],  [1,3^{(\ell-1)/3}]\text{ twice }\\
    \hline
    \end{array}
    \end{equation}
    \end{lem}

    \begin{proof}[Proof of Proposition \ref{prop:wreath-types}]
    In each case in Table~\ref{table:wreath}, we give a product-$1$ tuple corresponding to this ramification types. Let $G\leq S_\ell \wr S_2$ be the group generated by this tuples,  $K:=G\cap S_\ell^2$ and $\pi_i:K\ra S_\ell, i=1,2$ be the natural projections to the first and second coordinate, respectively. 
    
    Consider elements $a,b,c,d,f \in S_\ell$ satisfying $dcabf=1$ with $\langle a,b,c,d\rangle = A_\ell$ or $S_\ell$ and such that $a,d$ are nonconjugate in $S_\ell$. Then both tuples
    \begin{equation}\label{equ:main-branch}
    \begin{array}{ll}
    (a,d), (b,c), (f,1), s, (abf,(abf)^{-1})s; \text{ and } \\
    (a,d), (b,c), (f,1)s, (abf,(abf)^{-1})s, 
    \end{array}
    \end{equation}
    have product $1$. Furthermore, we note that  $G$  is primitive: Indeed, $G$ contains $(a,d), (b,c)$ and the conjugates of these elements by $(f,1)s$, which are $(d,f^{-1}af)$ and $(c,f^{-1}bf)$, respectively.  
    Then the image $\pi_1(K)$  contains $J:=\langle a,b,c,d\rangle$ and hence contains $A_\ell$. Also $G$ contains the square of $(f,1)s$ which is $(f,f)$, so that $\pi_2(K)$ contains $f,d,c,f^{-1}af,f^{-1}bf$
    and hence contains $J$, and also $A_\ell$.  
    Thus, we can apply Lemma \ref{lem:cyc-exist}, and as $a,d$ are nonconjugate deduce that $K\supseteq A_\ell^2$ and $G$ is primitive.
    
    Recall that an element of the form $(\alpha,\alpha^{-1})s$, $\alpha \in S_\ell^2$, is conjugate in $S_\ell \wr~S_2$ to $s$. Thus each of the ramification types in Table~\ref{table:wreath} corresponds to one of the product~$1$ tuples in \eqref{equ:main-branch}, with the exception of cases 
    \begin{equation}\label{equ:inf}
     \text{I1.1, I1A.2a, I1A.5a, I2.1b, I2.2a, I2.9a, I2.10b,} 
     \end{equation}
    where none of the branch cycles are conjugate to $s$, and all F4 cases in which there are four elements in $G\setminus K$. Moreover, for each ramification type which is not in \eqref{equ:inf}, there exists $a,b,c,d,f$ for which \eqref{equ:main-branch} corresponds to this ramification type, and such that $dcabf=1$ and $\langle a,b,c,d\rangle$ is  $A_\ell$ or $S_\ell$ by \cite[Section 11]{NZ} which relies on \cite[Chapter 3]{GS} and by Lemma \ref{lem:ram-possible} 
    for the I1A and F2 types, as desired.

    We next consider the exceptions in \eqref{equ:inf}. Consider the tuple
    \begin{equation}\label{equ:branch3} (b,d), (1,u)s, (bv,b^{-1})s, \end{equation}
    where $b,d,u,v \in S_\ell$ are elements which satisfy $dubv =1$ and $\langle b, d,u,v\rangle  = A_\ell$ or~$S_\ell$. 
    The tuple \eqref{equ:branch3} then has product $1$. 
    We claim that $\pi_i(K)$ contains $A_\ell$, for $i=1,2$. 
    Since $K$ contains $(b,d)$ and the squares  $((1,u)s)^2=(u,u)$ and $((bv,b^{-1})s)^2=(bvb^{-1},v)$, the projection $\pi_2(K)$ contains $d,u,v$, and hence also $b=u^{-1}d^{-1}v^{-1}$, so that  $\pi_2(K) \supseteq A_\ell$. Conjugating $\pi_2^{-1}(A_\ell)$ by $(1,u)s$ we get that $\pi_1(K)$ also contains $A_\ell$, proving the claim. 
    Hence by  Lemma \ref{lem:cyc-exist}, either $G\supseteq A_\ell^2$ or $K\subseteq \{(u,u^w)\,|\,u\in S_\ell\}$ for some $w\in S_\ell$. %
    
    For each of the ramification types $\mathcal R$ in \eqref{equ:inf}, 
     there exist $d,u,b, v\in S_\ell$  such that $dubv=1$, $\langle b,d,u,v\rangle  = A_\ell$ or $S_\ell$, and for which the tuple \eqref{equ:branch3} is conjugate to $\mathcal R$, by \cite[Section 11]{NZ} and by Lemma \ref{lem:ram-possible} for the I1A types. 
     In cases I1.1, I2.1b, I2.9a, the elements $a,d$ can be chosen to be nonconjugate, and hence $K\supseteq A_\ell^2$ and $G$ is primitive by Lemma \ref{lem:cyc-exist}. 
    
    Other than the F4, types, it remains to choose $b,d,u,v$ 
    in cases I2.2a, I2.10b, I1A.2a, I1A.5a such that $K\supseteq A_\ell^2$ or equivalently that $K$ is not contained $D_w:=\{(u,u^w)\,|\,u\in S_\ell\}$ for any $w\in S_\ell$. 
    In case I2.2a, we may choose 
    \begin{align*}
    b & = (1,2)(3,4)\cdots(\ell-3,\ell-2), \qquad v=(\ell-1,\ell) \\
    d & = (2,3)(4,5)\cdots(\ell-2,\ell-1), \qquad u = (2, 4, \ldots, \ell, \ell-1, \ell-3,\ldots, 1).
    \end{align*}
    As $K$ contains $(u,u) = ((1,u)s)^2$ and $(b,d)$, 
     an inclusion $K\subseteq D_w$ implies that $u^w = u$ and $b^w=d$. 
     As the centralizer of the $\ell$-cycle $u$ is $\langle u\rangle$, this implies $w\in \langle u\rangle$, but a straightforward check shows that $b^{u^i}\neq d$ for all $i$, contradiction.
     Hence $K\supseteq A_\ell^2$ and $G$ is primitive by Lemma \ref{lem:cyc-exist}. 
    
    The same argument also applies in case I2.10b by choosing:
    \begin{align*}
    b & = (1,2)(3,4)\cdots(\ell-2,\ell-1),\quad  v=(2k,1),\quad d = (2,3)(4,5)\cdots(\ell-1,\ell) \\
    u & = (2k-1,2k-3,\ldots,1)(2,4,\ldots, \ell-1,\ell, \ell-2, \ldots, 2k+1),
    \end{align*}
    for every $1\leq k<\ell-1$ prime to $\ell$. 
    As in the previous case, an inclusion $K\subseteq D_w$ implies $u^w =u$ and $b^w=d$. As $(k,\ell)=1$, the centralizer of $u$  is $$C_{S_\ell}(u)=\langle (1,3,\ldots, 2k-1), (2k+1,2k+3,\ldots, \ell, \ell-1,\ell-3,\ldots, 2)\rangle.$$ Since $b^w=d$, we have $w(\ell)=1$ but there is no such element in $C_{S_\ell}(u)$. 
    

     In  Case I1A.2a (resp.~I1A.5a),  we can take $b = d^{-1}$ to be  $(1,\ldots,\ell)$ (resp.~to be $(1,\ldots, k)(k+1,\ldots, \ell))$, and $u^{b} = v = (2, 3)$ (resp.~$u^{b} = v = (1, k+1)$). 
    Conjugating $(b,b^{-1})$ by $(1,u)s$ we get $((b^{-1})^u, b)\in G$, and hence $(u^{-1}b^{-1}ub,1)\in G$. Since $u^{-1}b^{-1}ub$ is not conjugate to $1$ in both cases, Lemma~\ref{lem:cyc-exist} gives $K\supseteq A_\ell^2$ and hence $G$ is primitive.

    For types F4.1-F4.5, consider the  tuple:
    \begin{equation}\label{equ:g1-branch} s,(a^{-1},a)s, (b,ub^{-1})s,(c^{-1}v,c)s, (d_1,d_2), (e_1,e_2), \end{equation}
    where all of $u, v, d_i,e_i, i=1,2$, are trivial except (A) one $3$-cycle in case F4.5,  (B) one element which is a product of two transpositions in case F4.4, and (C) two elements that are transpositions in cases F4.1-F4.3. 
    Note that the tuples in \eqref{equ:g1-branch} cover all cases F4.1-F4.5 as they are conjugate in $S_\ell \wr S_2$ to: 
    \begin{equation}\label{equ:g1-conjugate} s,s,(u,1)s, (v,1)s, (d_1,d_2), (e_1,e_2). \end{equation}
    
    We first consider cases F4.1-F4.5. The product-$1$ relation for the elements in \eqref{equ:g1-branch} is equivalent to the equations $abcd_1e_1 = 1$ and $a^{-1}ub^{-1}c^{-1}vd_2e_2=~1$. Setting $a^{-1} = bcd_1e_1$, we get that the product~$1$ relation is equivalent to: 
    \[
    [c^{-1},b^{-1}] =  (d_1e_1u)^{b^{-1}c^{-1}}vd_2e_2, 
    \]
    where $[c^{-1},b^{-1}] = cbc^{-1}b^{-1}$. Setting $b = (1,\ldots, \ell)$ and $c=(1,k)$ with $k>2$ and $(k-1,\ell)=1$ (resp.~$c = (1,2)$, $k=2$) in cases (B) and (C) (resp.~in case (A)), we get that $[c^{-1},b^{-1}]$ is a product of two transpositions (resp.~a $3$-cycle). 
    For types F4.4, F4.5,  note that $b,c$ can also be chosen even so that $G\subseteq A_\ell\wr S_2$, namely, set $c= (1,3)(2,4)$ (resp.\ $c=(1,4)(2,5)$) and $b=(1,\ldots,\ell)$ for odd $\ell$ in  (A) (resp.\  (B)), and set $c=(1,2,3)$ (resp.\ $c=(1,3)(2,4))$ and $b:=(1,2)(3,4,\ldots,\ell)$ in (A) (resp.\ (B)) for even $\ell$. 
    For each of the cases F4.1-F4.5, we then choose from $u,v,d_i,e_i,i=1,2$ the corresponding nontrivial element in cases (B) and (C) (resp.~two elements in case (A)) so that the tuple \eqref{equ:g1-branch} has product $1$. 
    
    We next claim that the group $G$ generated by the tuple \eqref{equ:g1-branch} is primitive in cases F4.1-F4.5. In these cases, the projections $\pi_i(K)$, $i=1,2$ are all of $S_\ell$ (resp.\ $A_\ell$ for F4.4,F4.5 with even $b,c$) since these contain the $\ell$-cycle $b$ (resp.\ $(1,2)(3,4,\ldots,\ell)$ for even $\ell$) and the transposition $(1,k)$ as $(k-1,\ell)=1$ (resp.\ and $c=(1,2,3)$ or $(1,3)(2,4)$ or $(1,4)(2,4)$). 
    Note that in cases F4.2, F4.3, F4.4 and F4.5, the tuple \eqref{equ:g1-branch} contains an element of the form $(x,1)$ with $x\neq 1$, arising either from $(d_1,d_2)$ or from $(e_1,e_2)$, and thus  $K$ is not contained in $D_w$ for any $w\in S_\ell$, implying that $G$ is primitive by Lemma \ref{lem:cyc-exist}. In case F4.1 we can assume $v$ is a transposition, and hence $(c^{-1}v,c)\in K$, but $c$ and $c^{-1}v$ are nonconjugate in $S_\ell$ (as they have a different sign), hence $K\not\not\subseteq D_w$ for all $w\in S_\ell$, and $G$ is primitive by Lemma \ref{lem:cyc-exist}. 
    \end{proof}
    \begin{proof}[Proof of Lemma \ref{lem:ram-possible}]
    In each case 
    let $G$ be the group generated by the elements in the corresponding product-$1$ tuple.
    
    \noindent {\bf Cases I1A.1,I1A.2,I1A.3:} Write the product-$1$ relation as $x_2x_3=x_1^{-1}$ where
    \begin{align*}
        x_2 & := (1,u+1) (2,u+2), \qquad x_3:= (1,2,\ldots,\ell), \\
    x_1^{-1} & := (1,u+2,3,4,\ldots ,u+1,2,u+3,u+4,\ldots,\ell),
    \end{align*}
    where $1\leq u\leq \ell-2$. 
    For $u>1$, this covers the product-$1$ relation in case I1A.3 and also in the case IA.2 by splitting $x_2$ into two elements $(1,u+1)$ and $(2,u+2)$. Note that if $u=1$ then $(1,u+1)(2,u+2)=(3,2,1)$ is a 3-cycle, covering case I1A.1.
    
    Now suppose that $(u,\ell)=1$ (which is of course true for $u=1$).
    Let $G = \langle x_1^{-1},x_2\rangle$. 
    Consider a partition of $\{1,2,\ldots,\ell\}$ preserved by $G$, and suppose that some part $B$ has
    size at least $2$.  Since the partition is preserved by $\langle x_3\rangle$, it must consist
    of congruence classes mod some proper divisor of~$\ell$.  
    Suppose that $1\in B$ and $B$ is not $\{1,2,\ldots,\ell\}$.  
    Then $2\not\in B$ and $u+1\not\in B$, as $(u,\ell)=1$.
    Hence $B$ is not fixed by $(1,u+1)(2,u+2)$, so it contains no fixed points of this
    permutation, whence $B=\{1,u+2\}$, so that $u=\ell/2-1$; but then its image $\{2,u+1\}=\{2,\ell/2\}$
    is also a block, contradicting the fact that  the block containing $\ell/2$ is $\{\ell/2,\ell\}$.
    Therefore the partition is trivial and $G$ is primitive.  Since $G$ contains an
    element with cycle structure $[2^2,1^{\ell-4}]$ or $[3,1^{\ell-3}]$,  it follows that $G$ contains $A_\ell$, 
    by Remark \ref{rem:jordan}.
    
    
    \noindent {\bf Case I1A.4:} Write the product-$1$ relation as $x_2x_3=x_1^{-1}$ where
    \begin{align*}
    x_2 & :=(1,a+1,\ell+1-a), \qquad x_3:= (1,2,\ldots,a)(a+1,a+2,\ldots,\ell), \\ 
    x_1^{-1} & :=(1,a+2,a+3,\ldots,\ell+1-a,2,3,\ldots,a)(a+1,\ell+2-a,\ell+3-a,\ldots,\ell).
    \end{align*}
    Any block $B$ of $G$ of size bigger than $1$ which contains $1$ must also contain a fixed point of $x_2$ unless  $B=\{1,a+1,\ell+1-a\}$, so in any case $B$ contains $\{1,a+1,\ell+1-a\}$.
    But the $a$-th power of $x_3$ is an $(\ell-a)$-cycle which fixes $1$ and hence $B$. 
    As $B$ contains $a+1$, it also contains the entire $(\ell-a)$-cycle.
    Likewise the $(\ell-a)$-th power of $x_3$ is an $a$-cycle which contains $1$ and fixes $a+1$, so it fixes
    $B$ and hence contains the entire $a$-cycle.  So  $B=\{1,2,\ldots,\ell\}$, whence $G$ is primitive.
    Since $G$ contains a $3$-cycle, it contains $A_\ell$ by Remark \ref{rem:jordan}.
      
    \noindent {\bf Case I1A.5:}  Use the identity $x x^{-1}  y y^{-1}=1$, where $$x:=(1,2,\ldots,a)(a+1,a+2,\ldots,\ell),\text{ and }y:=(1,a+1),$$
    for $(a,\ell)=1$. As in case I1A.4, $\langle x,y\rangle$ is primitive and hence is $S_\ell$ by Remark \ref{rem:jordan}. 
    
    \noindent {\bf Case I1A.6:} Write the product-$1$ relation as
    $x_1x_2=x_3$, with 
    \begin{align*}
        x_1 & :=(a+1,a,a-1,\ldots,2) (\ell,\ell-1,...,a+2,1),  \\
        x_2 & :=(1,2,\ldots,a) (a+1,\ldots,\ell),\qquad x_3 := (1,a+1)(2,a+2) 
    \end{align*} if $1<a<\ell$, 
    and 
    \begin{align*}
        x_1 & :=(\ell-1,...,3,\ell,2), \quad x_2:=(1,2,...,\ell-1), \quad x_3 := (1,2)(3,\ell) 
    \end{align*}
     if $a=1$ or $\ell$.
    If $1<a<\ell$, switching the roles of $x_2,x_3$, the same argument as in case I1A.4 applies unless the nontrivial block $B$ containing $1$  contains no element from $a+1,\ldots,\ell$ and no fixed point of $x_3$, so that  $B=\{1,2\}$. In this case $B$ is not fixed by $x_2$, contradicting $1^{x_2}=2$. If $a=1$ or $\ell$, and $C$ is a block containing $\ell$, then $C$ is fixed by $x_2$, and hence either $C=\{\ell\}$  or $C$ also contains the orbit $1,2,\ldots,\ell-1$ of $x_2$.  Thus $G$ is primitive. As in addition $G$ contains $x_3$, it contains $A_\ell$ by Remark \ref{rem:jordan}. 
    
    \noindent {\bf Case F3.2:} 
    For $\ell=3u+1$, 
    write the product-$1$ relation as $x_1x_2=x_3^{-1}$ where
    \begin{align*}
    x_1& :=(3u)(1,2,3u+1) \prod_{i=1}^{u-1} (3i,3i+1,3i+2), 
     \\
    x_2& :=(3u+1) \prod_{i=0}^{u-1} (3i+1,3i+2,3i+3), \\
    x_3^{-1}& :=(2,3u+1)(3u-2,3u) \prod_{i=1}^{u-1} (3i-2,3i,3i+2).
    \end{align*}
    
    Consider a block $B$ which contains $2$ and has size at least $2$.
    If  $3u+1\not\in B$ then $B$ is not fixed by
    $(x_3^{-1})^3=(2,3u+1)(3u-2,3u)$, so $B$ is either $\{2,3u-2\}$ or $\{2,3u\}$.  
    But $3u\not\in B$,
    since $3u$ is fixed by $x_1$, but $2^{x_1}= 3u+1$.
    So $B=\{2,3u-2\}$.  Then $B^{x_3^{-1}} = \{3u+1,3u\}$, while  $B^{x_1}=\{3u+1,3u-1\}$, contradiction.  So $B$ has to contain $3u+1$,
    and thus is fixed by $x_1$ and $x_3^{-1}$, whence the block
    contains the orbit of $2$ under $G$, so $B=\{1,2,\ldots, \ell\}$.  Hence $G$ is primitive and contains an element of cycle structure $[2^2,1^{\ell-4}]$,
    so 
    $G=A_\ell$ by Remark \ref{rem:jordan}.
    
    %
    %
    \noindent {\bf Case F2.2:} For $\ell=3u+2$, 
    write the product-$1$ relation as $x_1x_2=x_3^{-1}$ where
    \begin{align*}
    x_1& :=(2,3u+2)(1,3,4) \prod_{i=2}^u (3i-1,3i,3i+1), \quad
    x_2 :=(1)(2) \prod_{i=1}^u (3i,3i+1,3i+2), \\
    x_3& :=(1,4)(2,3u,3u+2) \prod_{i=1}^{u-1} (3i,3i+2,3i+4). 
    \end{align*}
    Consider a block $B$ which contains $2$ and has size at least $2$.
    Since $x_1^3=(2,3u+2)$, the block
    contains $3u+2$ and hence is fixed by $x_3^{-1}$.  
    Thus $B$ contains the orbit of $3u+2$ under $G=\langle x_1,x_3\rangle$, which is
    $\{1,2,\ldots,\ell\}$.  So $G$ is primitive and contains a $2$-cycle,
    whence $G=S_\ell$ by Remark \ref{rem:jordan}.
    
    %
    %
    \noindent {\bf Case F2.1:} For $\ell=3u$, 
    write the product-$1$ relation as $x_1x_2=x_3^{-1}$ where
    \begin{align*}
    x_1& :=(1,2,3u)(3u-2,3u-1) \prod_{i=1}^{u-2} (3i,3i+1,3i+2), \quad
    x_2 :=\prod_{i=0}^{u-1} (3i+1,3i+2,3i+3),\\
    x_3^{-1}& :=(3u-1)(3u-3,3u-5)(2,3u-2,3u) \prod_{i=1}^{u-2} (3i-2,3i,3i+2).
    \end{align*}
    The same argument as in case F2.2 applied to the block $B$ of $3u-1$ gives $G=S_\ell$. 
    \end{proof} 
    
    Up to conjugacy, there are four primitive subgroups $A_\ell^2\leq G\leq S_\ell\wr S_2$. These are (1) $A_\ell \wr S_2$, (2) $S_\ell\wr S_2$, (3) $(S_\ell\#_{C_2} S_\ell)\rtimes S_2$ where $S_\ell\#_{C_2} S_\ell:=\{(\sigma,\tau)\in S_\ell\times S_\ell\suchthat \sgn(\sigma)=\sgn(\tau)\}$ and $\sgn:S_\ell\ra C_2$ is the sign map, and (4)  $A_\ell^2\rtimes C_4$ where $C_4=\langle (a,1)s \rangle$ and $a\in S_\ell$ is a transposition. The following lemma determines the group for each of the ramification types in Table \ref{table:wreath}:
    \begin{lem}\label{lem:groups}
    Let $f$ be an indecomposable covering with monodromy group $A_{\ell^2}\leq G\leq S_\ell\wr S_2$ and ramification type in Table \ref{table:wreath}. \\ 
    Then $G=A_\ell\wr C_2$ in the following cases:  F2.3, F4.4, F4.5; when $\ell$ is odd in cases I1A.1, I1A.3;  when $\ell$ is even in cases I1A.4, I1A.6; when $\ell\equiv 1$ mod $4$ in cases I2.5, I2.8, F1A.5, F1A.8; when $\ell\equiv 3$ mod $4$ in cases I2.3, I2.6; when $\ell\equiv 0$ mod $4$ in cases I2.11, I2.15; when $\ell\equiv 2$ mod $4$ in cases I2.13, F1A.9; and when $\ell\equiv 1$ mod $8$ in case F3.2.

For the following types, under the complementary congruence conditions,  $G=(S_\ell\#_{C_2}S_\ell)\rtimes S_2$:  F1A.3a, F1A.4a, F1A.6a, F1A.7a, 
F4.4, F4.5; when $\ell$ is even in cases I1A.1, I1A.3;  when $\ell$ is odd in cases I1A.4, I1A.6; when $\ell\equiv 3$ mod $4$ in cases I2.5, I2.8, F1A.5, F1A.8; when $\ell\equiv 1$ mod $4$ in cases I2.3, I2.6; when $\ell\equiv 2$ mod $4$ in cases I2.11, I2.15; when $\ell\equiv 0$ mod $4$ in cases I2.13, F1A.9; and when $\ell\equiv 5$ mod $8$ in case F3.2.

In cases I1A.2a, I1A.5a, I2.2a, and I2.10b,  one has $G=A_\ell^2 \rtimes C_4$.

Finally,  $G=S_\ell \wr S_2$ in the remaining cases: I1.1, I1A.2b, I1A.2c, I1A.5b, I1A.5c, I1A.7a,I1A.7b, I2.1a, I2.1b, I2.2b, I2.4, I2.7, I2.9a, I2.9b, I2.10a, I2.12, I2.14, F2.1, F2.2, F3.1, F3.3, F1A.1a, F1A.1b, F1A.2a, F1A.2b, F1A.3b, F1A.4b, F1A.6b, F1A.7b, F4.1, F4.2, F4.3. 
\end{lem}
\begin{proof}
To determine $G$  one applies the following procedure to each ramification type $R$ in Table \ref{table:wreath}. As always set $K=G\cap S_\ell^2$. Let $x_1,\ldots,x_r$ be a product-$1$ tuple corresponding to $R$. \\ 
(1) If some $x_i$ is conjugate to $(a,b)$ where exactly one of $a$ and $b$ is an odd permutation, then  $G=S_\ell\wr S_2$. Indeed, since  the projections of $K$ to each copy of $S_\ell$ are isomorphic  as in Remark \ref{rem:setup}.(2), and since $x_i\not\in A_\ell^2$, both of the projections are $S_\ell$. However, as $K$ is not contained in $S_\ell\#_{C_2}S_\ell$, this forces $K=S_\ell^2$ and hence $G=S_\ell\wr S_2$.

Assume henceforth that no $x_i$ satisfies (1), and let $\mathcal S$ be the subset of $x_i$'s conjugate to some $(a,1)s$ for  $a\in S_\ell\setminus A_\ell$. Let $\pi:S_\ell\wr S_2\ra D_4$ be the composition of the projection to $S_\ell\wr S_2/A_\ell^2\cong S_2\wr S_2$ with an isomorphism to the Dihedral group $D_4$ of order $8$. 

\noindent (2)   If $\mathcal S\neq\emptyset$, then $G=A_\ell^2\rtimes C_4$ or $G=S_\ell\wr S_2$. 
Indeed, an element of $\mathcal S$ projects in $G/A_\ell^2$ to an element of order $4$, but the groups $A_\ell\wr S_2/A_\ell^2\cong S_2$ and $(S_\ell\#_{C_2}S_\ell)\rtimes S_2/A_\ell^2\cong C_2\times C_2$ contain no such element. \\
\noindent (2a) If every $x_i\notin K$ is in $\mathcal S$, then  $G=A_\ell^2\rtimes C_4$. Indeed for such $x_i$,  $\pi(x_i)$  is an element of order $4$ which lies in the unique cyclic subgroup $C_4$ of order $4$ in $D_4$. As $\pi^{-1}(C_4)=A_\ell^2\rtimes C_4$, all such $x_i$'s are contained in $A_\ell^2\rtimes C_4$. Finally, the last group also contains all $x_i$'s which are in $K$, since by assumption these are in $S_\ell\#_{C_2} S_\ell=(A_\ell^2\rtimes C_4)\cap S_\ell^2$.  As all elements of $R$ are contained in $A_\ell^2\rtimes C_4$, so is $G$. As noted in (2), this forces $G=A_\ell^2\rtimes C_4$.  

\noindent (2b) If $\mathcal S\neq\emptyset$ and some $x_i$ is conjugate to $(a,1)s$ where $a$ is an even permutation, then $G=S_\ell \wr S_2$. Indeed, an element from $\mathcal S$ generates the unique order $4$ cyclic subgroup of $S_\ell\wr S_2/A_\ell^2\cong S_2\wr S_2$ and its square is in $K$. As $x_i$ is conjugate  to $(a,1)s$, its image   in $S_2\wr S_2/A_\ell^2$ is of order $2$ and is not in $K$, so that together with an element of $\mathcal S$ it generates all of $S_2\wr S_2$, forcing $G=S_\ell\wr S_2$. 

\noindent (3) Assume henceforth $\mathcal S=\emptyset$. 
Then all $x_i$'s are in $(S_\ell\#_{C_2} S_\ell)\rtimes S_2\lhd S_\ell\wr S_2$ and hence $G=A_\ell\wr S_2$ or $(S_\ell\#_{C_2} S_\ell)\rtimes S_2$. \\
(3a) If some $x_i$ is conjugate to $(a,b)$ for odd permutations $a,b$, then $G=(S_\ell\#_{C_2} S_\ell)\rtimes S_2$ since  $x_i\notin A_\ell\wr S_2$. \\
(3b) If every $x_i$ is conjugate either to an element in $A_\ell^2$ or to $s$, and there are exactly two $x_i$'s not in $K$, then $G=A_\ell \wr S_2$. Indeed, by reordering the product-$1$ tuple using the swapping steps $(x_i,x_{i+1})\mapsto (x_{i+1},x_i^{x_{i+1}})$ and by simultaneously conjugating the tuple, we may assume $x_1=(u,1)s$ for some  $u\in A_\ell$; $x_2=(v,w)s$,  $vw  \in  A_\ell$; and $x_3,\ldots, x_r\in A_\ell^2$. 
The product-$1$ condition
 then says that each of $v$ and $w$ is a product of elements of $A_\ell$, and hence is in $A_\ell$. As all $x_i$'s are in $A_\ell\wr S_2$, we have $G=A_\ell\wr S_2$. \\
\noindent (3c) The remaining types F4.4 and F4.5 appear with both monodromy groups $A_\ell\wr S_2$ and $(S_\ell\#_{C_2} S_\ell)\rtimes C_2$. Indeed,  product-$1$ tuples for these ramification types with  both groups are given in the proof of  Proposition \ref{prop:wreath-types} by specifying the parameters in \eqref{equ:g1-branch}. 
\end{proof}

    \appendix
    \section{Non-occurring ramification types} \label{sec:no-ram}
    We give an alternative explicit proof for types I1A.N1-N2, F4.N1-N2 in Lemma \ref{cor:inf2-3} using: 
    %
    \begin{lem}\label{lem:cyc-non-exist} 
    Let $a,b\in S_\ell$ be $\ell$-cycles and $c,d\in S_\ell$ be $2$-cycles. 
    \begin{enumerate}
    \item If $adbc=1,$
    then there exists $z \in S_\ell$ such that $z^2=1$,  $zaz=b$ and $zcz=d.$
    \item If $c,d$ are disjoint $2$-cycles, and $abcd=1$,
    then there exists $v \in S_\ell$ such that $v^2=cd$ and $vbv^{-1}=a$ and $vav^{-1}=cdbcd$.
    \end{enumerate}
    \end{lem}
    \begin{proof}
    First suppose that $adbc=1$ or $ad = cb^{-1}$. Without loss of generality, assume $b= (1,\ldots \ell)$ and $c = (1,i)$. 
    Then $cb^{-1}$ is a product of the two disjoint cycles $(i-1,\ldots, 1)$ and $(\ell,\ldots,i)$. Similarly, writing $a = (u_1,\ldots, u_\ell)$ and $d= (u_1,u_j)$ we get $ad = (u_1,\ldots u_{j-1})(u_j,\ldots, u_\ell)$. 
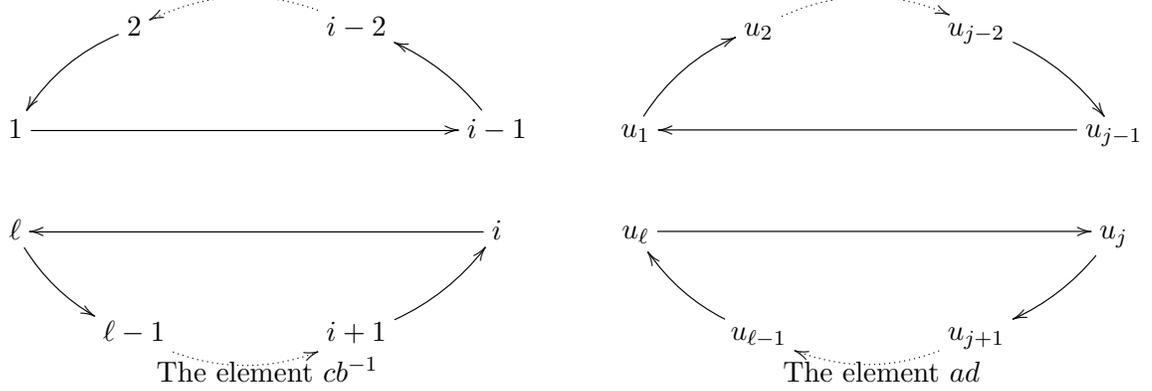
    $$
    \begin{array}{ccc}
    \xymatrix{
    & 2 \ar@/_0.5pc/[dl] & & i-2 \ar@{..>}@/_1pc/[ll]  & \\
    1 \ar[rrrr] & & & & i-1 \ar@/_0.5pc/[ul] \\
    \ell  \ar@/_0.5pc/[dr] & & & & i \ar[llll]  \\
    & \ell -1  \ar@{..>}@/_1pc/[rr] & & i+1 \ar@/_0.5pc/[ur] &  \\
    } &
    \,\,\,\,\,
    & 
    \xymatrix{
    & u_2  \ar@{..>}@/^1pc/[rr] & & u_{j-2}  \ar@/^0.5pc/[dr] & \\
    u_1 \ar@/^0.5pc/[ur] & & & & u_{j-1} \ar[llll] \\
    u_\ell \ar[rrrr]  & & & & u_j  \ar@/^0.5pc/[dl] \\
    & u_{\ell -1} \ar@/^0.5pc/[ul]  & & \ar@{..>}@/^1pc/[ll] u_{j+1}  &  \\
    }
    \\
    \text{The element $cb^{-1}$} & & \text{The element $ad$}
    \end{array} 
    $$
    Since $ad = cb^{-1}$ we get either $i-1=j-1$ or $i-1=\ell-j+1$. Without loss of generality assume $i=j$. 
    
    Defining $z\in S_\ell$ by $z(i) = u_i$, we get $(cb^{-1})^z = (ad)^{-1}$ and $c^z =d$. Hence $(bc)^z = ad$ and $b^z = a$.   Since $(bc)^z = ad = (bc)^{-1}$, we deduce that $z^2$ acts trivially by conjugation on $bc$ and on $c$. Since $b$ and $c$ generate $S_\ell$, $z^2$ is in the center of $S_\ell$, so $z^2=1$, completing part~(1). 
    
    Assume $abcd=1$ and rewrite this relation as $a_0dbc=1$ where $a_0:=a^{d^{-1}}=a^d$. Applying part (1) to the latter relation, we obtain an involution $z\in S_\ell$ such that $(a_0)^z = b$ and $c^z=d$. Setting $v=dz$, we have $a^v = (a_0^d)^v = a_0^z = b$, and $v^2 = dz^2d^z = dc = cd$, proving (2). 
    \end{proof}
    For types F4.N1-N2 
    the proof is based on the following theorem from \cite{JP}. Let $[a,b]:=a^{-1}b^{-1}ab$ denote the commutator. 
    \begin{thm}\label{lem:commutator}
    Let $a,b$ be elements of $S_\ell$ such that $[a,b]$ is a product of two transpositions in $A_\ell$. 
    Then there exists  $z\in S_\ell$ such that $a^z=a^{-1}$ and $b^z=b^{-1}$.
    \end{thm}
    \begin{proof}[Proof of Lemma \ref{cor:inf2-3}, types I1A.N1-N2, F4.N1-N2]
    We show there is no product-$1$ tuple $x_1,\ldots,x_r$ corresponding to the ramification types in Table~\ref{table:I1AN} and generating a primitive group. 
    In case I1A.N1 by conjugating the tuple by an element in $S_\ell^2$, we may assume that the product-$1$ tuple is of the form:
    $$(a,b),  (c,d),  (e^{-1},e) s,  s,$$
    where $a$ and $b$ are $\ell$-cycles, and $c$ and $d$ are $2$-cycles. The product-$1$ relation gives $e=ac$ and $bde=bdac=1$.  
    By Lemma \ref{lem:cyc-non-exist}, there exists $z \in S_\ell$ such that
    $z^2=1$, $b=zaz$, $d=zcz$, so that $$ze^{-1}z=zbdz=(zbz)(zdz)=ac=e,$$ $zbz=a$, $zdz=c$, and $zez=e^{-1}$.
    Since $K:=G\cap S_\ell^2$ is generated by $(a,b),(c,d)$ and $(e,e^{-1})$, it is contained in $\{(u,zuz):u \in S_\ell\}$,  hence doesn't
    contain $A_\ell^2$ and does not generate a primitive subgroup $G\leq S_\ell \wr S_2$ by Lemma \ref{lem:cyc-exist}. 
    
    In case I1A.N2, conjugating the tuple with an element in $S_\ell^2$, we may assume that the product-$1$ tuple is of the form:  $$(b,a),  (cd,1) s,  (e,e^{-1})s$$
    where $c,d$ are disjoint $2$-cycles and $a,b$ are $\ell$-cycles. The product-$1$ relation gives $e=bcd$, and $ae=abcd=1$.
    By the previous lemma there exists $v \in S_\ell$ such that $v^2=cd$ and $vbv^{-1}=a$ and $vav^{-1}=cdbcd$;
    since the intersection of the group with $S_\ell^2$ is generated by $(b,a), (cd,cd)$ and $(a,cdbcd)$ (which is conjugate to the product of $(cd,1) s$ and $(e,e^{-1})s$),
    it is contained in $\{(u,vuv^{-1})\,|\,u \in S_\ell\}$, 
     hence doesn't
    contain $A_\ell^2$ and does not generate a primitive subgroup $G\leq S_\ell \wr S_2$ by Lemma \ref{lem:cyc-exist}. 
    
In cases F4.N1, F4.N2,  conjugating by an element in $S_\ell^2$, we may assume that the product-$1$ tuple is of the form:
    $$ s, (a^{-1},a)s, (b,b^{-1})s, (c^{-1}v,c)s, $$
    where $v$ is a $3$-cycle or the product of two disjoint $2$-cycles. 
    The product-$1$ relation amounts to the equalities $abc=1$ and $a^{-1}b^{-1}c^{-1}v=1$. Replacing $c=b^{-1}a^{-1}$, we have $[a,b]=a^{-1}b^{-1}ab=v^{-1}$. Since $(a^{-1},a)=(a^{-1},a)s\cdot s\in K$ and similarly $(b^{-1},b)\in K$, the projections of $K$ to the $i$-th coordinate is $K_i = \langle a, b\rangle$, for $i=1,2$. 
    Since $G$ is primitive, $K_i$ is primitive by Lemma \ref{rem:product-type}.(3). 
    Since $K_i$ is primitive and contains $v$, which is either a $3$-cycle or a product of two disjoint transpositions, 
    Remark \ref{rem:jordan} implies that $K_i\supseteq A_\ell$, for $i=1,2$. 
    Since $K_i\supseteq A_\ell$, $i=1,2$, and $G$ is primitive, 
    we have $G\supseteq A_\ell^2$ by Lemma \ref{lem:primitive}. 
    On the other hand, 
    by Theorem \ref{lem:commutator}, there exists $z\in S_\ell$ such that $a^z=a^{-1}$ and $b^z=b^{-1}$. Since $K$ is generated by $(a^{-1},a)$ and $(b^{-1},b)$, we get that $K\subseteq \{(u,u^z)\suchthat u\in S_\ell\}$, contradicting $K\supseteq A_\ell^2$. 
        \end{proof}

    
    \bibliographystyle{plain}

    \end{document}